\documentclass[12pt]{amsart}
\usepackage{amsmath}
\usepackage{amsxtra}
\usepackage{amstext}
\usepackage{amssymb}
\usepackage{amsthm}
\usepackage{latexsym}
\usepackage{dsfont} 
\usepackage{verbatim}

\theoremstyle{plain}
\newtheorem{thm}{Theorem}[section]
\newtheorem{cor}[thm]{Corollary}
\newtheorem{lem}[thm]{Lemma}
\newtheorem{prop}[thm]{Proposition}

\newtheorem{alt}[thm]{Alternative}

\theoremstyle{definition}
\newtheorem{defn}[thm]{Definition}

\newtheorem*{goal}{MAIN GOAL}
\newtheorem{cla}[thm]{Claim}

\numberwithin{equation}{section}
\newcommand{\ci}[1]{_{{}_{\scriptstyle{#1}}}}

\def\R1{\widetilde{R}}
\def\T1{\widetilde{T}}
\def\dist{\operatorname{dist}}
\def\supp{\operatorname{supp}}
\def\Lip{\operatorname{Lip}}
\def\eps{\varepsilon}
\def\kap{\varkappa}
\def\R{\mathbb{R}}
\def\RSO{\mathcal{R}}
\def\TSO{\mathcal{T}}

\def\osc{\operatorname{osc}}

\def\dy{\mathcal{D}}
\def\dysel{\mathcal{D}_{\operatorname{sel}}}
\def\dyselA{\widehat{\mathcal{D}}_{\operatorname{sel}}}
\def\wh{\widehat}
\def\wt{\widetilde}

\def\Xint#1{\mathchoice
   {\XXint\displaystyle\textstyle{#1}}%
   {\XXint\textstyle\scriptstyle{#1}}%
   {\XXint\scriptstyle\scriptscriptstyle{#1}}%
   {\XXint\scriptscriptstyle\scriptscriptstyle{#1}}%
   \!\int}
\def\XXint#1#2#3{{\setbox0=\hbox{$#1{#2#3}{\int}$}
     \vcenter{\hbox{$#2#3$}}\kern-.5\wd0}}

\def\dashint{\Xint-}

\begin{document}

\title[The Riesz transform and the Wolff Energy]
{The Riesz transform of codimension smaller than one and the Wolff energy}

\author[B. Jaye]{Benjamin Jaye}
\address{Department of Mathematical Sciences, Kent State University, Kent, Ohio 44240, USA}
\email{bjaye@kent.edu}
\author[F. Nazarov]{Fedor Nazarov}
\address{Department of Mathematical Sciences, Kent State University, Kent, Ohio 44240, USA}
\email{nazarov@math.kent.edu}
\author[M. C. Reguera]{Maria Carmen Reguera}
\address{School of Mathematics, University of Birmingham, Birmingham, UK}
\email{m.reguera@bham.ac.uk}
\author[X. Tolsa]{Xavier Tolsa}
\address{ICREA and Departament de Matematiques, Universitat Autonoma de Barcelona, 08193, Bellaterra, Barcelona, Catalonia}
\email{xtolsa@mat.uab.cat}

\thanks{B.J. supported in part by NSF DMS-1500881 and a Kent State Summer Research award.\\
\indent F.N. supported in part by NSF DMS-1265623. \\
\indent M.R. supported partially by the Vinnoya VINNMER Marie Curie grant 2014-01434.\\
\indent X.T. supported by the ERC grant 320501 of the European Research Council (FP7/2007-2013) and partially supported by 2014-SGR-75 (Catalonia), MTM2013-44304-P (Spain), and by the Marie Curie ITN MAnET (FP7-607647).}

\date{\today}

\begin{abstract}  Fix $d\geq 2$, and $s\in (d-1,d)$.  We characterize the non-negative locally finite non-atomic Borel measures $\mu$ in $\R^d$ for which the associated $s$-Riesz transform is bounded in $L^2(\mu)$ in terms of the Wolff energy. This extends the range of $s$ in which the Mateu-Prat-Verdera characterization of measures with bounded $s$-Riesz transform is known.

As an application, we give a metric characterization of the removable sets for locally Lipschitz continuous solutions of the fractional Laplacian operator $(-\Delta)^{\alpha/2}$, $\alpha\in (1,2)$, in terms of a well-known capacity from non-linear potential theory.  This result contrasts sharply with removability results for Lipschitz harmonic functions.\end{abstract}

\maketitle


\tableofcontents

\section{Introduction}

Fix $d\geq 2$, and $s\in (d-1,d)$.  The $s$-Riesz transform is the Calder\'{o}n-Zygmund operator given by convolution with the kernel $K(x) = \tfrac{x}{|x|^{s+1}}$, $x\in \R^d$.

 In this paper we characterize the non-negative locally finite non-atomic Borel measures $\mu$ in $\R^d$ for which the associated $s$-Riesz transform is bounded in $L^2(\mu)$ in terms of the Wolff energy
$$\mathcal{W}_{2}(\mu, Q) = \int_Q\Bigl(\int_0^{\infty}\Bigl[\frac{\mu(B(x,r)\cap Q)}{r^s}\Bigl]^2\frac{dr}{r}\Bigl)d\mu(x) \;\;(Q\subset \R^d\text{ a cube}).$$

\pagebreak

\begin{thm}\label{introthm}  Fix $s\in (d-1,d)$.  Suppose that $\mu$ is a locally finite non-atomic Borel measure in $\R^d$.  Then the $s$-Riesz Transform operator associated to $\mu$ is bounded in $L^2(\mu)$, in the sense that there is a constant $C\in (0,\infty)$ such that
 $$\sup_{\eps>0}\int_{x\in \R^d}\Bigl|\int_{y\in \R^d: \, |x-y|>\eps}K(x-y)f(y)d\mu(y)\Bigl|^2d\mu(x)\leq C\|f\|^2_{L^2(\mu)}
$$
for every $f\in L^2(\mu)$, if and only if there exists a constant $\wt{C}\in (0,\infty)$ such that
\begin{equation}\label{wolffenergy}\mathcal{W}_2(\mu,Q)\leq \wt{C}\mu(Q),\text{  for every cube }Q\subset \R^d.
\end{equation}
\end{thm}

\vspace{0.1cm}

 This statement is by no means intuitive because there is no a priori reason for the boundedness of a Calder\'{o}n-Zygmund operator coming from the cancelation in the kernel to be equivalent to the boundedness of a (non-linear) positive operator.  The `only if' direction of the theorem fails for $s\in \mathbb{N}$ because the $s$-Riesz transform operator is bounded on an  $s$-plane, or more generally on a uniformly $s$-rectifiable set, see \cite{Dav1, DS}, and the Wolff energy condition fails for (the $s$-Hausdorff measure restricted to) an $s$-plane or indeed any Ahlfors-David regular set of dimension $s$.  It was therefore quite a surprise when Mateu-Prat-Verdera \cite{MPV} showed that if $s\in (0,1)$, then the $s$-Riesz transform associated to $\mu$ is bounded in $L^2(\mu)$ if and only if (\ref{wolffenergy}) holds.  The proof given in \cite{MPV} is heavily rooted to the case of $s\in (0,1)$, as it relies on the curvature technique introduced to the area by Melnikov.  Over the subsequent years, it had been conjectured that it should be possible to extend the Mateu-Prat-Verdera characterization to cover all $s\in (0,d)$, $s\not\in\mathbb{Z}$, see for instance \cite{Tol3, ENV1}.  Here we settle the case of $s\in (d-1,d)$.   The case of $s\in (1,d-1)\backslash \mathbb{Z}$ remains open.\\

The `if' direction of Theorem \ref{introthm} holds for any $s\in (0,d)$, integer or not, and is not particularly subtle.  The proof is essentially contained in the paper \cite{MPV}, see also \cite{ENV1}.  A rather similar argument had also previously appeared in Mattila's paper \cite{Mat1}.  For a concise proof see Appendix A of \cite{JN3}.  Consequently, in proving Theorem \ref{introthm} we shall be concerned with the statement that if $\mu$ is a measure whose associated $s$-Riesz transform operator is bounded in $L^2(\mu)$, then (\ref{wolffenergy}) holds for every cube.\\

The structure of a non-atomic measure $\mu$ whose associated $s$-Riesz transform is bounded in $L^2(\mu)$ has been heavily studied.  It is well known (see, for instance \cite{Dav1} or Lemma \ref{trunctobilinear} below) that this condition implies that there is a constant $C>0$ such that $\mu(B(x,r))\leq Cr^s$ for every $x\in \R^d$ and $r>0$.   Vihtil\"{a} \cite{Vih} then proved that
$$\mu\Bigl(\Bigl\{x\in \R^d: \liminf_{r\rightarrow 0}\frac{\mu(B(x,r))}{r^s}>0\Bigl\}\Bigl)=0
$$
whenever $s\not\in \mathbb{Z}$.  Notice that this is a much more qualitative conclusion than the Wolff energy condition in Theorem \ref{introthm}, which certainly implies that $\int_0^1\bigl(\frac{\mu(B(x,r))}{r^s}\bigl)^2\frac{dr}{r}<\infty$ for $\mu$-almost every $x\in \R^d$.
In spite of all the work that has taken place since, Vihtil\"{a}'s theorem remains the only result that tells us anything non-trivial about the structure of general measures with bounded $s$-Riesz transform that covers all non-integer $s$, and until recently the results that applied when $s>1$ concerned only particular types of Cantor measures, see for instance \cite{MT, Tol11, EV}.\\

Eiderman-Nazarov-Volberg \cite{ENV} subsequently made a breakthrough in this area by showing that if $s\in (d-1,d)$, and $\mathcal{H}^s(\supp(\mu))<\infty$, then $\mu$ is the zero measure.
To compare this result to Theorem \ref{introthm}, we remark that it is equivalent to the conclusion that $\lim_{r\rightarrow 0}\frac{\mu(B(x,r))}{r^s}=0$ for $\mu$-almost every $x\in \R^d$.  (Vihtil\"{a}'s theorem only guarantees that $\liminf_{r\rightarrow 0}\frac{\mu(B(x,r))}{r^s}=0$ for $\mu$-almost every $x\in \R^d$.)
The new ideas introduced in \cite{ENV} played a key role in the solution of the David-Semmes problem in co-dimension 1 recently given by Nazarov-Tolsa-Volberg \cite{NToV}.\\ 

A weak quantitative version of \cite{ENV} was proved in \cite{JNV} involving some very weak  non-linear potential of exponential type.   Reguera-Tolsa \cite{RT} then verified Theorem \ref{introthm} under the restrictive assumption that the support of $\mu$ is uniformly disconnected.  Finally, Jaye-Nazarov \cite{JN3} proved that, for general measures, $\mathcal{W}_p(\mu,Q)\leq C\mu(Q)$ for every cube $Q$, and some large $p=p(s,d)>0$.  The techniques developed in these papers will play a significant role in our analysis, as will the Nazarov-Treil-Volberg $\mathcal{T}(1)$-theorem for suppressed kernels \cite{NTV2}.\\

We remark that if one replaces the condition of the boundedness of the $s$-Riesz transform with the (morally stronger but more qualitative) condition of existence of principal values, then the situation is much better understood.  For instance, in \cite{RdVT} it is shown that if $s\in (0,d)$ and $\mu$ is a finite measure satisfying $\mathcal{H}^s(\supp(\mu))<\infty$ along with the growth condition $\mu(B(x,r))\leq r^s$ for every $x\in \R^d$ and $r>0$, then the existence of the limit $\lim_{\eps\rightarrow 0^+}\int_{|x-y|>\eps}K(x-y)d\mu(y)$ for $\mu$-almost every $x\in \R^d$ implies that $s\in \mathbb{Z}$.\\

Theorem \ref{introthm} has an interesting consequence in the theory of the fractional Laplacian:   We say that a compact set $E$ is \emph{$\alpha$-removable (in the Lipschitz category)} if every function $u$ that is Lipschitz continuous in some open neighbourhood $U$ of $E$, grows slowly at infinity in the sense that $\int_{\R^d}\frac{|u(x)|}{(1+|x|)^{d+\alpha}}dm_d(x)<\infty$, and satisfies $(-\Delta)^{\alpha/2}u =0$ in $U\backslash E$ in the sense of distributions (see Appendix \ref{distributiontheory})  in fact satisfies $(-\Delta)^{\alpha/2}u =0$ in $U$.

\begin{thm}\label{remove} Fix $\alpha\in (1,2)$.  A compact set $E$ is $\alpha$-removable in the Lipschitz category if and only if it cannot support a non-zero measure $\mu$ with \begin{equation}\label{Wolffuniform}\sup_{x\in \R^d}\int_0^{\infty}\Bigl(\frac{\mu(B(x,r))}{r^{d-\alpha+1}}\Bigl)^2\frac{dr}{r}<\infty,,\end{equation}
which is to say that $\operatorname{cap}_{\tfrac{2}{3}(\alpha-1), \tfrac{3}{2}}(E)=0$ in the language of non-linear potential theory (see \cite{AH}).
\end{thm}

We refer to the book \cite{AH} for more information on the capacity $\operatorname{cap}_{\tfrac{2}{3}(\alpha-1), \tfrac{3}{2}}(E)$, including its role in approximation theory for Sobolev spaces.  It was previously known that non-removable sets for the $\alpha$-Laplacian with $\alpha\in (1,2)$ necessarily have infinite $(d-\alpha+1)$-dimensional Hausdorff measure, and this was a consequence of the theorem in \cite{ENV} mentioned above.  Theorem \ref{remove} follows from Theorem \ref{introthm} along with Prat's \cite{Pra} extension of the theorems of Tolsa \cite{Tol3} on analytic capacity and Volberg  \cite{Vol} on Lipschitz harmonic capacity (see Appendix \ref{distributiontheory} for more details).  

The direct analogue of Theorem \ref{remove} fails for the Laplacian operator ($\alpha=2$), where hyperplanes are non-removable for the Lipschitz harmonic functions\footnote{Notice two things: (1) the function $\max(x_d,0)$ ($x=(x_1,\dots,x_d)\in \R^d$) is a Lipschitz harmonic function outside of the hyperplane $\{x_d=0\}$ that is obviously not harmonic in $\R^d$, and (2) a hyperplane cannot even support a nonzero measure $\mu$ with $\lim_{r\rightarrow 0}\tfrac{\mu(B(x,r))}{r^{d-1}}=0$ at $\mu$-almost every point $x\in \R^d$, a much weaker condition than (\ref{Wolffuniform}) with $\alpha=2$.}.  See Nazarov-Tolsa-Volberg \cite{NToV2} for a characterization of the sets of finite $(d-1)$-dimensional Hausdorff measure that are removable for Lipschitz harmonic functions in terms of rectifiability.  

\part*{Part 0: Preliminaries}

\section{Notation}

Fix $d\geq 2$ and $s\in (d-1,d)$.  We set $K(x) = \tfrac{x}{|x|^{s+1}}$ to be the $s$-Riesz kernel.

By a measure, we shall always mean a non-atomic non-negative locally finite Borel measure.  Consequently, the Borel regularity theorem (Chapter 1 of \cite{Mat}) will apply to any measure we consider, meaning that for every measure $\mu$, and Borel set $E\subset \R^d$ with $\mu(E)<\infty$, we have
\begin{equation}\nonumber\begin{split}\mu(E) &= \sup\bigl\{\mu(K): K\text{ compact, } K\subset E\bigl\} \\&= \inf\bigl\{\mu(U): U\text{ open, } U\supset E\bigl\}.\end{split}\end{equation}

We denote by $\supp(\mu)$ the closed support of $\mu$.

For a cube $Q\subset \R^d$, $\ell(Q)$ denotes its side-length.  We shall write $\ell(Q)\asymp \ell$ if $\tfrac{\ell}{2}\leq \ell(Q)\leq \ell$.  For $A>0$, we denote by $AQ$ the cube concentric to $Q$ of sidelength $A\ell(Q)$.

We define the ratio of two cubes $Q$ and $Q'$ by
$$[Q':Q]=\Bigl|\log_2 \frac{\ell(Q')}{\ell(Q)} \Bigl|.$$

The density of a cube $Q$ (with respect to a measure $\mu$) is given by $\displaystyle D_{\mu}(Q)=\frac{\mu(Q)}{\ell(Q)^s},$  while the density of an open ball $B(x,r)$ is defined by $\displaystyle D_{\mu}(B(x,r)) =  \frac{\mu(B(x,r))}{r^s}.$

For a set $U\subset \R^d$, we denote by $\Lip_0(U)$ the set of Lipschitz continuous functions that are compactly supported in the interior of $U$.

For a set $E\subset \R^d$, and a function $f$ defined on $E$ (either scalar or vector valued), we set
$$\text{osc}_E(f) = \sup_{x,x'\in E}|f(x)-f(x')|.
$$

Normally, we shall denote a large positive constant by $C$ and a small positive constant by $c$.  When new constants have to be defined in terms of some previously chosen ones, we number them.  The conventions are that all constants may depend on $d$ and $s$ in addition to parameters explicitly mentioned in parentheses, and a numbered constant with index $j$ can be chosen in terms of constants with  indices less than $j$ (say, $C_{12}$ can be chosen in terms of $c_4$ and $C_{10}$).

We will also use the notation $A\ll B$ to mean $A<c_0 B$ where $c_0=c_0(s,d)>0$ is a sufficiently small positive constant (its choice does not depend on any other constants in the paper and can be made at the very beginning).  Every time this notation is used, it should be read as ``the following argument is true, provided that $c_0$ was chosen small enough''.  The statement $A\gg B$ is equivalent to $B\ll A$.

\subsection{The lattice $\mathcal{D}$ of triples of dyadic cubes}   Let $\mathcal{Q}$  denote a dyadic lattice.

Let $\mathcal{D}=\mathcal{D}(\mathcal{Q})$ denote the lattice of concentric triples of open dyadic cubes from $\mathcal{Q}$.  Cubes in $\mathcal{D}$ are therefore not disjoint on a given level, but have finite overlap.

Set $Q_0= 3(0,1)^d = (-1,2)^d$.  For a cube $Q\in \dy$, we set $\mathcal{L}_Q$ to be the canonical linear map (a composition of a dilation and a translation) satisfying $\mathcal{L}_Q(Q_0)=Q$.  

The cubes in $\mathcal{D}$ have a natural family tree:  A cube $P\in \mathcal{D}$ is the ancestor of $Q\in \mathcal{D}$ of sidelength $2^m\ell(Q)$, $m\geq0$, if $P=3\underline{P}$ and $Q=3\underline{Q}$ where $\underline{P}$ is the unique dyadic cube containing $\underline{Q}$ with $\ell(\underline{P})=2^m\ell(\underline{Q})$.  If $m=2$ we call the corresponding ancestor of  $Q$ its grandparent.

\begin{lem}\label{dycubecontain} Suppose that $Q=3\underline{Q}\in \mathcal{D}$, and $P$ is any cube that intersects $Q$  with $\ell(P)\leq \ell(Q)$.  Then the grandparent $\wt{Q}$ of $Q$ contains $P$ (in fact, it contains the whole of $3Q$).\end{lem}

\begin{proof} The cube $\wt{Q}$ is the triple of a cube $\wt{\underline{Q}}$ that contains $\underline{Q}$, and $\ell(\wt{\underline{Q}})=4\ell(\underline{Q})$.  Consequently, it follows that $\wt{Q}= 3\wt{\underline{Q}}\supset 9\underline{Q}=3Q$, which yields the claim.\end{proof}

We endow the dyadic lattice $\mathcal{Q}$ with a graph structure $\Gamma(\mathcal{Q})$ by connecting each dyadic cube with an edge to its children, parent, and all neighbouring cubes of the same sidelength.  The graph distance on $\mathcal{Q}$ is the shortest path from $\underline{Q}\in \mathcal{Q}$ to $\underline{Q'}\in \mathcal{Q}$ in the graph $\Gamma(\mathcal{Q})$.

For $Q,Q'\in \mathcal{D}$, the symbol $d(Q,Q')$ denotes the graph distance between the dyadic cubes $\underline{Q}$ and $\underline{Q'}$ with $Q=3\underline{Q}$ and $Q'=3\underline{Q'}$.

\subsection{Weak Limits}\label{weaklimintro}

We next collect the standard facts regarding weak limits of measures that we shall use in two blow up arguments.  A good reference for the material here is Chapter 1 of \cite{Mat}.

We say that a sequence of measures $\mu_k$ converges weakly to a measure $\mu$ if
$$\lim_{k\rightarrow \infty}\int_{\R^d}fd\mu_k = \int_{\R^d}fd\mu\text{ for every } f\in C_0(\R^d),
$$
where $C_0(\R^d)$ denotes the set of compactly supported continuous functions on $\R^d$.

The separability of the space $C_0(\R^d)$, along with the Riesz representation theorem, yields the following compactness result:

\begin{lem}[Weak compactness]  If $\mu_k$ is a sequence of measures that satisfy $$\sup_{k}\mu_k(B(0,R))<\infty\text{ for every }R>0,$$ then the sequence has a weakly convergent subsequence.
\end{lem}

Using the Borel regularity theorem, it is not hard to prove the following semi-continuity properties of the weak limit.

\begin{lem}[Semi-continuity of the weak limit]\label{semicontinuity}  Suppose that $\mu_k$ converge weakly to $\mu$.  Then
\begin{enumerate}
\item $\mu(U)\leq \liminf_{k\rightarrow \infty}\mu_k(U)$ for any open set $U$,
\item $\mu(K)\geq \limsup_{k\rightarrow \infty}\mu_k(K)$ for any compact set $K$,
\item If $K_k$ are compact sets that converge to a compact set $K$ in Hausdorff metric, then
$$\mu(K)\geq \limsup_{k\rightarrow \infty}\mu_k(K_k).$$
\end{enumerate}
\end{lem}

\section{A primer on singular integrals}

Our approach calls for careful notation, as we shall need to study the convolution of a singular kernel with a measure from several standpoints.

\subsection{The potential of a finite measure}

Suppose that $\nu$ is a finite signed measure.  Then the integral $\int_{\R^d}K(x-y)d\nu(y)$ converges absolutely for $m_d$-almost every $x\in \R^d$.  Thus we can define the potential
$$\RSO(\nu)(x) = \int_{\R^d}K(x-y)d\nu(y)
$$
for $m_d$-almost every $x\in \R^d$, and \emph{for every} $x\not\in \supp(\nu)$.

Moreover, if $\nu$ has bounded density with respect to $m_d$, then the potential $\RSO(\nu)$ is a bounded continuous function on $\R^d$ that converges to zero at infinity.

Similarly, for a finite signed vector valued measure $\nu=(\nu_1,\dots, \nu_d)$, we can define the adjoint Riesz transform
$$\RSO^*(\nu)(x) = \int_{\R^d}K(x-y)\cdot d\nu(y) = \sum_j^n \int_{\R^d} \frac{x_j-y_j}{|x-y|^{s+1}}d\nu_j(y)
$$
for $m_d$-almost every $x\in \R^d$.

\subsection{Diffuse measures and an associated bilinear form}

 Let $k(\cdot,\cdot)$ be an anti-symmetric kernel satisfying $|k(x,y)|\leq \tfrac{1}{|x-y|^s}$ for $x,y\in \R^d$, $x\neq y$.

 A measure $\mu$ is said to be \textit{diffuse} in an open set $U\subset \R^d$ if the function $(x,y)\rightarrow \tfrac{\chi_U(x)\chi_U(y)}{|x-y|^{s-1}}$ belongs to $L^1_{\text{loc}}(U\times U, \mu\times\mu)$, that is, if for any compact set $K\subset U$,
 $$\iint_{K\times K}\frac{1}{|x-y|^{s-1}}d\mu(x)d\mu(y)<\infty.
 $$
 If we say that a measure is diffuse (without reference to an open set), we shall mean that it is diffuse in the entire space $\R^d$.

 For a measure $\mu$ that is diffuse in an open set $U$, and for $f,\psi \in \Lip_0(U)$, we may define
\begin{equation}\label{bilinT}\langle \TSO(f\mu),\psi \rangle_{\mu}=  \iint_{\mathbb{R}^d\times\mathbb{R}^d}k(x,y)H_{f,\psi}(x,y)d\mu(x)d\mu(y),
\end{equation}
where
$$H_{f,\psi} = \frac{1}{2}\bigl[f(y)\psi(x) - \psi(y) f(x)\bigl].
$$
Notice that $H_{f,\psi}$ is a Lipschitz continuous function on $\R^d\times\R^d$ with $H_{f,\psi}(x,x)=0$ for $x\in \R^d$.  Consequently, $|H_{f,\psi}(x,y)|\leq C(f,\psi)|x-y|$ for $x,y\in \R^d$.  On the other hand, $H_{f,\psi}$ is clearly supported in some compact subset $S$ of $U$.  Therefore,
$$ \iint\limits_{\mathbb{R}^d\times\mathbb{R}^d}|k(x,y)||H_{f,\psi}(x,y)|d\mu(x)d\mu(y)\leq \!C(f,\psi) \!\!\iint\limits_{S\times S}\!\frac{1}{|x-y|^{s-1}}d\mu(x)d\mu(y),
$$
and the right hand side here is finite since $\mu$ is diffuse.

In the event that $\int_{\R^d\times \R^d}|k(x,y)||f(y)||\psi(x)|d\mu(x)d\mu(y)<\infty$, then we can write $\langle \TSO(f\mu),\psi \rangle_{\mu} = \int_{\R^d\times \R^d}k(x,y)f(y)\psi(x)d\mu(x)d\mu(y)$.

In the case when $k(x,y) = K(x-y)$ is the Riesz kernel, we shall denote the bilinear form by $\langle \RSO(f\mu),\psi \rangle_{\mu}$.

\subsection{The extension of the bilinear form to an operator $\TSO_{\mu}$} If there exists $C>0$ such that
\begin{equation}\label{TSObdd}|\langle \TSO (f\mu),\psi\rangle_{\mu}|\leq C\|f\|_{L^2(\mu)}\|\psi\|_{L^2(\mu)}
\end{equation}
for every $f,\psi\in \Lip_0(\R^d)$, then by duality we can find a (unique) bounded linear operator $\TSO_{\mu}:L^2(\mu)\mapsto L^2(\mu)$ with norm at most $C$, satisfying
$$\langle \TSO_{\mu}(f),\psi\rangle_{\mu}=\langle \TSO(f\mu),\psi\rangle_{\mu} \text{ whenever } f,\psi\in \Lip_0(\R^d).
$$

\begin{lem}\label{restrictcoincidence}  Suppose that $\mu$ is a diffuse measure, and $\TSO_{\mu}:L^2(\mu)\rightarrow L^2(\mu)$.  For a Borel set $E\subset \R^d$, set $\mu' = \chi_{E}\mu$.  Then
$$\langle \TSO_{\mu}(\chi_{E}f),\psi\chi_{E}\rangle_{\mu} = \langle \TSO(f\mu'),\psi\rangle_{\mu'}
$$
for any $f,\psi\in \Lip_0(\R^d)$.
\end{lem}

\begin{proof} Fix $f,\psi\in \Lip_0(\R^d)$.  According to the Borel regularity theorem, we can find a sequence $g_k\in \Lip_0(\R^d)$ such that $0\leq g_k\leq 1$ on $\R^d$ and $g_k\rightarrow \chi_E$ pointwise $\mu$-almost everywhere on $\R^d$.

But then $fg_k$ and $\psi g_k$ converge to $f$ and $\psi$ respectively in $L^2(\mu)$, and so
$$\langle \TSO_{\mu}(\chi_{E}f),\psi\chi_{E}\rangle_{\mu}=\lim_{k\rightarrow \infty}\langle \TSO_{\mu}(f g_k),\psi g_k\rangle_{\mu}.
$$

Now, for each $k$,
$$\langle \TSO_{\mu}(f g_k),\psi g_k\rangle_{\mu} = \iint\limits_{\R^d\times\R^d}k(x,y)H_{f,\psi}(x,y)g_k(x)g_k(y)d\mu(x)d\mu(y).
$$
But the function $(x,y)\mapsto k(x,y)H_{f,\psi}(x,y)\in L^1(\mu\times\mu)$, so the dominated convergence theorem yields that
\begin{equation}\begin{split}\nonumber\lim_{k\rightarrow \infty} & \iint\limits_{\R^d\times\R^d}k(x,y)H_{f,\psi}(x,y)g_k(x)g_k(y)d\mu(x)d\mu(y)\\&=\iint\limits_{E\times E}k(x,y)H_{f,\psi}(x,y)d\mu(x)d\mu(y),\end{split}\end{equation}
and the lemma follows.
\end{proof}

\begin{cor}\label{obvious} If $\TSO_{\mu}$ is bounded in $L^2(\mu)$, and $\mu' = \chi_{E}\mu$, where $E$ is a Borel set, then $\TSO_{\mu'}$ is bounded in $L^2(\mu')$ and for every $f\in L^2(\mu)$ $\TSO_{\mu'}(f) = \TSO_{\mu}(f\chi_{E})$ $\mu'$-almost everywhere.
\end{cor}

 Let's now assume that $\mu$ is a \emph{finite} diffuse measure and $\TSO_{\mu}$ is bounded on $L^2(\mu)$, let us now suppose that $f,\psi\in L^2(\mu)$ satisfy $$\dist(\supp(f),\supp(\psi))>0.$$  Then we can find sequences $f_n$ and $\psi_n$ of functions in $\Lip_0(\R^d)$ such that $$\dist(\supp(f_n), \supp(\psi_n))\geq \tfrac{1}{2}\dist(\supp(f),\supp(\psi))\text{ for every }n\in \mathbb{N},$$
while $f_n\rightarrow f$ and $\psi_n\rightarrow \psi$ in $L^2(\mu)$ respectively.  Then
\begin{equation}\begin{split}\nonumber\langle \TSO_{\mu}(f_n), \psi_n\rangle_{\mu} & = \langle \TSO(f_n\mu), \psi_n\rangle_{\mu} \\&=\iint_{\R^d\times\R^d}k(x,y)f_n(y)\psi_n(x)d\mu(x)d\mu(y),
\end{split}\end{equation}
where in the second equality we have used the separation in the supports of $f_n$ and $\psi_n$ to rewrite (\ref{bilinT}) in the stated manner.  But now, we may readily use the dominated convergence theorem to pass to the limit to obtain that
$$\langle \TSO_{\mu}(f), \psi\rangle_{\mu} = \iint_{\R^d\times\R^d}k(x,y)f(y)\psi(x)d\mu(x)d\mu(y).
$$
In particular, we have the following lemma:
\begin{lem}\label{aedefine}
Suppose that $\mu$ is a finite diffuse measure, $\TSO_{\mu}$ is bounded in $L^2(\mu)$, and $f\in L^2(\mu)$.  Then for $\mu$-almost every $x\in \R^d\backslash \supp(f)$ we have that
$$\TSO_{\mu}(f)(x) = \int_{\R^d}k(x,y)f(y)d\mu(y).
$$
\end{lem}

Our next lemma will be used in order to apply certain $T(1)$-theorems from the literature.  We shall suppose that, for every $\delta>0$, there is a bounded anti-symmetric kernel $k_{\delta}$ such that
\begin{itemize}
\item $|k_{\delta}(x,y)|\leq |k(x,y)|$ for $x,y\in \R^d$, $x\neq y$, and
\item $\lim_{\delta\rightarrow 0^+}k_{\delta}(x,y)= k(x,y)$ whenever $x\neq y$.
\end{itemize}
\begin{lem}\label{trunckernelslimit}  Suppose that $\mu$ is a finite diffuse measure, and there is a family of kernels $k_{\delta}$ for $\delta>0$ satisfying the above assumptions and also that there is a constant $C>0$ such that
$$\sup_{\delta>0}\int_{\R^d}\Bigl|\int_{\R^d} k_{\delta}(x,y) f(y)d\mu(y)\Bigl|^2d\mu(x)\leq C\|f\|_{L^2(\mu)}^2
$$
for every $f\in L^2(\mu)$.  Then,
$$|\langle \TSO(f\mu), \psi\rangle_{\mu}|\leq C\|f\|_{L^2(\mu)}\|\psi\|_{L^2(\mu)}
$$
for every $f,\psi\in \Lip_0(\R^d)$.
\end{lem}

\begin{proof} Using anti-symmetry of the kernel $k_{\delta}$, along with the Cauchy-Schwarz inequality, we have that for any $\delta>0$ and $f,\psi\in \Lip_0(\R^d)$.
$$\Bigl|\iint_{\R^d\times\R^d}k_{\delta}(x,y)H_{f,\psi}(x,y)d\mu(x)d\mu(y)\Bigl|\leq C\|f\|_{L^2(\mu)}\|\psi\|_{L^2(\mu)}.
$$
On the other hand, since $\mu$ is diffuse, the dominated convergence theorem yields that
$$\Bigl|\iint_{\R^d\times\R^d}k(x,y)H_{f,\psi}(x,y)d\mu(x)d\mu(y)\Bigl|\leq C\|f\|_{L^2(\mu)}\|\psi\|_{L^2(\mu)}
$$
for every $f,\psi\in \Lip_0(\R^d)$.\end{proof}

\subsection{The truncated Riesz transform and diffuseness}

\begin{lem}\label{balldiffuseeps}  Fix a (non-atomic) measure $\mu$.  If
$$\sup_{\eps>0}\int\limits_{x\in B}\Bigl|\int\limits_{y\in B:\,|x-y|>\eps} K(x-y)d\mu(y)\Bigl|^2d\mu(x)\leq \mu(B),
$$
for some ball $B=B(x_0, r)$, then $\mu$ is diffuse in $B$, and moreover
$$\iint\limits_{B\times B}\frac{1}{|x-y|^{s-1}}d\mu(x)d\mu(y)\leq 2r\mu(B).
$$
\end{lem}

\begin{proof}  Suppose that $G\in L^2(\mu)$ is a vector field.  From the antisymmetry of the kernel $K$, we infer that
\begin{equation}\begin{split}\nonumber\Bigl|\iint\limits_{\substack{(x,y)\in B\times B\\|x-y|>\eps}}&K(x-y)\cdot(G(x)-G(y))d\mu(x)d\mu(y)\Bigl|\\&\leq 2\Bigl|\int_{x\in B}G(x)\cdot\Bigl[ \int_{y\in B:\, |x-y|>\eps} K(x-y)d\mu(y)d\mu(x)\Bigl]\Bigl|,
\end{split}\end{equation}
and by the assumption of the lemma, we have that
$$\Bigl|\int_{x\in B}G(x)\cdot \Bigl[\int_{y\in B:\, |x-y|>\eps} K(x-y)d\mu(y)d\mu(x)\Bigl]\Bigl|\leq \|G\|_{L^2(\mu)}\sqrt{\mu(B)}.
$$
Now, let $G(x) = (x-x_0)\chi_B$.  Then \begin{equation}\begin{split}\nonumber\iint\limits_{\substack{(x,y)\in B\times B\\|x-y|>\eps}}&\frac{1}{|x-y|^{s-1}}d\mu(x)d\mu(y)\\&=\iint\limits_{\substack{(x,y)\in B\times B\\|x-y|>\eps}}K(x-y)\cdot[G(x)-G(y)]d\mu(x)d\mu(y).
\end{split}\end{equation}
But $\|G\|_{L^2(\mu)}\leq r\sqrt{\mu(B)}$, and the set $\{(x,x): x\in \R^d\}$ is $\mu\times\mu$ null because $\mu$ is non-atomic, so the lemma follows from the monotone convergence theorem.
\end{proof}

The next lemma is in fact well known for all $s\in (0,d)$, see for instance David \cite{Dav1}.  We provide a quick proof in the case $s>1$ that will suffice for our purposes.

\begin{lem}\label{trunctobilinear}  Let $s\in (1,d)$.  Suppose that $\mu$ is a (non-atomic) measure such that
$$\sup_{\eps>0}\int_{\R^d}\Bigl|\int_{\R^d: \, |x-y|>\eps}K(x-y)f(y)d\mu(y)\Bigl|^2d\mu(x)\leq \|f\|^2_{L^2(\mu)}
$$
 for every $f\in L^2(\mu)$.  Then there is a constant $C>0$ such that $D_{\mu}(B(x,r))\leq C$ for every $x\in \R^d$ and $r>0$.  Consequently $\mu$ is diffuse and so
$$|\langle \RSO(f\mu), \varphi\rangle_{\mu}|\leq \|f\|_{L^2(\mu)}\|\varphi\|_{L^2(\mu)} \text{ for every }f,\varphi\in \Lip_0(\R^d).
$$
\end{lem}

\begin{proof}  From Lemma \ref{balldiffuseeps} we infer that for every ball $B=B(x_0, r)$, $$\iint\limits_{B\times B}\frac{1}{|x-y|^{s-1}}d\mu(x)d\mu(y)\leq 2r\mu(B),$$ which yields that $\mu(B(x_0,r))\leq 2^sr^s$ if $s>1$.  But now we have that $\mu$ is diffuse, and so an application of Lemma \ref{trunckernelslimit} completes the proof.
\end{proof}

\subsection{Restricted growth at infinity and reflectionless measures} If $\mu$ is diffuse in an open set $U\subset \R^d$ and has restricted growth at infinity, in the sense that $\int_{|x|\geq 1}\tfrac{1}{|x|^{s+1}}d\mu(x)<\infty$, then we may define the pairing $\langle \RSO(f\mu),\varphi \rangle_{\mu}$ when $f \in \Lip_0(U)$ satisfies $\int_{\mathbb{R}^d}f\,d\mu=0$, and $\varphi$ is merely a bounded Lipschitz function.  To do this, fix $\psi\in \Lip_0(U)$ that is identically equal to $1$ on a neighbourhood of the support of $f$, and set
$$\langle \RSO(f\mu),\varphi \rangle_{\mu} = \langle \RSO(f\mu),\psi\varphi \rangle_{\mu}+ \int_{\mathbb{R}^d}\RSO(f\mu)(x)[1-\psi(x)]\varphi(x) \,d\mu(x).$$
The mean zero property of $f$ ensures that $|\RSO(f\mu)(x)|\leq \tfrac{C_{f,\psi}}{(1+|x|)^{s+1}}$ for $x\in \supp(1-\psi)$, which combined with the restricted growth at infinity implies that the second integral converges absolutely.  The value of $\langle \RSO(f\mu),\varphi \rangle_{\mu}$ does not depend on the particular choice of $\psi$.

We say that a measure $\mu$, diffuse in $U$ with restricted growth at infinity, is \textit{reflectionless} in $U$ if
$$\langle \RSO(f\mu),1 \rangle_{\mu}=0\text{ for every }f \in \Lip_0(U)\text{ satisfying } \int_{\mathbb{R}^d}f\,d\mu=0.
$$
If we say that a measure is reflectionless without reference to an open set $U$, we mean that it is reflectionless in the entire space $\R^d$.

\section{A revised statement}

We now reduce Theorem \ref{introthm} to the statement that we shall spend the remainder of the paper proving.  First notice that there is a constant $C>0$ such that for any measure $\mu$, and any $x\in \R^d$, $\int_0^{\infty}D_{\mu}(B(x,r))^2\frac{dr}{r}\leq C\sum_{Q\in \mathcal{D}}D_{\mu}(Q)^2\chi_Q(x).$  Therefore
$$\int_{\R^d}\Bigl[\int_0^{\infty}D_{\mu}(B(x,r))^2\frac{dr}{r}\Bigl]d\mu(x)\leq C\sum_{Q\in \mathcal{D}}D_{\mu}(Q)^2\mu(Q).
$$

Suppose that a non-atomic measure $\mu$ satisfies
\begin{equation}\label{epsbdd}\sup_{\eps>0}\int_{\R^d}\Bigl|\int_{\R^d: \, |x-y|>\eps}K(x-y)f(y)d\mu(y)\Bigl|^2d\mu(x)\leq \|f\|^2_{L^2(\mu)}
\end{equation}
for every $f\in L^2(\mu)$.  Then for any cube $Q$, the inequality (\ref{epsbdd}) continues to hold if we replace the measure $\mu$ with its restriction to $Q$.  In addition, Lemma \ref{trunctobilinear} ensures that $\sup_{x\in \R^d, \, r>0}D_{\mu}(B(x,r))$ is bounded by some absolute constant.  Therefore, in order to prove (the `only if' direction of) Theorem \ref{introthm}, it suffices to establish the following result:

\begin{thm}\label{thm}  Let $s\in (d-1,d)$.  Suppose that $\mu$ is a finite measure satisfying the growth condition $\sup_{x\in \R^d, \, r>0}D_{\mu}(B(x,r))<\infty$.  If
\begin{equation}\label{bilinearthmest}|\langle \RSO(f\mu), 1\rangle_{\mu}|\leq\|f\|_{L^2(\mu)}\sqrt{\mu(\R^d)}
\end{equation}
 for every $f\in \Lip_0(\R^d)$ with $\int_{\R^d} fd\mu=0$.  Then
$$\sum_{Q\in \mathcal{D}}D_{\mu}(Q)^2\mu(Q) \leq C \mu(\mathbb{R}^d),
$$
where $C>0$ depends only on $s$ and $d$.
\end{thm}

\section{The general scheme:  Finding a large Lipschitz oscillation coefficient}

Fix a (locally finite non-negative Borel) measure $\mu$.  For $A>0$, and a cube $Q\in \mathcal{D}$, define the set of functions
$$\Psi^{A}_{\mu}(Q) = \Bigl\{\psi\in \Lip_0(AQ) : \|\psi\|_{\Lip}\leq\frac{1}{\ell(Q)}, \int_{\mathbb{R}^d}\psi d\mu=0\Bigl\}.$$
The system  $\Psi^A_{\mu}(Q)$ $(Q\in \mathcal{D})$ forms a Riesz system, that is, there exists a constant $C(A)>0$, such that for any 
sequence $(a_Q)_Q\in \ell^2(\mathcal{D})$ with only finitely many non-zero entries, and every choices of $\psi_Q\in \Psi^A_{\mu}(Q)$ ($Q\in \mathcal{D}$),
$$\Bigl\|\sum_{Q\in \mathcal{D}}\frac{a_Q \psi_Q}{\sqrt{\mu(3AQ)}}\Bigl\|_{L^2(\mu)}^2 \leq C(A)\|a_Q\|_{\ell^2}^2.$$
The reader may consult Appendix B of \cite{JN3} for the simple proof of this fact.

If $\mu$ is diffuse in $AQ$ with restricted growth at infinity, we may define the \textit{Lipschitz oscillation coefficient}
$$\Theta_{\mu}^{A}(Q) =\sup_{\psi\in \Psi^{A}_{\mu}(Q)}|\langle \RSO(\psi \mu), 1 \rangle_{\mu}|.
$$

If in addition $\mu$ is a finite measure such that (\ref{bilinearthmest}) holds, then we obtain from the Riesz system property and duality that for any choices of $\psi_Q\in \Psi_{\mu}^A(Q)$ ($Q\in \mathcal{D}$),
\begin{equation}\label{afterdual}\sum_{Q\in \mathcal{D}}\frac{|\langle \RSO(\psi_Q \mu), 1 \rangle_{\mu}|^2}{\mu(3AQ)}\leq C(A)\mu(\R^d).
\end{equation}
 To verify this, we may assume that only finitely many functions $\psi_Q$ are chosen to be non-zero.  Then the left hand side of the previous inequality equals the square of
$$\sum_{Q\in \mathcal{D}}\frac{a_Q\cdot\langle \RSO(\psi_Q \mu), 1 \rangle_{\mu}}{\sqrt{\mu(3AQ)}},
$$
for some finite (vector valued) sequence $a_Q=(a_Q^{(1)},\dots, a_Q^{(d)})$ with $\|a_Q\|_{\ell^2}\leq 1$.  This in turn can be bounded by $\sum_{j=1}^d|\langle\RSO(f_j\mu),1\rangle_{\mu}|$,
where
$$f_j=\sum_{Q\in \mathcal{D}}\frac{a_Q^{(j)}\psi_Q}{\sqrt{\mu(3AQ)}}.
$$
But the Riesz system property tells us that $\|f_j\|_{L^2(\mu)}\leq C(A)$.  Thus, from (\ref{bilinearthmest}) we infer that for each $j\in \{1,\dots, d\}$, $|\langle\RSO(f_j\mu),1\rangle_{\mu}|\leq C(A)\sqrt{\mu(\R^d)},$ and so (\ref{afterdual}) follows.

We arrive at the following simple lemma (see also \cite{JN3} Lemma 4.2).

\begin{lem}\label{LipOscWolff}   Suppose that $\mu$ is a finite diffuse measure and (\ref{bilinearthmest}) holds.  Let $\mathcal{F}(\mu)\subset \mathcal{D}$.  If there exist $A>0$ and $\Delta>0$ such that
$$\Theta_{\mu}^A(Q)\geq \Delta D_{\mu}(Q)\mu(Q) \text{ for every }Q\in \mathcal{F}(\mu),
$$
then
\begin{equation}\label{Fest}\sum_{Q\in \mathcal{F}(\mu)}D_{\mu}(Q)^2\frac{\mu(Q)}{\mu(3AQ)}\mu(Q)\leq \frac{C(A)}{\Delta^2}\mu(\mathbb{R}^d).
\end{equation}
\end{lem}

Our goal is to find a rule $\mathcal{F}$ that associates to each measure $\mu$ some family of cubes $\mathcal{F}(\mu)\subset \mathcal{D}$, along with universal constants $A>0$, $\Delta>0$, and $c>0$, so that the following two things occur:

(\textbf{A})  (Large Wolff Potential).  If a measure $\mu$ satisfies $\sup_{Q\in \mathcal{D}}D_{\mu}(Q)<\infty$, then
\begin{equation}\label{Fgenest}\sum_{Q\in \mathcal{F}(\mu)}D_{\mu}(Q)^2\frac{\mu(Q)}{\mu(3AQ)}\mu(Q)\geq c \sum_{Q\in \mathcal{D}}D_{\mu}(Q)^2\mu(Q).
\end{equation}

(\textbf{B}) (Large Lipschitz Oscillation Coefficient).  For any measure $\mu$ and $Q\in \mathcal{F}(\mu)$, we have that $\mu$ is diffuse in $AQ$ with restricted growth at infinity, and moreover
\begin{equation}\label{Deltgenest}\Theta_{\mu}^A(Q)\geq \Delta D_{\mu}(Q)\mu(Q).
\end{equation}

Once such a rule has been found, Theorem \ref{thm} will follow from Lemma \ref{LipOscWolff} by comparing (\ref{Fest}) and (\ref{Fgenest}).

For a fixed rule $\mathcal{F}$ under consideration, establishing whether (\textbf{A}) holds is usually a routine matter. To prove (\textbf{B}) we would like to argue via contradiction.  We shall assume that there exist no universal $A$ and $\Delta$ so that (\ref{Deltgenest}) holds.   Then, for arbitrarily large $A\gg1$, we can find a measure $\mu$ and a cube $Q\in \mathcal{F}(\mu)$ so that the quotient
$$\frac{\Theta_{\mu}^A(Q)}{D_{\mu}(Q)\mu(Q)}
$$
is arbitrarily small.  Re-scaling the cube $Q$ to $Q_0$, and re-scaling the measure $\mu$ accordingly so that $\mu(Q_0)=1$, we thereby arrive at a sequence of measures $\mu_k$ satisfying $\mu_k(Q_0)=1$ and $\Theta_{\mu_k}^A(Q_0)\rightarrow 0$ as $k\rightarrow \infty$.  Provided that the measures $\mu_k$ are sufficiently regular (we shall require that they are \emph{uniformly diffuse} in $AQ_0$, as defined in Section \ref{reflprelims}), we may pass to the weak limit to deduce that there is a measure $\mu$ with  $\mu(\overline{Q}_0)\geq 1$ that is \emph{reflectionless} in $AQ_0$.  The game then becomes to exploit any properties that $\mu$ has inherited from the rule $\mathcal{F}$ to preclude its existence if $A$ is sufficiently large, and so conclude that (\textbf{B}) holds after all.

As such, there are two distinct parts of the proof needed to establish (\textbf{B}) via contradiction:\\
\indent Part I: A blow-up argument leading to the existence of a certain reflectionless measure.\\
\indent Part II: Proving the non-existence of said reflectionless measure.

The blow-up techniques presented in this paper (Sections \ref{energysec}--\ref{localize})  can be trivially adapted to any $s$-dimensional Calder\'{o}n-Zygmund operator with smooth (away from the diagonal) homogeneous kernel with $s\in (0,d)$.  It is the non-existence results that require the restriction to the $s$-Riesz transform with $s\in (d-1,d)$.

Examples of diffuse measures with restricted growth at infinity that are reflectionless in $\R^d$ for the $s$-Riesz transform with $s\in (d-1,d)$ include $m_d$ (the $d$-dimensional Lebesgue measure on $\R^d$), $\mathcal{H}^{d-1}|_L$ for some hyperplane $L$ (the $d-1$-dimensional Hausdorff measure restricted to a hyperplane), and $\sum_{j\in \mathbb{Z}}\mathcal{H}^{d-1}|_{L+jv}$, where $L$ is a hyperplane and $v$ is a vector perpendicular to $L$.  If one could provide a full description of the reflectionless measures associated to the $s$-Riesz transform, then it is likely that a much more straightforward proof of Theorem \ref{thm} could be found.  This, however, remains an interesting open problem.

In lieu of a complete description of the reflectionless measures for the $s$-Riesz transform, we employ two non-existence results for reflectionless measures.  Firstly, we use Proposition \ref{regularprop} below, which is the main technical result for the Riesz transform from \cite{JN3}.  The second non-existence result and its proof occupies the latter half of this paper (beginning in Section \ref{pt2scheme}), and is a synthesis of the techniques introduced in \cite{ENV} and \cite{RT}.




 Currently, all non-existence results (at least for $s>1$) rely heavily on the theory of the differential operator associated to the Riesz transform (i.e., the fractional Laplacian).  While there is a degree to which this is natural, we suspect that our techniques are currently overreliant on this relationship.

We shall make use of two refining procedures on the lattice $\mathcal{D}$ in defining the rule $\mathcal{F}$.  These procedures are introduced in the next two sections.

\section{Upward domination}

Fix $B\gg 1$, $0<a\ll 1$, and $0< \eps \ll 1$.  From the start, we shall assume that $a^{-1}\ll B$.  Eventually $a^{-1}$ will be chosen to be of the order of a very small positive power of $B$.  We shall assume that $a$ and $B$ are both powers of $2$.


Fix a measure $\mu$.

\begin{defn}    We say that $Q'\in \dy$  \emph{dominates}  $Q\in \dy$  \emph{from above} if $aBQ'\supset BQ$ and
$$D_{\mu}(Q')\geq 2^{\eps  [Q':Q]}D_{\mu}(Q).
$$
The set of those cubes in $\mathcal{D}$ that cannot be dominated from above by another cube in $\mathcal{D}$ is denoted by $\dysel(\mu)$ (or just $\dysel$).
\end{defn}

Of course, in order for $Q'$ to dominate $Q$ from above we must have that $\ell(Q')> \ell(Q)$.  Also notice that domination from above is transitive:  If $Q'$ dominates $Q$ from above, and $Q''$ dominates $Q'$ from above, then $Q''$ dominates $Q$ from above.

We begin by showing that cubes in $\dysel$ make up a noticeable portion of the Wolff potential.

\begin{lem}\label{upwolff}  Suppose that  $\sup_{Q\in \mathcal{D}}D_{\mu}(Q)<\infty$.  Then there exists a constant $c(B,\eps )>0$ such that
$$\sum_{Q\in \dysel(\mu)}D_{\mu}(Q)^2\mu(BQ) \geq c(B,\eps) \sum_{Q\in \dy}D_{\mu}(Q)^2\mu(BQ).
$$
\end{lem}

\begin{proof}

We first claim that every $Q\in \dy\backslash \dysel$ that satisfies $\mu(Q)>0$ can be dominated from above by a cube $\wt{Q}\in \dysel$.

To see this, first note that if $Q'$ dominates $Q$ from above, then we must have that
$$[Q':Q]\leq\frac{1}{\eps}\log_2\Bigl(\frac{\sup_{Q''\in \mathcal{D}}D_{\mu}(Q'')}{D_{\mu}(Q)}\Bigl),$$ or else $Q'$ would have density larger than $\sup_{Q''\in \mathcal{D}}D_{\mu}(Q'')$.  As any such $Q'$ must also satisfy $aBQ'\supset BQ$, we see that there are only finitely many candidates for a cube that dominates $Q$ from above.

We choose $\wt{Q}\in \dy$ to be a cube of maximal sidelength that dominates $Q$ from above.  Since domination from above is transitive, we conclude that $\wt{Q}\in \dysel$.  The claim is proved.


Now, for each $Q\in \mathcal{D} \backslash \dysel$  with $\mu(Q)>0$, we choose a cube $\widetilde{Q}\in \dysel$ that dominates $Q$ from above.  Certainly we have that $BQ\subset aB\widetilde{Q}\subset B\widetilde{Q}$.

For each fixed $P\in \dysel$, write
\begin{equation}\begin{split}\nonumber\sum_{Q\in \dy\backslash \dysel: \widetilde{Q}=P}& D_{\mu}(Q)^2\mu(BQ) = \sum_{m\geq 1} \sum_{\substack{Q\in \dy\backslash \dysel:\\\ell(Q)=2^{-m}\ell(P), \widetilde{Q}=P}} D_{\mu}(Q)^2\mu(BQ)\\
&\leq \sum_{m\geq 1}2^{-2\eps  m}D_{\mu}(P)^2\Bigl[\sum_{\substack{Q\in \dy:\\\ell(Q)=2^{-m}\ell(P), BQ\subset BP}} \mu(BQ)\Bigl]
\end{split}\end{equation}
We examine the term in square brackets.   The sum is taken over cubes $Q$ of a fixed level, with $BQ\subset BP$.  Consequently, it is bounded by $CB^d\mu(BP)$.

By summing over such $P\in \dysel$, we see that
\begin{equation}\nonumber \sum_{Q\in \dy\backslash \dysel}D_{\mu}(Q)^2\mu(BQ)\leq C(B,\eps)\sum_{Q\in\dysel}D_{\mu}(Q)^2\mu(BQ).\end{equation}
This inequality clearly proves the lemma.
\end{proof}



Before continuing, we fix a parameter regime
\begin{equation}\label{Bepregime}B^{\eps }\leq 2,\end{equation}
and  $a^{-2}\ll B$ (much more than this will be assumed in due course).

\begin{lem}\label{cubecontain} If a cube $Q''\cap BQ \neq \varnothing$, and $\ell(Q'')\geq \tfrac{4}{a}\ell(Q)$, then $aBQ''\supset BQ$.
\end{lem}

\begin{proof} There is a point $z\in Q''\cap BQ$.  But then the cube $\wt{Q}$ centred at $z$ of sidelength $\bigl(aB-1\bigl)\ell(Q'')$ is contained in $aBQ''$.  Since $aB>2$, we have that $aB-1\geq \tfrac{aB}{2}$.  On the other hand $\ell(Q'')\geq \tfrac{4}{a}\ell(Q)$, and so $\bigl(aB-1\bigl)\ell(Q'')\geq 2B\ell(Q)$.  Thus $\wt{Q}\supset BQ$, which yields the claim.
\end{proof}

\begin{lem}\label{upcontrol} There is a constant $C_1>0$ (depending on $s$ and $d$) such that if $Q\in \dysel$, then for any cube $Q'$ (not necessarily in $\mathcal{D}$) that satisfies $Q'\cap BQ\neq \varnothing$ and $\ell(Q')\geq a^{-1}\ell(Q)$, we have
\begin{equation}\label{case12est}D_{\mu}(Q')\leq C_1 2^{\eps [Q':\,Q]}D_{\mu}(Q).\end{equation}
\end{lem}

\begin{proof} Certainly the cube $4Q'$ intersects $BQ$, and $\ell(4Q')\geq \tfrac{4}{a}\ell(Q)$. So Lemma \ref{cubecontain} ensures that $aB(4Q')\supset BQ$.  On the other hand, if we choose a cube $P\in \mathcal{D}$ of sidelength between $4\ell(Q')$ and $8\ell(Q')$ that intersects $4Q'$, then its grandparent $\wt{P}$ contains $4Q'$ (this is Lemma \ref{dycubecontain}).  Consequently, $aB\wt{P}\supset BQ$, and, since $Q\in \dysel$ we have
$$D_{\mu}(\wt{P})\leq 2^{\eps[\wt{P}:\, Q]}D_{\mu}(Q).
$$
We need now only notice that $\wt{P}\supset Q'$ and $\ell(\wt{P})\leq C\ell(Q')$ to deduce that
$D_{\mu}(Q')\leq C 2^{\eps[Q':\, Q]}D_{\mu}(Q)$,
and the lemma follows.
 \end{proof}

 \begin{cor}\label{controlcor}  If $Q\in \dysel$, and $Q'$ is any cube, then the following two statements hold:
 \begin{enumerate}
 \item if $Q'\cap BQ \neq \varnothing$ satisfies  $ a^{-1}\ell(Q)\leq \ell(Q')\leq B\ell(Q)$, then $D_{\mu}(Q')\leq 2C_1 D_{\mu}(Q)$,
\item if $Q'\cap \tfrac{B}{2}Q\neq \varnothing$ satisfies $\ell(Q)\leq \ell(Q')\leq B\ell(Q)$, then $D_{\mu}(Q')\leq \tfrac{2\cdot4^sC_1}{a^s}D_{\mu}(Q)$.
\end{enumerate}
 \end{cor}

 \begin{proof}
 In the case of statement (1), we have that $2^{\eps[Q':\, Q]}\leq 2^{\eps \log_2B}\leq 2$, where (\ref{Bepregime}) has been used in the final inequality.  Plugging this into (\ref{case12est}) yields the statement.

 The second estimate is only not already proved in the case that $\ell(Q')\leq \tfrac{1}{a}\ell(Q)$.  But in this case, consider the enlargement $Q''=\tfrac{4}{a}Q'$.  Then $\ell(Q'')<B\ell(Q')$ as $a^{-2}\ll B$, and so
$$D_{\mu}(Q')\leq \frac{4^s}{a^s}D_{\mu}(Q'')\leq 2\frac{4^sC_1}{a^s}D_{\mu}(Q),
$$
where the second inequality follows from statement (1).
\end{proof}

\begin{lem}\label{Bcontain}  There is a constant $C_2>0$ such that if $Q\in \dysel$, and  $Q'\in \mathcal{D}\backslash \dysel$ satisfies $Q'\cap  \tfrac{1}{2}BQ\neq\varnothing$ and $D_{\mu}(Q')\geq C_2 a^{-s}2^{\eps[Q:\,Q']}D_{\mu}(Q)$, then every cell $Q''\in\dysel$ dominating $Q'$ from above satisfies
$$BQ''\subset BQ.
$$
\end{lem}

\begin{proof}
For the claimed inclusion to fail, we must have that a dominating cube $Q''$ satisfies $\ell(Q'')\geq \tfrac{1}{2}\ell(Q)$.  In this case, we infer from Lemma \ref{cubecontain} that  $\wt{Q}''$, the ancestor of $Q''$ of sidelength $\tfrac{16}{a}\ell(Q'')$, satisfies $aB\widetilde{Q}''\supset BQ$.  But then notice that (\ref{Bepregime}) ensures that $2^{\eps[\wt{Q}'': Q'']}\leq 2$, and so
\begin{equation}\begin{split}\nonumber D_{\mu}(\widetilde{Q}'')&\geq \frac{a^s}{16^s}D_{\mu}(Q'')\geq  \frac{a^s}{2\cdot16^s}2^{\eps [\widetilde{Q}'', Q'']}D_{\mu}(Q'')\\&\geq\frac{a^s}{2\cdot16^s}2^{\eps [\widetilde{Q}'':\,Q']}D_{\mu}(Q')\geq \frac{a^s}{2\cdot16^s}2^{\eps [\widetilde{Q}'':\,Q]}\frac{C_{2}}{a^s}D_{\mu}(Q).
\end{split}\end{equation}
But now if $C_2\geq 2\cdot16^s$, the right hand side is at least $2^{\eps [\widetilde{Q}'':\,Q]}D_{\mu}(Q)$.   This contradicts the assumption that $Q\in \dysel$.
\end{proof}

\section{The shell and the downward domination}

We now describe a second refinement process.  Fix a measure $\mu$, and $Q\in \dysel(\mu)$.


Let us assume that we have some way of associating a closed cube $\wh{Q}^{\mu}$ to each cube $Q\in \dysel(\mu)$ so that $2aBQ\subset \widehat{Q}^{\mu} \subset 4aBQ$.  We refer to $\wh{Q}^{\mu}$ as the \textit{shell} of $Q$.

We postpone the precise selection of the shell cubes until later.  When it does not cause too much confusion, we shall just write $\wh{Q}$ instead of $\wh{Q}^{\mu}$, but the reader should always keep in mind that the choice of the shell will depend on the underlying measure.

\begin{defn}  We say that $Q\in \dysel$ \textit{is dominated from below by a (finite) bunch of cubes} $Q_j$ if the following conditions hold:
\begin{itemize}
\item  $Q_j\in \dysel$,
\item $D_{\mu}(Q_j)\geq 2^{\eps  [Q:Q_j]}D_{\mu}(Q)$,
\item $BQ_j$ are disjoint,
\item $BQ_j\subset BQ$,
\item $\displaystyle\sum_j D_{\mu}(Q_j)^22^{-2\eps [Q:Q_j]}\mu(\wh Q_j) \geq D_{\mu}(Q)^2 \mu(\wh Q).$
\end{itemize}
We define $\dyselA=\dyselA(\mu)$ to be the set of all cubes $Q$ in $\dysel$ that cannot be dominated from below by a bunch of cubes except for the trivial bunch consisting of $Q$.
\end{defn}

Notice that if $Q_1, \dots, Q_N$ is a non-trivial bunch of cubes that dominate $Q$ from below, then $\ell(Q_j)\leq \tfrac{\ell(Q)}{2}$ for every $j$ or else the property that $BQ_j\subset BQ$ would fail.

\begin{lem}\label{downwolff}  Suppose that $\sup_{Q\in \mathcal{D}}D_{\mu}(Q)<\infty$.  There exists $c(B,\eps)>0$  such that
$$\sum_{Q\in \dyselA}D_{\mu}(Q)^2\mu(\wh Q) \geq c(B,\eps) \sum_{Q\in \dysel}D_{\mu}(Q)^2\mu(\wh Q).
$$
\end{lem}

\begin{proof}We start with a simple claim.

\textbf{Claim.}   Every $Q\in \dysel$ with $\mu(Q)>0$ is dominated from below by a bunch of cubes $P_{Q,j}$ in $\dyselA(\mu)$.

To prove the claim we make two observations.  Firstly, note that if the bunch $Q_1, \dots, Q_N$ dominates $Q'\in \dysel$ from below, and if $Q_1$ is itself dominated from below by a bunch $P_1, \dots, P_{N'}$, then the bunch $P_1, \dots, P_{N'}$, $ Q_2, \dots, Q_N$ dominates $Q'$.

The second observation is that, since any cube $Q'$ participating in a bunch of cubes that dominates a cube $Q$ from below satisfies $D_{\mu}(Q')\geq 2^{\eps  [Q:Q']}D_{\mu}(Q)$, we have that
$$[Q:Q']\leq\frac{1}{\eps}\log_2\Bigl(\frac{\sup_{Q''\in \mathcal{D}}D_{\mu}(Q'')}{D_{\mu}(Q)}\Bigl),$$ or else $Q'$ would have density larger than $\sup_{Q''\in \mathcal{D}}D_{\mu}(Q'')$.  

Now, if the cube $Q$ lies in $\dyselA$, then we are done already.  Otherwise, we find some non-trivial bunch $Q_1,\dots, Q_N$ of cubes in $\dysel$ that dominates $Q$ from below with $\ell(Q_j)\leq \tfrac{\ell(Q)}{2}$ for each $j$.  If for each $j=1,2,\dots N$, $Q_j$ lies in $\dyselA$ then we are done.  Otherwise, we replace each cube $Q_j\not\in \dyselA$ with a non-trivial bunch of cubes $Q_{j,k}$ that dominates $Q_j$ from below with $\ell(Q_{j,k})\leq \tfrac{\ell(Q_j)}{2}$.  The resulting bunch consisting of the cubes $Q_j\in \dyselA$ along with the cubes $Q_{j,k}$ for $Q_j\not\in \dyselA$ dominates $Q$ from below (the transitive property).  This process can be iterated at most $\frac{1}{\eps}\log_2\bigl(\tfrac{\sup_{Q''\in \mathcal{D}}D_{\mu}(Q'')}{D_{\mu}(Q)}\bigl)$ times before it must terminate in a bunch of cubes $P_{Q,j}$ that dominate $Q$ from below, each of which lies in $\dyselA$.


Now write
\begin{equation}\begin{split}\nonumber\sum_{Q\in \dysel}& D_{\mu}(Q)^2\mu(\wh{Q}) \leq \sum_{Q\in \dysel}\sum_j D_{\mu}(P_{Q,j})^2\mu(\wh{P}_{Q,j})2^{-2\eps [Q:P_{Q,j}]}\\
& \leq \sum_{P\in \dyselA}D_{\mu}(P)^2\mu(\wh{P})\Bigl[\sum_{Q: BQ\supset BP}2^{-2\eps [Q:P]}\Bigl].
\end{split}\end{equation}
But the inner sum does not exceed $\tfrac{CB^d}{\eps }$, and this proves the lemma.\end{proof}

\section{The main goal}

We begin with a lemma.

\begin{lem}\label{wolffrelate}  There is a constant $c(B, \eps)>0$ such that for any measure $\mu$ satisfying $\sup_{Q\in \dy}D_{\mu}(Q)<\infty$,
$$\sum_{Q\in \dyselA(\mu)}D_{\mu}(Q)^2\frac{\mu(Q)}{\mu(3BQ)}\mu(Q) \geq c(B,\eps)\sum_{Q\in \mathcal{D}}D_{\mu}(Q)^2\mu(Q).
$$
\end{lem}

\begin{proof}
By Lemma \ref{upcontrol}, we have that if $Q\in \dysel(\mu)$, then
\begin{equation}\label{3Bcontrol}\mu(3BQ)\leq CB^s2^{\eps[3BQ: Q]}\mu(Q)\leq CB^s\mu(Q),
\end{equation}
where (\ref{Bepregime}) has been used in the second inequality.  Since $\wh{Q}\subset BQ$, we see from (\ref{3Bcontrol}) that
$$\sum_{Q\in \dyselA(\mu)}D_{\mu}(Q)^2\frac{\mu(Q)}{\mu(3BQ)}\mu(Q)\geq cB^{-2s}\sum_{Q\in \dyselA(\mu)}D_{\mu}(Q)^2 \mu(\wh{Q}).
$$
Applying Lemma \ref{downwolff}, we see that the right hand side of this inequality is at least
$$c(B, \eps)\sum_{Q\in \dysel(\mu)}D_{\mu}(Q)^2 \mu(\wh{Q}).
$$
But now (\ref{3Bcontrol}) also yields
$$\sum_{Q\in \dysel(\mu)}D_{\mu}(Q)^2 \mu(\wh{Q})\geq c B^{-s}\sum_{Q\in \dysel(\mu)}D_{\mu}(Q)^2 \mu(BQ),
$$
after which we infer from Lemma \ref{upwolff} that
$$\sum_{Q\in \dysel(\mu)}D_{\mu}(Q)^2 \mu(\wh{Q})\geq c(B, \eps)\sum_{Q\in \mathcal{D}}D_{\mu}(Q)^2 \mu(BQ).
$$
The lemma follows.
\end{proof}

We now state our main goal.

\begin{goal}\label{mainlipgoal}  We want to choose $B$, $a$, and $\eps$ satisfying the restrictions imposed in all previous sections and define the shells $\wh{Q}^{\mu}$ in an appropriate way to ensure that there exists $\Delta>0$ such that for any measure $\mu$ and any $Q\in \dyselA(\mu)$, $\mu$ is diffuse in $\tfrac{B}{2}Q$ and
$$\Theta_{\mu}^{B/2}(Q)\geq \Delta D_{\mu}(Q)\mu(Q).
$$
\end{goal}

Notice that once this goal has been achieved, Theorem \ref{thm} will follow.  Indeed, Lemma \ref{LipOscWolff} ensures that if $\mu$ is a finite measure with bounded density for which (\ref{bilinearthmest}) holds, then there is a constant $C(\eps, a, B)$ such that
$$\sum_{Q\in \dyselA(\mu)}D_{\mu}(Q)^2\frac{\mu(Q)}{\mu(\tfrac{3B}{2}Q)}\mu(Q)\leq \frac{C(\eps,a, B)}{\Delta^2}\mu(\R^d).$$
But then the theorem follows from Lemma \ref{wolffrelate}.


\part*{Part I:  The blow-up procedures} 

\section{Preliminary results regarding reflectionless measures}\label{reflprelims}

To perform blow up arguments, we need some weak continuity properties of the form $\langle \RSO(f\mu),1 \rangle_{\mu}$.  Following Section 8 of \cite{JN}, we make the following definitions.

 A sequence of measures $\mu_k$ is called \textit{uniformly diffuse} in the cube $R_0Q_0$ (where $R_0$ may equal $+\infty$, in which case $R_0Q_0=\mathbb{R}^d$) if, for each $R<R_0$ and $\delta>0$, there exists $r>0$ such that for all sufficiently large $k$,
\begin{equation}\label{closesmall}\iint\limits_{\substack{RQ_0\times RQ_0\\|x-y|<r}}\frac{\,d\mu_k(x)\,d\mu_k(y)}{|x-y|^{s-1}}\leq \delta.
\end{equation}

A sequence of measures $\mu_k$ is said to have \textit{uniformly restricted growth (at infinity)} if, for each $\delta>0$, there exists $R\in(0,\infty)$ such that for all sufficiently large $k$,
\begin{equation}\label{tailR}\int_{\mathbb{R}^d\backslash \overline{RQ_0}}\frac{1}{|x|^{s+1}}\,d\mu_k(x)\leq \delta.
\end{equation}

For $R>0$, define
$$\Psi_{\mu}^R = \Bigl\{\psi\in \Lip_0(RQ_0): \int_{\mathbb{R}^d}\psi d\mu=0, \|\psi\|_{\Lip}<1\Bigl\}.
$$

We shall appeal to the following weak convergence lemma.  The proof is an exercise in the definitions\footnote{Also required is the fact that if $\mu_k$ are locally finite Borel measures that converge weakly to $\mu$, then $\mu_k\times \mu_k$ converge weakly to $\mu\times\mu$.   This is really a statement about the density of linear combinations of functions of the form $(x,y)\mapsto f(x)g(y)$, with $f,g\in C_0(\R^d)$, in the space $C_4(\R^d\times \R^d)$, which can be proved using the Stone-Weierstrass Theorem.} but the details may be found in Section 8 of \cite{JN}.

\begin{lem}\label{weakconv}  Suppose that $\mu_k$ is a sequence of measures that are uniformly diffuse in $R_0Q_0$ (for $R_0\in (0, +\infty]$) with uniformly restricted growth at infinity.  Further assume that the sequence $\mu_k$ converges weakly to a measure $\mu$ (and so $\mu$ is diffuse in $R_0Q_0$, and has restricted growth at infinity).  Suppose that $\gamma_k$ is a non-negative sequence converging to zero, and $R_k$ is a sequence converging to $R_0$.

If, for every $k$,
$$|\langle \RSO(\psi\mu_k), 1 \rangle_{\mu_k}|\leq \gamma_k \text{ for every } \psi\in \Psi_{\mu_k}^{R_k},
$$
then
$$\langle \RSO(\psi\mu), 1\rangle_{\mu}=0 \text{ for every } \psi\in \Psi_{\mu}^{R_0},
$$
i.e., $\mu$ is reflectionless in $R_0Q_0$.
\end{lem}

Finally, let us recall a non-existence result on reflectionless measures from \cite{JN3}. It is an immediate consequence of combining Proposition 2.3 and Lemma 5.6 in \cite{JN3}.

\begin{prop}\label{regularprop}  Let $s\in (d-1,d)$.  There exists $\eps_0 >0$ depending on $d$ and $s$ such that for any dyadic lattice $\mathcal{Q}$ and the associated lattice of triples $\mathcal{D}=\mathcal{D}(\mathcal{Q})$, the only measure $\mu$ that is reflectionless in $\mathbb{R}^d$ and satisfies the estimate
$$D_{\mu}(Q')\leq 2^{\eps_0d(Q',Q_0)}D_{\mu}(\overline{Q}_0) \text{ for all }Q'\in \mathcal{D}
$$
is the zero measure.
\end{prop}



\section{The basic energy estimates}\label{energysec}

We define the energy of a set $F\subset \R^d$ with respect to $\mu$ by
$$\mathbb{E}^{\mu}(F) = \iint\limits_{F\times F}\frac{1}{|x-y|^{s-1}}d\mu(x)d\mu(y).
$$

Estimates on the energy will play a substantial role in the following arguments.

When proving that a sequence of measures is uniformly diffuse, truncated energy integrals naturally arise.   For $r>0$, and a set $F\subset \mathbb{R}^d$, set
$$\mathbb{E}^{\mu}_{r}(F) = \iint\limits_{\substack{F\times F:\\ |x-y|<r}}\frac{1}{|x-y|^{s-1}}d\mu(x)d\mu(y).
$$

In order to best utilize the special properties of a cube in $\dyselA$, it will be useful to consider a `dyadic' analogue of the truncated energy.  If $P\subset \R^d$ is a cube, and $r>0$, then set
$$\mathcal{E}_{r}^{\mu}(P) = \sum_{\substack{Q'\in \mathcal{D}\\\ell(Q')\leq r}}\ell(Q')D_{\mu}(Q')\mu(Q'\cap P).
$$
We just write $\mathcal{E}^{\mu}(P)$ for $\mathcal{E}^{\mu}_{\ell(P)}(P)$.

\begin{lem}\label{dyintdom}There is a constant $C>0$ such that for any cube $P$, and $r>0$,
$$\mathbb{E}_{r}^{\mu}(P) \leq C \mathcal{E}_{8r}^{\mu}(P), \text{ and }\;\mathbb{E}^{\mu}(P) \leq C \mathcal{E}^{\mu}(P).
$$
\end{lem}

\begin{proof} Fix $r>0$.  For $x\in \R^d$ and $y\in B(x,r)$, $y\neq x$, choose $Q'\in \dy$ such that $x\in Q'$ and $|x-y|\leq \ell(Q')< 2|x-y|$.  Then Lemma \ref{dycubecontain} ensures that the grandparent  $\wt{Q}'$ of $Q'$ contains $y$, and so $\tfrac{1}{|x-y|^{s-1}}\leq 8^{s-1}\ell(\wt{Q}')^{-(s-1)}$.  Consequently, for any $y\in B(x,r)$,
$$\frac{1}{|x-y|^{s-1}} \leq C\sum_{\substack{Q'\in \mathcal{D}\\ \ell(Q')< 8r}}\frac{1}{\ell(Q')^{s-1}}\chi_{Q'}(x)\chi_{Q'}(y).
$$
Integrating both sides of the inequality over $P\times P$ with respect to $\mu \times \mu$ yields the first inequality.

For the second inequality, note that
$$\iint\limits_{\substack{P\times P:\\|x-y|\geq \tfrac{1}{8}\ell(P)}}\frac{1}{|x-y|^{s-1}}d\mu(x)d\mu(y)\leq C\ell(P)D_{\mu}(P)\mu(P).
$$
But then by the pigeonhole principle there must be a cube $Q'\in \dy$ with $\ell(Q')\asymp\ell(P)$ and $\mu(P)\leq C\mu(Q'\cap P)$.  The second inequality now follows from the first.
\end{proof}

We shall need several accurate estimates regarding the contribution toward the energy from cubes of different types.

\emph{For the remainder of Section \ref{energysec}, fix a measure $\mu$ and $Q\in \dyselA(\mu)$.}

\subsection{The High Density Energy}

Set $C_3 = \max(C_2, 2\cdot4^{s}C_1)$.  Consider the set
$$
HD = \Bigl\{Q'\in \mathcal{D}: Q'\cap \tfrac{B}{2}Q\neq \varnothing ;\; D_{\mu}(Q')>\frac{C_3}{a^s}2^{\eps[Q:Q']}D_{\mu}(Q)\Bigl\}
$$
of high density cubes intersecting $\tfrac{B}{2}Q$.    Notice that Lemma \ref{upcontrol} along with Corollary \ref{controlcor} imply that any cube $Q'\in HD$ must satisfy $\ell(Q')<\ell(Q)$.

It will be convenient to estimate the total energy coming from high density cubes.  For $m\in \mathbb{Z}_+$, set
$$\mathcal{E}_{HD, 2^{-m}\ell(Q)} (\tfrac{B}{2}Q) = \sum_{\substack{Q'\in HD,\\ \ell(Q')\leq 2^{-m}\ell(Q)}}\ell(Q')D_{\mu}(Q')\mu(Q'\cap \tfrac{B}{2}Q).
$$

\begin{prop}\label{highdensityenergy}  There is a constant $C>0$ such that for every $m\geq 0$,
$$\mathcal{E}_{HD,2^{-m}\ell(Q)}(\tfrac{B}{2}Q)\leq Ca^{-(d-s)}(1+m)2^{-(1-\eps)m}\ell(Q)D_{\mu}(Q)\mu(\wh Q),
$$
where $\wh{Q}$ is the shell of $Q$.
\end{prop}

To prove this proposition, it will be convenient to split up the collection of cubes $HD$.  Fix $T>C_3a^{-s}$ and $m\in \mathbb{Z}_+$.   Define
\begin{equation}\begin{split}\nonumber HD_{m,T} = \Bigl\{ Q'\in \dy:\, \,Q'\cap \tfrac{B}{2}Q\neq \varnothing,& \,\ell(Q')=2^{-m}\ell(Q),\\&   D_{\mu}(Q')\geq T2^{\eps [Q,\,Q']}D_{\mu}(Q)\Bigl\}.
\end{split}\end{equation}

\begin{lem}\label{hdlayermeas} There is a constant $C>0$ such that
$$\sum_{Q'\in HD_{m,T}}\mu(Q')\leq \frac{C(m+1)}{a^dT^2}\mu(\wh Q).
$$
\end{lem}

\begin{proof}
For each $Q'\in HD_{m,T}$, there is a cube $Q''\in \dysel$  satisfying

$\bullet$ $Q'\subset aBQ''\subset \tfrac{1}{2}\wh{Q}''$,

$\bullet$ $BQ''\subset BQ$, and

$\bullet$ $D_{\mu}(Q'')\geq T2^{\eps [Q: Q'']}D_{\mu}(Q).$

Indeed, either $Q'\in \dysel$ and $Q''=Q'$, or $Q'\not\in \dysel$, and we set $Q''$ to be any cube in $\dysel$ that dominates $Q'$ from above.  In the second case, the first and third of the claimed properties follow from the definitions, while the second property follows from Lemma \ref{Bcontain}.

It is clear that $0\leq [Q:\,Q'']\leq m$.   For $n\in [0,m]$, consider those $Q''$ with $[Q:Q'']=n$.  Then the corresponding shells $\wh{Q}''$ differ in size by at most a factor of $2$.  Thus, we may cover the union of the sets $\tfrac{1}{2}\wh{Q}''$ by a collection $\mathcal{G}_{n}$ of the cubes $\wh{Q}''$ with bounded overlap\footnote{For instance, with $\rho = 2aB2^{-n}\ell(Q)$, choose a maximal $\rho/16$ separated set $(x_k)_k$ in $\bigcup \bigl\{\tfrac{1}{2}\wh{Q}'': [Q:Q'']=n\bigl\}$. Each point $x_k$ lies in some $\tfrac{1}{2}\wh{Q}''_k$, and the collection $\wh{Q}''_k$ satisfies all the required properties.}.  Since $$\frac{1}{4a}\leq \frac{\ell(BQ'')}{\ell(\wh{Q}'')}\leq \frac{1}{2a},$$ the collection $\mathcal{G}_{n}$ can in turn be split into at most $\tfrac{C}{a^d}$ disjoint subfamilies $\mathcal{G}_{n,j}$, so that within each subfamily the cubes $BQ''$ are pairwise disjoint.  Now $Q$, as a member of $\dyselA$, cannot be dominated from below by a bunch, and so we must have that for each $j$,
$$T^2D_{\mu}(Q)^2\sum_{Q''\in \mathcal{G}_{n,j}} \mu(\wh{Q}'')\leq \sum_{Q''\in \mathcal{G}_{n,j}}D_{\mu}(Q'')^22^{-2\eps [Q: Q'']}\mu(\wh{Q}'') \leq D_{\mu}(Q)^2\mu(\wh{Q}).$$
After a summation in $j$ we arrive at
$$T^2\sum_{Q''\in \mathcal{G}_{n}} \mu(\wh{Q}'') \leq \frac{C}{a^d}\mu(\wh{Q}).
$$
But since every $Q'\in HD_{m,T}$ is covered by the union of the cubes $\wh{Q}''$ with $\wh{Q}''\in \cup_{n=0}^{m}\mathcal{G}_{n}$, the desired inequality follows after a summation in $n$.
\end{proof}

\begin{proof}[Proof of Proposition \ref{highdensityenergy}]
The proof of the proposition is rather routine with the previous lemma in hand.   For fixed $n,k\in \mathbb{Z}_+$, and $T=2^k C_3a^{-s}$, each $Q'$ in the set $HD_{n,T}\backslash HD_{n, 2T}$ satisfies
$$ 2^{k}C_3a^{-s}2^{\eps[Q:Q']}D_{\mu}(Q)\leq D_{\mu}(Q')< 2^{k+1}C_3a^{-s}2^{\eps[Q:Q']}D_{\mu}(Q).$$
Consequently, by writing
$$\mathcal{E}_{HD,2^{-m}\ell(Q)}(\tfrac{B}{2}Q)\leq \sum_{n\geq m}2^{-n}\ell(Q)\sum_{k \in \mathbb{Z}_+}\sum_{\substack{Q'\in HD_{n,T}\backslash HD_{n, 2T}\\T=2^k C_3a^{-s}}}D_{\mu}(Q')\mu(Q'),
$$
we may bound $\mathcal{E}_{HD,2^{-m}\ell(Q)}(\tfrac{B}{2}Q)$ by
$$Ca^{-s}\ell(Q)D_{\mu}(Q)\sum_{n\geq m}2^{-n+\eps n}\sum_{k\in \mathbb{Z}_+}2^{k}\sum_{\substack{Q'\in HD_{n, T}\\T = 2^kC_3a^{-s}}}\mu(Q').
$$
Applying the estimate of Lemma \ref{hdlayermeas} with $T=C_32^k a^{-s}$, we get
$$\mathcal{E}_{HD,2^{-m}\ell(Q)}(\tfrac{B}{2}Q)\leq Ca^{s-d}\ell(Q)D_{\mu}(Q)\mu(\wh{Q})\sum_{n\geq m}2^{-n+\eps n}(n+1)\sum_{k\in \mathbb{Z}_+} 2^{-k},
$$
from which the result follows.
\end{proof}

\subsection{The small cube energy}

We call a cube $Q'\in \dy$ \emph{small} if $\ell(Q')\leq \tfrac{1}{a}\ell(Q)$.  Recall that every cube in $HD$ has sidelength at most $\ell(Q)$, and so, in particular, these cubes are small.

For $R\subset \tfrac{B}{2}Q$ and $m\in \mathbb{Z}_+$, set
$$\mathcal{E}_{\text{small},\, \tfrac{2^{-m}}{a}\ell(Q)}(R) = \sum_{\substack{Q'\in \mathcal{D}:\\ \ell(Q')\leq \frac{2^{-m}}{a}\ell(Q)}}\ell(Q')D_{\mu}(Q')\mu(Q'\cap R).
$$

\begin{lem}\label{smallcubeenergy} There is a constant $C>0$ such that for any $R\subset\tfrac{B}{2}Q$,
\begin{equation}\begin{split}\nonumber\mathcal{E}_{\operatorname{small}, \,\tfrac{2^{-m}}{a}\ell(Q)}(R)\leq \frac{C}{a^{d-s}}&\frac{(m+1)2^{-m(1-\eps)}}{a^{1-\eps}}\ell(Q)D_{\mu}(Q)\mu(\wh{Q})\\&+\frac{C}{a^{s}}\frac{2^{-m(1-\eps)}}{a^{1+\eps}}\ell(Q)D_{\mu}(Q)\mu(R).
\end{split}\end{equation}
\end{lem}

\begin{proof}  We first consider the contribution of the high density cubes.  Set $m' =\max(m-\log_2(1/a),0)$.  Then certainly $\mathcal{E}_{HD, \,2^{-m'}\ell(Q)} (\tfrac{B}{2}Q)$ bounds the contribution of those high density cubes in $HD$ toward $\mathcal{E}_{\operatorname{small}, \,\tfrac{2^{-m}}{a}\ell(Q)}(R)$. But applying Proposition \ref{highdensityenergy} yields that \begin{equation}\begin{split}\nonumber\mathcal{E}_{HD, 2^{-m'}\ell(Q)} (\tfrac{B}{2}Q)&\leq \frac{C}{a^{d-s}}\frac{(m'+1)}{2^{m'(1-\eps)}}\ell(Q)D_{\mu}(Q)\mu(\wh{Q})\\&\leq \frac{C}{a^{d-s}}\frac{(m+1)2^{-m(1-\eps)}}{a^{1-\eps}}\ell(Q)D_{\mu}(Q)\mu(\wh{Q}).\end{split}\end{equation} The remaining small cubes $Q'$ have density $D_{\mu}(Q')\leq \tfrac{C_3}{a^s}2^{\eps[Q:Q']}D_{\mu}(Q)$, and therefore
\begin{equation}\begin{split}\nonumber\sum_{\substack{Q'\in \mathcal{D}: \,Q'\not\in HD,\\ \ell(Q')\leq \tfrac{2^{-m}}{a}\ell(Q)}}&\ell(Q')D_{\mu}(Q')\mu(Q'\cap R)\\&\leq \frac{C_3}{a^s}D_{\mu}(Q)\sum_{\substack{n\in \mathbb{Z}:\\n\geq m-\log_2(1/a)}}2^{-n}2^{\eps|n|}\ell(Q)\sum_{\substack{Q'\in\dy:\\\ell(Q')=2^{-n}\ell(Q)}}\mu(Q'\cap R).
\end{split}\end{equation}
With $n$ fixed, the inner sum satisfies
$$\sum_{\substack{Q'\in \mathcal{D}:\\\ell(Q')=2^{-n}\ell(Q)}}\mu(Q'\cap R)\leq C\mu(R),
$$
after which the summation in $n$ yields the required estimate, as
$$\sum_{n\geq m-\log_2(1/a)}2^{-n}2^{\eps|n|}\leq C\frac{2^{-m(1-\eps)}}{a^{1+\eps}}.
$$
The lemma is proved.
\end{proof}

\subsection{The large cube energy}  A cube $Q'$ is called \emph{large} if $\ell(Q')\geq \tfrac{1}{a}\ell(Q).$  Corollary \ref{controlcor} ensures that all the large cubes that intersect $BQ$ with $\ell(Q')\leq B\ell(Q)$ satisfy $D_{\mu}(Q')\leq CD_{\mu}(Q)$.

Let $R\subset \tfrac{B}{2}Q$, and $m\in \mathbb{Z}_+$.  Define
$$\mathcal{E}_{\operatorname{large},\, 2^{-m}\ell(R)}(R) = \sum_{\substack{Q'\in \dysel,\, \ell(Q')\leq 2^{-m}\ell(R):\\Q' \text{ is large}}}\ell(Q')D_{\mu}(Q')\mu(Q'\cap R).
$$

The large cube energy is simple to estimate:

\begin{lem}\label{largecubeenergy}  There is a constant $C>0$ such that
$$\mathcal{E}_{\operatorname{large},2^{-m}\ell(R)}(R)\leq C2^{-m}D_{\mu}(Q)\ell(R)\mu(R).
$$
\end{lem}

\begin{proof}Since $R\subset \tfrac{B}{2}Q$, any large cube $Q'$ intersecting $R$ with $\ell(Q')\leq \ell(R)$ satisfies the density estimate $D_{\mu}(Q')\leq CD_{\mu}(Q)$. For fixed $n\geq 0$, this density estimate yields that
\begin{equation}\begin{split}\nonumber\sum_{\substack{Q'\in \dy,\, \ell(Q')\asymp 2^{-n}\ell(R):\\Q' \text{ is large}}}&\ell(Q')D_{\mu}(Q')\mu(Q'\cap R)\\&\leq C2^{-n}\ell(R)D_{\mu}(Q)\sum_{\substack{Q'\in \mathcal{D}:\\ \ell(Q')\asymp 2^{-n}\ell(R)}}\mu(Q'\cap R).
\end{split}\end{equation}
But the right hand side is at most $C2^{-n}\ell(R)D_{\mu}(Q)\mu(R)$.  Summing the resulting inequalities over all $n\geq m$, we arrive at the estimate.
\end{proof}

A particular consequence of Lemmas \ref{smallcubeenergy} and \ref{largecubeenergy} applied with $m=0$ and $R=\tfrac{B}{2}Q$ is that
$$\mathcal{E}^{\mu}(\tfrac{B}{2}Q)<\infty.
$$
Due to Lemma \ref{dyintdom}, this implies that
$$\mathbb{E}^{\mu}(\tfrac{B}{2}Q)<\infty,
$$
from which we conclude that $\mu$ is diffuse in $\tfrac{B}{2}Q$.

\section{Blow up I:  The density drop}

First note that for any measure $\mu$ and $Q\in \dyselA(\mu)$,  the energy estimates of the previous section ensure that $\mu$ is diffuse in $\tfrac{B}{2}Q$, while the upward control in Lemma \ref{upcontrol} certainly implies that $\mu$ has restricted growth at infinity.  Consequently, it makes sense to talk about the Lipschitz Oscillation coefficient $\Theta^{B/2}_{\mu}(Q)$.

The goal of this section is to show that cubes $Q\in \dyselA(\mu)$ already have a large Lipschitz Oscillation coefficient unless they have a drop of density at many scales above and around $Q$.  The remaining sections of this paper will then concern cubes with such a density drop.

Now fix $\eps_1$ with $ \eps_1 \ll \eps_0$, with $\eps_0$ as in Proposition \ref{regularprop}.  (One can fix $\eps_1$ to be exactly $\eps_1=c_0\eps_0$.)  Assuming that $\eps\ll \eps_1$, our goal is to prove the following technical proposition.

\begin{prop}\label{densdropprop}
There exists $\beta>0$,  $\Delta>0$, $B_0\gg1 $, and $a_0\ll 1$, such that if $B>0$, $a>0$, and $\eps>0$ satisfy $$B\geq B_0, \; a\leq a_0,\; 1/a^{\beta}\ll B, \text{ and }B^{\eps }\leq 2,$$ then for every measure $\mu$ and $Q\in \dyselA(\mu)$ we have that either,

(i) (Large Oscillation coefficient.) $\Theta_{\mu}^{B/2}(Q)\geq \Delta D_{\mu}(Q)\mu(Q)$, or

(ii) (Large and lasting drop in density.) $D_{\mu}(Q')\leq a^{\eps_1}D_{\mu}(Q)$ for all $Q'\in \dy$ with $Q'\cap\tfrac{B}{4}Q\neq \varnothing$ and $a^2B\ell(Q)\leq \ell(Q')\leq \sqrt{a}B\ell(Q).$

\end{prop}

We begin by examining a describing operation that will be carried out under the assumption that part (ii) of the proposition fails.

\subsection{The $\eps_0$-regular cube}\label{epsregcubesec}

Consider a measure $\mu$, and $Q\in \dyselA(\mu)$.   Suppose that there is a cube $R\in \dy$ intersecting $\tfrac{B}{4}Q$ that satisfies $a^2B\ell(Q)\leq \ell(R)\leq \sqrt{a}B\ell(Q)$ and $D_{\mu}(R)\geq a^{\eps_1}D_{\mu}(Q)$.  (In other words, the second alternative in Proposition \ref{densdropprop} fails for this measure $\mu$ and cube $Q\in \dyselA$.)

Consider all cubes $P\in \dy$ satisfying $P\subset \tfrac{B}{2}Q$ and $\ell(P)\geq \tfrac{1}{a}\ell(Q)$, and amongst them choose a cube $R^*$ that maximizes the quantity $$D_{\mu}(P)2^{-\eps_0d(R,P)},$$
where $\eps_0$ is the constant appearing in Proposition \ref{regularprop}.

Notice that the triangle inequality for the graph metric $d(\cdot, \cdot)$ ensures that whenever $Q'\subset \tfrac{B}{2}Q$ satisfies $\ell(Q')\geq \tfrac{1}{a}\ell(Q)$, we have $D_{\mu}(Q')\leq 2^{\eps_0d(Q', R^*)}D_{\mu}(R^*)$, since otherwise $R^*$ would not be a maximizer.  In particular, if $a^{-3}\ll B$ (as we shall henceforth assume), then $\ell(R^*)\geq \tfrac{1}{a}\ell(Q)$ and so $$D_{\mu}(R^{*})2^{-\eps_0d(R,R^{*})}\geq D_{\mu}(R).$$

We first show that there is a bound for the distance between $R^*$ and $R$ that depends on $a$ only.

\begin{cla}\label{maxclose}  $d(R, R^*)\leq \tfrac{1}{4}\log_2(1/a)$.
\end{cla}

\begin{proof}Corollary \ref{controlcor} ensures that $D_{\mu}(R^{*})\leq CD_{\mu}(Q)$, from which we deduce the chain of inequalities
$$CD_{\mu}(Q)2^{-\eps_0d(R,R^{*})}\geq D_{\mu}(R^{*})2^{-\eps_0d(R,R^{*})}\geq D_{\mu}(R)\geq a^{\eps_1}D_{\mu}(Q).
$$
The penultimate inequality follows from the maximizing property of $R^*$.  But this implies that $$d(R, R^{*})\leq \frac{\eps_1}{\eps_0}\log_2 (1/a)+\frac{C}{\eps_0}\leq \frac{1}{4}\log_2(1/a),$$
where the final inequality follows since $\eps_1\ll 1$ and $a\ll 1$  (recall that $\eps_0>0$ is fixed in terms of $d$ and $s$).
\end{proof}

Let us now record two consequences of this statement:

$\bullet$ The logarithmic ratio $[R:R^{*}]\leq \tfrac{1}{4}\log_2 1/a$, and so
$$a^{9/4}B\ell(Q)\leq \ell(R^{*})\leq a^{1/4}B\ell(Q).
$$

$\bullet$ The Euclidean distance from $R$ to $R^{*}$ is at most $Ca^{-1/4}\sqrt{a}B\ell(Q) = Ca^{1/4}B\ell(Q)$.  In particular, since $a\ll 1$, this distance estimate ensures that
\begin{equation}\label{aenlargedeepB}a^{-1/8}R^*\subset \tfrac{B}{3}Q,\end{equation}  thus ensuring that $R^{*}$ is deep inside $\tfrac{B}{2}Q$.

\begin{lem}\label{largenbhdeps0} If $Q'\in \dy$ satisfies $Q'\cap a^{-1/8}R^*\neq \varnothing$ and $\ell(Q')\geq \tfrac{1}{a}\ell(Q)$, then
$$D_{\mu}(Q')\leq 2^{\eps_0d(Q', R^*)}D_{\mu}(R^*).
$$
\end{lem}

\begin{proof}  If $Q'\subset \tfrac{B}{2}Q$, then the required estimate follows from the maximizing property of $R^*$.  Otherwise, since $a^{-1/8}R^*\subset \tfrac{B}{3}Q$, we must have that $Q'\cap \tfrac{B}{3}Q\neq \varnothing$ and $Q'\cap (\tfrac{B}{2}Q)^c\neq\varnothing$, so $\ell(Q')\geq \tfrac{B}{12}\ell(Q)$.

Now, from Lemma \ref{dycubecontain} we infer that the ancestor $\wt{Q}'$ of sidelength $64\ell(Q')$ contains $BQ$ (the ancestor of $Q'$ of sidelength $16\ell(Q')$ has sidelength at least $B\ell(Q)$, and so its grandparent contains $BQ$).  Thus,
$$D_{\mu}(Q')\leq C D_{\mu}(\wt{Q}')\leq C2^{\eps [\wt{Q}': Q]}D_{\mu}(Q)\leq C2^{\eps [Q': Q]}D_{\mu}(Q).
$$
 On the other hand, since $\ell(R^*)\leq a^{1/4}B\ell(Q)$ and $\ell(Q')\geq \tfrac{B}{12}\ell(Q)$, we have that \begin{equation}\label{largecontroldistance}d(Q', R^*)\geq [Q':R^*]\geq c\log_2\tfrac{1}{a}\gg 1.\end{equation}

Now, notice that
$$[Q':Q]\leq [Q':R^*]+[Q:R^*]\leq [Q':R^*]+\log_2 B.
$$
Thus $C2^{\eps [Q':Q]}\leq C2^{\eps [Q':R^*]+1}\leq 2^{\tfrac{\eps_0}{2}d(Q', R^*)}$ as $\eps\ll\eps_0$.  It remains to observe the following chain of inequalities $D_{\mu}(Q)\leq a^{-\eps_1}D_{\mu}(R)\leq a^{-\eps_1}D_{\mu}(R^*)= 2^{\eps_1\log_2 1/a}D_{\mu}(R^*)\leq 2^{\tfrac{\eps_0}{2}d(Q', R^*)}D_{\mu}(R^*)$, where the final inequality follows from (\ref{largecontroldistance}) and the fact that $\eps_1\ll 1$ (again, $\eps_0>0$ is a fixed constant).\end{proof}

\begin{cor}\label{rescalmeas} There is a constant $c>0$ such that if $a^{-3}\ll B$, then
\begin{itemize}
\item $D_{\mu}(R^*)\mu(R^*)\geq D_{\mu}(Q)\mu(Q),$ and
\item $D_{\mu}(R^*)\mu(R^*)\geq c a^{3s}D_{\mu}(Q)\mu(BQ).$
\end{itemize}
\end{cor}

\begin{proof} We bring together two estimates: On the one hand, we have the density estimate $D_{\mu}(R^*)\geq a^{\eps_1}D_{\mu}(Q).$ On the other hand, we have the length estimate $\ell(R^*)\geq a^{9/4}B\ell(Q)$.

In particular,  these two estimates combine to ensure that $\mu(R^*)\geq a^{\eps_1+9s/4}B^s\mu(Q)$, and so, since $\eps_1\ll 1$, $$D_{\mu}(R^*)\mu(R^*)\geq a^{2\eps_1+\tfrac{9s}{4}}B^sD_{\mu}(Q)\mu(Q)\geq a^{3s}B^sD_{\mu}(Q)\mu(Q),$$
and the right hand side is at least $D_{\mu}(Q)\mu(Q)$ if $a^{-3}\ll B$.  But also $D_{\mu}(BQ)\leq CD_{\mu}(Q)$, and so $D_{\mu}(R^*)\mu(R^*)\geq ca^{3s}D_{\mu}(Q)\mu(BQ).$ \end{proof}


\subsection{Energy estimates around $R^*$}



We now use the estimates of the previous section to record a crucial energy estimate.

\begin{lem}\label{Rstarenergy}  There are constants $C>0$, and $\beta_2\geq 3$ depending on $d$ and $s$, such that for every $A\in (1,a^{-1/8})$,
$$\frac{1}{A\ell(R^*)}\mathcal{E}_{2^{-m}A\ell(R^*)}^{\mu}(AR^*)\leq C\Bigl[\frac{a^{-\beta_2}}{AB}+A^{s+2\eps_0}2^{-m(1-\eps_0)}\Bigl]D_{\mu}(R^*)\mu(R^*).
$$
\end{lem}

\begin{proof}
We recall that a cube is small if $\ell(Q')\leq \tfrac{1}{a}\ell(Q)$.  In Lemma \ref{smallcubeenergy}, we estimated the small cube energy in $\tfrac{B}{2}Q$.  Note that since $a^{-1/8}R^*\subset \tfrac{B}{2}Q$, we have
$$\sum_{\substack{Q'\in \dy:\\ Q'\text{ small}}}\ell(Q')D_{\mu}(Q')\mu(Q'\cap AR^*)\leq \mathcal{E}_{\text{small},\tfrac{\ell(Q)}{a}}(\tfrac{B}{2}Q).
$$
But now using Corollary \ref{rescalmeas} to bound the right hand side of the estimate appearing in Lemma \ref{smallcubeenergy}, we find that there is some $\beta_1>0$ such that
$$\sum_{\substack{Q'\in \dy:\\ Q'\text{ small}}}\ell(Q')D_{\mu}(Q')\mu(Q'\cap AR^*)\leq Ca^{-\beta_1}\ell(Q)D_{\mu}(R^*)\mu(AR^*).
$$
Now recall that $\ell(R^*)\geq a^{9/4}B\ell(Q)$, and so
$$\frac{1}{A\ell(R^*)}\sum_{\substack{Q'\in \dy:\\ Q'\text{ small}}}\ell(Q')D_{\mu}(Q')\mu(Q'\cap AR^*)\leq C\frac{a^{-\beta_1-\tfrac{9}{4}}}{AB}D_{\mu}(R^*)\mu(AR^*).
$$
Now recall from Lemma \ref{largenbhdeps0} that for every $Q'$ with $Q'\cap a^{-1/8}R^*\neq \varnothing$ and $\ell(Q')\geq \tfrac{1}{a}\ell(Q)$, we have that
\begin{equation}\label{largenbhdeps0recap}D_{\mu}(Q')\leq 2^{\eps_0d(Q', R^*)}D_{\mu}(R^*).
\end{equation}
This allows us to estimate the remaining part of the energy in a quite straightforward manner.  Fix some $n\geq m$, and consider the sum
$$\frac{1}{A\ell(R^*)}\sum_{\substack{ Q'\in \dy,Q'\cap AR^*\neq \varnothing:\\ Q' \text{ is large}\\ \ell(Q')\asymp2^{-n}A\ell(R^*)}}\ell(Q')D_{\mu}(Q')\mu(Q'\cap AR^*).$$
Notice that if $Q'\cap AR^*\neq \varnothing$ and $\ell(Q')\asymp 2^{-n}A\ell(R^*)$, then $d(Q', R^{*})\leq \log_2 A+n+C$.  Thus we bound the previous sum using (\ref{largenbhdeps0recap}) by
$$ C2^{-n}A^{\eps_0}2^{\eps_0n}D_{\mu}(R^*)\!\!\!\!\!\!\!\sum_{\substack{Q'\in \dy\\ \ell(Q')\asymp2^{-n}\ell(AR^*)}}\mu(Q'\cap AR^*),
$$
which is at most $CA^{\eps_0}2^{-n(1-\eps_0)}D_{\mu}(R^*)\mu(AR^*)$.  After a summation over $n\geq m$, we get
$$\frac{1}{A\ell(R^*)}\!\!\!\!\!\!\!\!\sum_{\substack{Q'\in \dy:\\ Q'\text{ large},\\ \ell(Q')\leq2^{-m}A\ell(R^*)}}\!\!\!\!\ell(Q')D_{\mu}(Q')\mu(Q'\cap AR^*)\leq
C\frac{A^{\eps_0}}{2^{m(1-\eps_0)}}D_{\mu}(R^*)\mu(AR^*).
$$
Finally, we claim that
$$\mu(AR^*)\leq CA^{s+\eps_0}\mu(R^*)\leq Ca^{-(s+\eps_0)/8}\mu(R^*),
$$
from which the lemma will follow.  To prove the claim, note that $AR^{*}$ is contained in the union of at most $3^d$ cubes in $\mathcal{D}$ of sidelength at most $2A\ell(R^*)$.  Lemma \ref{largenbhdeps0} ensures that each of those cubes has density at most $CA^{\eps_0}D_{\mu}(R^*)$, and the claim follows.
\end{proof}

Of course, because of Lemma \ref{dyintdom}, the previous lemma also yields the estimate
\begin{equation}\label{readytorescaleRstar}\frac{1}{A\ell(R^*)}\mathbb{E}_{2^{-m}A\ell(R^*)}^{\mu}(AR^*)\leq C\Bigl[\frac{a^{-\beta_2}}{AB}+A^{s+2\eps_0}2^{-m(1-\eps_0)}\Bigl]D_{\mu}(R^*)\mu(R^*),
\end{equation}
whenever $1<A<a^{-1/8}$ and $m\geq 0$.

\subsection{The rescaling}  Recall that $Q_0 = 3(0,1)^d$, and for a cube $R$, $\mathcal{L}_R$ is the canonical linear map satisfying $\mathcal{L}_R(Q_0)=R$. Define the rescaled measure

$$\mu^{*}(\, \cdot\,) = \frac{\mu(\mathcal{L}_{R^*}(\,\cdot\,))}{\mu(R^*)}.$$

Thus $\mu^{*}(Q_0)=1$ and $D_{\mu^{*}}(Q_0)=\tfrac{1}{3^s}$. Notice that the pre-image of cubes from $\mathcal{D}$ under the mapping $\mathcal{L}_{R^*}$ are contained in a lattice $\dy^*$ consisting of concentric triples of cubes from a shifted dyadic lattice $\mathcal{Q}^*$.


Note the following scaling property of the energy:
\begin{equation}\nonumber\begin{split}\frac{1}{A\ell(Q_0)D_{\mu^*}(Q_0)\mu^*(Q_0)}&\mathbb{E}_{2^{-m}A\ell(Q_0)}^{\mu^*}(AQ_0)\\&= \frac{1}{A\ell(R^*)D_{\mu}(R^*)\mu(R^*)}\mathbb{E}_{2^{-m}A\ell(R^*)}^{\mu}(AR^*).
\end{split}\end{equation}

Consequently, we deduce from Lemma \ref{Rstarenergy} that
\begin{equation}\label{rescaleRstar}\frac{1}{A}\mathbb{E}_{2^{-m}A\ell(Q_0)}^{\mu^*}(AQ_0) \leq C\Bigl[\frac{a^{-\beta_2}}{AB}+A^{s+2\eps_0}2^{-m(1-\eps_0)}\Bigl],
\end{equation}
whenever $1<A<a^{-1/8}$ and $m\geq 0$.

Finally, we examine what happens to Lemma \ref{largenbhdeps0} under the scaling.  Recall that $\ell(R^*)\geq a^{9/4}B\ell(Q)$, so any cube $Q'\in \mathcal{D}$ with $\ell(Q')\geq \tfrac{1}{a^{13/4}B}\ell(R^*)$ that intersects $a^{-1/8}R^*$ satisfies $D_{\mu}(Q')\leq 2^{\eps_0 d(Q', R^*)}D_{\mu}(R^*)$.  Consequently, we arrive  at the following result.

\begin{cor}\label{epsregscalecont}  Any cube $Q'\in \mathcal{D}^*$ with $\ell(Q')\geq \tfrac{1}{a^{13/4}B}\ell(Q_0)$ that intersects the cube $a^{-1/8}Q_0$ satisfies $D_{\mu^*}(Q')\leq 2^{\eps_0d(Q', Q_0)}D_{\mu^*}(Q_0)$.
\end{cor}

\section{Proof of Proposition \ref{densdropprop}}


Fix $\beta>\max(\beta_2,8)$.

Suppose that the proposition fails to hold.  Then there are sequences $B_k\rightarrow\infty$, $a_k\rightarrow 0$, and $\eps^{(k)} > 0$, satisfying $B_k^{\eps^{(k)}}\leq 2$ and $\tfrac{1}{a_k^{\beta}}\ll B_k$, along with  measures $\widetilde{\mu}_k$, and  cubes $Q_k\in \dyselA(\widetilde{\mu}_k)$, such that
$$|\langle \RSO(\varphi\widetilde\mu_k), 1\rangle_{\widetilde{\mu}_k}|\leq 2^{-k}D_{\widetilde{\mu}_k}(Q_k)\widetilde{\mu}_k(Q_k),
$$
for all $\varphi \in \Lip_0(\tfrac{B_k}{2}Q_k)$ with $\int_{\R^d} \varphi \,d\widetilde{\mu}_k=0$ and $\|\varphi\|_{\Lip}\leq\tfrac{1}{\ell(Q_k)}$.  Additionally, there exists some $R_k\subset \tfrac{B_k}{4}Q_k$, with $a_k^2B_k\ell(Q_k)\leq \ell(R_k)\leq \sqrt{a_k}B_k\ell(Q_k)$ and $D_{\wt{\mu}_k}(R_k)\geq a_k^{\eps_1}D_{\wt{\mu}_k}(Q_k)$.

For each $R_k$ we locate our favourite maximizing cube $R^*_k$ as defined in Section \ref{epsregcubesec}.  Then by Corollary \ref{rescalmeas},
$$|\langle \RSO(\varphi\widetilde\mu_k), 1\rangle_{\widetilde{\mu}_k}|\leq 2^{-k}D_{\widetilde{\mu}_k}(R_k^*)\wt{\mu}_k(R_k^*),
$$
for all $\varphi \in \Lip_0(\tfrac{B_k}{2}Q_k)$ with $\int \varphi \,d\widetilde{\mu}_k=0$ and  $\|\varphi\|_{\Lip}\leq\tfrac{1}{\ell(Q_k)}$.

Set
$$\mu_k(\, \cdot\,) = \frac{\widetilde{\mu}_k(\mathcal{L}_{R^*}(\,\cdot\,))}{\widetilde{\mu}_k(R^*)}.$$

The pre-images of the cubes in $\dy$ under the mapping $\mathcal{L}_{R^*}$ are contained in a lattice $\mathcal{D}_k$ that contains the unit cube $Q_0$.   Consequently, with the aid of the diagonal argument we pass to a subsequence of the measures $\mu_k$ so that the lattices stabilize in the sense that there is a fixed lattice $\mathcal{D'}$ such that every cube $Q'\in \mathcal{D}'$ lies in $\mathcal{D}_k$ for all sufficiently large $k$.

The measure $\mu_k$ satisfies the following properties:
\begin{enumerate}
\item Since $a_k^{-1/8}R_k^*\subset \tfrac{B_k}{2}Q_k$ (see (\ref{aenlargedeepB})), we have that for every $\varphi \in \Lip_0(a_k^{-1/8}Q_0)$ with $\|\varphi\|_{\Lip}<1$ and $\int_{\R^d} \varphi \,d\mu_k=0$,
    $$|\langle \RSO(\varphi\mu_k), 1\rangle_{\mu_k}|\leq 2^{-k}.$$
\item  The estimate $D_{\mu_k}(Q')\leq 2^{\eps_0 d(Q',Q_0)}$ holds for all $Q'\in \dy_k$ satisfying $Q'\cap a_k^{-1/8}Q_0\neq \varnothing$ with $\ell(Q')\geq \tfrac{1}{a_k^{13/4}B_k}\ell(Q_0).$  (See Corollary \ref{epsregscalecont}.)
\item (Uniform Diffuseness.) Each measure $\mu_k$ satisfies \begin{equation}\nonumber\mathbb{E}_{2^{-m}A\ell(Q_0)}^{\mu_k}(AQ_0)\leq C\Bigl[\frac{a_k^{-\beta_2}}{AB_k}+A^{s+2\eps_0}2^{-m(1-\eps_0)}\Bigl]A,
\end{equation}
whenever $A<a_k^{-1/8}$ and $m\geq 0$.  As $\beta>\beta_2$, we have $a_k^{-\beta_2}/B_k\rightarrow 0$ as $k\rightarrow \infty$, and of course $a_k^{-1/8}\rightarrow \infty$ as $k\rightarrow \infty$.  Therefore, we conclude that the measures $\mu_k$ are uniformly diffuse in $\R^d$.
\item $\mu_k(\overline{Q}_0)\geq 1$.
\end{enumerate}

We now derive a contradiction.  From item (2), it follows that there is a constant $C>0$ such that \begin{equation}\label{ballcontblowup1}\sup_k \mu_k(RQ_0)\leq CR^{s+\eps_0}\text{ for any } R>1.\end{equation}  Consequently, we may pass to a subsequence of the measures $\mu_k$ that converges weakly to a measure $\mu$.  The growth condition (\ref{ballcontblowup1}) also implies that the measures have uniformly restricted growth at infinity:  There is a constant $C>0$ such that
$$\int_{\R^d\backslash RQ_0}\frac{1}{|x|^{s+1}}d\mu_k(x)\leq \frac{C}{R^{1-\eps_0}}.
$$
Since the third property tells us that the measures $\mu_k$ are uniformly diffuse in the entire space $\R^d$,  we are permitted to apply Lemma \ref{weakconv} to conclude that $\mu$ is reflectionless in $\R^d$.  We now employ the semi-continuity properties of the weak limit (Lemma \ref{semicontinuity}).  From item (4), we see that $\mu(\overline{Q}_0)\geq 1$.  However, item (2) implies that $\mu$ satisfies $D_{\mu}(Q')\leq 2^{\eps_0d(Q', Q_0)}D_{\mu}(\overline{Q}_0)$ for all $Q'\in \mathcal{D}'$.  We have arrived at a contradiction with Proposition \ref{regularprop}.

\section{The choice of the shell}

Fix a measure $\mu$, and $Q\in \dysel(\mu)$.   Up to this point, all that has been required from the shell $\wh{Q}$ is that $2aBQ\subset \wh{Q}\subset 4aBQ$.  In this section we make the precise choice of the shell cube.

 Corollary \ref{controlcor} ensures that $D_{\mu}(4aB Q)\leq CD_{\mu}(Q)$. Consequently, we have that
 \begin{equation}\begin{split}\label{constantmeasure}\mu(4aBQ)&= (4aB\ell(Q))^sD_{\mu}(4aBQ)\\&\leq Ca^sB^s\ell(Q)^sD_{\mu}(Q)\leq Ca^sB^s\mu(2aBQ).\end{split}\end{equation}

We shall now use (\ref{constantmeasure}) to locate a doubling cube.  Let $\lambda=\tfrac{1}{s\log_2B}$.  Set $Q^{(0)}=2aBQ$, and $Q^{(j)}= (1+j\lambda)Q^{(0)}$ for $1\leq j \leq \tfrac{1}{\lambda}$.  If it holds that $\mu(Q^{(j-1)}) < \tfrac{1}{2}\mu(Q^{(j)})$ for all $1\leq j\leq \tfrac{1}{\lambda}$, then $\mu(4aBQ)>2^{1/\lambda-1}\mu(2aBQ)=\tfrac{1}{2}B^s\mu(2aBQ)$, which is absurd given (\ref{constantmeasure}) and that $a\ll 1$.


Consequently, there exists some $1\leq j\leq \tfrac{1}{\lambda}$ such that $\mu(Q^{(j-1)})\geq \tfrac{1}{2}\mu(Q^{(j)})$.  Set $\wh{Q}^{\mu}$ to be the closure of the cube $(1+(j-\tfrac{1}{2})\lambda)Q^{(0)}$.  Note that
\begin{equation}\label{doubleshell}
\mu((1-\tfrac{\lambda}{8})\wh{Q})\geq \tfrac{1}{2}\mu((1+\tfrac{\lambda}{8})\wh{Q}).
\end{equation}
(Recall that $2\ell(Q^{(0)})\geq \ell(Q^{(j)})$.)

\section{The long and lasting drop in density yields an improved energy estimate in the shell}

Fix a measure $\mu$ and $Q\in \dyselA(\mu)$.  Assume that $D_{\mu}(Q')\leq a^{\eps_1}D_{\mu}(Q)$ for all $Q'$ intersecting $\wh Q$ with $a^2B\ell(Q)\leq \ell(Q')\leq \sqrt{a}B\ell(Q).$  We shall prove the following estimate:

\begin{lem}\label{blowupshellenergy}  There are constants $C>0$ and $\beta_3=\beta_3(d,s)>0$ such that
$$\mathbb{E}^{\mu}((1+\tfrac{\lambda}{8})\wh{Q})\leq C\Bigl(a^{\eps_1}+\frac{1}{a^{\beta_3}B}\Bigl)D_{\mu}(Q)\mu(\wh{Q})\ell(\wh{Q}).
$$
\end{lem}

\begin{proof}  We set $\wt{Q} = (1+\tfrac{\lambda}{8})\wh{Q}$.  Certainly $\wh{Q}\subset \wt{Q}\subset 2\wh{Q}$.  Using the doubling property (\ref{doubleshell}), it suffices to prove the estimate in the lemma with the factor $\mu(\wh{Q})$ on the right hand side replaced by $\mu(\wt{Q})$.  Due to Lemma \ref{dyintdom}, it suffices to estimate $\mathcal{E}^{\mu}(\wt{Q})$.

Recalling that $2aB\ell(Q)\leq \ell(\wh{Q})\leq 4aB\ell(Q)$, we see that the additional property that the density drop provides us with is that every cube $Q'$ intersecting $\wt{Q}$ of sidelength in the range $\tfrac{a}{2}\ell(\wh{Q})\leq \ell(Q')\leq \tfrac{1}{4\sqrt{a}}\ell(\wh{Q})$ satisfies $D_{\mu}(Q')\leq a^{\eps_1}D_{\mu}(Q)$.  In particular, we may estimate the contribution to the energy from these cubes in a straightforward manner:
$$\sum_{\substack{Q'\in \dy\\a\ell(\wt{Q})\leq \ell(Q')\leq \ell(\wt{Q})}}\ell(Q')D_{\mu}(Q')\mu(Q'\cap \wt{Q})\leq Ca^{\eps_1}D_{\mu}(Q)\mu(\wt{Q})\ell(\wt{Q}),
$$
where we have done nothing more than split the sum over the dyadic levels, and use the density estimate $D_{\mu}(Q')\leq a^{\eps_1}D_{\mu}(Q)$ in each dyadic level.

The remainder of the energy is bounded by
$$\mathcal{E}_{\text{small},\frac{\ell(Q)}{a}}(\wt{Q})+ \mathcal{E}_{\text{large}, \, a\ell(\wt{Q})}(\wt{Q}).$$
Appealing to Lemma \ref{smallcubeenergy}, and recalling that $\tfrac{\ell(Q)}{\ell(\wh{Q})}\leq \frac{1}{aB}$, we get that
$$\mathcal{E}_{\text{small},\frac{\ell(Q)}{a}}(\wt{Q})\leq \frac{C}{a^{\max(s,d-s)+1+\eps}}\ell(Q)D_{\mu}(Q)\mu(\wt{Q})\leq \frac{C}{a^{\beta_3}B}\ell(\wh{Q})D_{\mu}(Q)\mu(\wt{Q}),$$
while Lemma \ref{largecubeenergy} yields that
$$\mathcal{E}_{\text{large}, \, a\ell(\wt{Q})}(\wt{Q})\leq Ca\ell(\wt{Q})D_{\mu}(Q)\mu(\wt{Q}),
$$
and the lemma follows.
\end{proof}



\section{Blow up II:  Doing away with $\eps $}\label{epslimit}

Let us fix $B\geq B_0$ sufficiently large, and $a\leq a_0$ sufficiently small, satisfying $a^{-\beta}\ll B$.  The upshot of Proposition \ref{densdropprop} is that in order to achieve our main goal, we now only need to associate a large Lipschitz coefficient to those $Q\in \dyselA$ with the additional property that part (ii) of Proposition \ref{densdropprop} holds.   In this section we shall examine what happens in the case that this fails to occur in the limit as $\eps\rightarrow 0^+$.  The statement that we shall obtain is perhaps a little convoluted, so we present it as a summary at the end of the section.

\textbf{ASSUMPTION.} \emph{ Suppose that for every $\eps>0$ satisfying $B^{\eps}\leq 2$, and for every $\Delta>0$ we can find a measure $\mu$ and $Q\in \dyselA$ satisfying $D_{\mu}(Q')\leq a^{\eps_1}D_{\mu}(Q)$ for all $Q'\in \dy$ with $Q'\cap\tfrac{B}{4}Q\neq \varnothing$ and $a^2B\ell(Q)\leq \ell(Q')\leq \sqrt{a}B\ell(Q),$ so that
$$\Theta^{B/2}_{\mu}(Q)\leq \Delta D_{\mu}(Q)\mu(Q).
$$}

Under this assumption, we can find a sequence $(\eps^{(k)})_k$ that tends to zero, with $B^{\eps^{(k)}}\leq 2$, such that for each $k$, there is a measure $\wt{\mu}_k$ and a cube $Q_k\in \dyselA(\wt{\mu}_k)$ satisfying $D_{\wt{\mu}_k}(Q')\leq a^{\eps_1}D_{\wh{\mu}}(Q_k)$ for all $Q'\in \dy$ with $Q'\cap\tfrac{B}{4}Q_k\neq \varnothing$ and $a^2B\ell(Q_k)\leq \ell(Q')\leq \sqrt{a}B\ell(Q_k)$, and furthermore $$|\langle \RSO(\psi \wt{\mu}_k),1\rangle_{\wt{\mu}_k}|\leq 2^{-k} D_{\wt{\mu}_k}(Q_k)\wt{\mu}_k(Q_k)\text{ for all }\psi \in \Psi_{\wt{\mu}_k}^{B/2}(Q_k).$$

We shall now shift and scale $Q_k$ to be the cube $Q_0$.  Set $$\mu_k(\,\cdot\,) = \frac{\wt{\mu}_k(\mathcal{L}_{Q_k}(\,\cdot\,))}{\wt{\mu}_k(Q_k)},
$$
so that $\mu_k(Q_0)=1$.  As in the previous blow up argument, we have that the pre-images of cubes in $\mathcal{D}$ under the mapping $\mathcal{L}_{Q_k}$ belong to some lattice $\dy_k$ containing $Q_0$.  Passing to a subsequence if necessary, we may assume that these lattices stablize insofar as there is a lattice $\dy'$ such that every $Q'\in \dy'$ lies in $\dy_k$ for sufficiently large $k$.

Under the scaling, we have that
\begin{equation}\label{blowup2deltsmall}
|\langle \RSO(\psi \mu_k),1\rangle_{\mu_k}|\leq 2^{-k},\text{ for all }\psi \in \Psi_{\mu_k}^{B/2}.
\end{equation}

Our first aim will be to show that (by passing to a subsequence if necessary) the measures $\mu_k$ converge weakly to a measure $\mu$ that is reflectionless in $\tfrac{B}{2}Q_0$.

First notice that Lemma \ref{upcontrol} ensures that
\begin{equation}\label{betteratinfty}
D_{\mu_k}(Q')\leq C_12^{\eps^{(k)}[Q':Q_0]}D_{\mu_k}(Q_0)\leq C_12^{\eps^{(k)}[Q':Q_0]},
\end{equation}
for all $Q'\in \dy_k$ that intersect $BQ_0$ and satisfy $\ell(Q')\geq \tfrac{\ell(Q_0)}{a}$, while Corollary \ref{controlcor} ensures that $D_{\mu_k}(Q')\leq \tfrac{2\cdot 4^sC_1}{a^s}$ for any cube $Q'$ with $Q'\cap \tfrac{B}{2}Q\neq \varnothing$ and $\ell(Q_0)\leq \ell(Q')\leq \tfrac{\ell(Q_0)}{a}$.  From these density bounds we readily derive that there is a constant $C(a)>0$ such that for all $R>1$,
$$\int_{\R^d\backslash RQ_0}\frac{1}{|x|^{s+1}}d\mu(x)\leq \frac{C(a)}{R^{1-\eps^{(k)}}}\leq \frac{C(a)}{\sqrt{R}}.
$$
In particular, the sequence $\mu_k$ has uniformly restricted growth at infinity.

Let us now show that the uniform diffuse property holds for the sequence $\mu_k$ in the cube $\tfrac{B}{2}Q_0$.  Since $a$ and $B$ are now fixed, this amounts to showing that $\mathbb{E}_{2^{-m}}^{\mu_k}(\tfrac{B}{2}Q_0)$ tends to zero uniformly in $k$ as $m\rightarrow \infty$.  First notice that
$$\mathbb{E}_{2^{-m}\ell(Q_0)}^{\mu_k}(\tfrac{B}{2}Q_0)=\frac{\ell(Q_0)D_{\mu_k}(Q_0)\mu_k(Q_0)}{\ell(Q_k)D_{\wt{\mu}_k}(Q_k)\wt{\mu}_k(Q_k)}\mathbb{E}^{\wt{\mu}_k}_{2^{-m}\ell(Q_k)}(\tfrac{B}{2}Q_k),
$$
and using Lemma \ref{dyintdom}, we may estimate the right hand side by a constant multiple of
\begin{equation}\begin{split}\nonumber\frac{1}{\ell(Q_k)D_{\wt{\mu}_k}(Q_k)\wt{\mu}_k(Q_k)}&\mathcal{E}^{\wt{\mu}_k}_{2^{-m+3}\ell(Q_k)}(\tfrac{B}{2}Q_k)\\&\leq\frac{CB^{s+1}}{\ell(BQ_k)D_{\wt{\mu}_k}(Q_k)\wt{\mu}_k(BQ_k)}\mathcal{E}^{\wt{\mu}_k}_{2^{-m+3}\ell(Q_k)}(\tfrac{B}{2}Q_k).
\end{split}\end{equation}
However, since $a$ and $B$ are fixed, it is clear that by combining the estimates given in Lemmas \ref{smallcubeenergy} and \ref{largecubeenergy} respectively, the truncated  energy $\mathcal{E}^{\wt{\mu}_k}_{2^{-m+3}\ell(Q_k)}(\tfrac{B}{2}Q_k)$ is bounded by $C(a,B)2^{-m(1-\eps^{(k)})}\ell(Q_k)D_{\wt{\mu}_k}(Q_k)\wt{\mu}_k(\tfrac{B}{2}Q_k)$.  Thus $\mathbb{E}_{2^{-m}\ell(Q_0)}^{\mu_k}(\tfrac{B}{2}Q_0)\leq C(a,B)2^{-m/2}$ and the sequence $\mu_k$ is therefore uniformly diffuse in $\tfrac{B}{2}Q_0$.

Of course, the inequality (\ref{betteratinfty}) ensures that $\sup_k\mu_k(B(0,R))<\infty$ for any $R>0$, and so we may pass to a further subsequence if necessary to ensure that the measures $\mu_k$ converge weakly to a measure $\mu$ with $\mu(\overline{Q}_0)\geq 1$.

Recalling (\ref{blowup2deltsmall}), we may now appeal to Lemma \ref{weakconv} to conclude that $\mu$ is reflectionless in the cube $\tfrac{B}{2}Q_0$.

Finally, the scaled shells $\wh{Q}_0^{\mu_k}=\mathcal{L}_{Q_k}^{-1}(\wh{Q}_k^{\wt{\mu}_k})$ all lie in the compact set $\overline{4aBQ_0}$, and so, passing to a further subsequence if necessary, may be assumed to converge in the Hausdorff metric to a closed cube $\wh{Q}_0$ satisfying $2aBQ_0\subset \wh{Q}_0 \subset  \overline{4aBQ_0}$.

In the next few subsections, we describe the additional properties of $\mu$ that are important for our analysis.


\subsection{Density properties of $\mu$}\label{limitdensityprops}

Since cubes in $\mathcal{D}'$ are open, we may freely employ the lower-semicontinuity of the weak limit to deduce that if $Q'\in \mathcal{D}'$, and $Q'\supset \wh{Q}_0$, then
$$D_{\mu}(Q')\leq C_1.$$
Indeed, the open cube $Q'$ contains any of the compact sets $\wh{Q}_0^{\mu_k}$ for sufficiently large $k$, and also $Q'\in \mathcal{D}_k$ for large enough $k$.  Thus
$$D_{\mu_k}(Q')\leq C_1 2^{\eps^{(k)}[Q':Q_0]},
$$
for large enough $k$ (see (\ref{betteratinfty})).  Since $\eps^{(k)}$ tends to zero, the property follows.

Similarly, any cube $Q'\in\mathcal{D}'$ intersecting $\tfrac{B}{4}Q_0$ and satisfying $a\ell(\wh{Q}_0)< \ell(Q')< \tfrac{1}{4\sqrt{a}}\ell(\wh{Q}_0)$ also lies in $\mathcal{D}_k$, and satisfies $a\ell(\wh{Q}_0^{\mu_k})\leq \ell(Q')\leq \tfrac{1}{4\sqrt{a}}\ell(\wh{Q}_0^{\mu_k})$ for large enough $k$.  By assumption each such cube satisfies
$D_{\mu_k}(Q')\leq a^{\eps_1},
$
and therefore
$$D_{\mu}(Q')\leq a^{\eps_1}.
$$

Finally, suppose that $Q'\in \dy'$ intersects $\tfrac{B}{2}Q_0$ and satisfies $D_{\mu}(Q')>C_1\frac{2\cdot 4^s}{a^s}$.  By the semi-continuity properties of the weak limit, we have for all $k$ sufficiently large that $D_{\mu_k}(Q')> C_1\frac{2\cdot 4^s}{a^s}2^{\eps^{(k)}[Q_0:\,Q']}D_{\mu_k}(Q_0)$.  But then Lemma \ref{upcontrol} and Corollary \ref{controlcor} ensure that $\ell(Q')\leq \tfrac{\ell(Q_0)}{2}$.

\subsection{The doubling property and the energy property of $\wh{Q}_0$}

We first claim that the cube $\wh{Q}_0$ inherits a doubling property from the shells $\wh{Q}_0^{\mu_k}$.

\begin{cla}  $\mu(\wh{Q}_0)\leq 2\mu ((1-\tfrac{\lambda}{8})\wh{Q}_0)$.
\end{cla}
\begin{proof}To prove this claim, let $U$ be the open $\tfrac{\lambda}{16}\ell(\wh{Q}_0)$ neighbourhood of $\wh{Q}_0$.  Then
$$\mu(\wh{Q}_0)\leq \mu(U)\leq \liminf_{k\rightarrow \infty}\mu_k(U),
$$
but $U\subset (1+\tfrac{\lambda}{8})\wh{Q}_0^{\mu_k}$ for large enough $k$.  For those $k$, $\mu_k(U)\leq 2\mu_k((1-\tfrac{\lambda}{8})\wh{Q}_0^{\mu_k})$.  However, since the sequence of closed cubes $\wh{Q}_0^{\mu_k}$ converges to $\wh{Q}_0$ in Hausdorff metric,
\begin{equation}\label{Hausdupper}\mu((1-\tfrac{\lambda}{8})\wh{Q}_0)\geq \limsup_{k\rightarrow \infty}\mu_k((1-\tfrac{\lambda}{8})\wh{Q}_0^{\mu_k}),
\end{equation}
and the claim follows.\end{proof}

Next, we turn to estimating the energy $\mathbb{E}^{\mu}(\wh{Q}_0)$.

\begin{cla}  $$\mathbb{E}^{\mu}(\wh{Q}_0)\leq C\Bigl(a^{\eps_1}+\frac{1}{a^{\beta_3}B}\Bigl)\mu(\wh{Q}_0)\ell(\wh{Q}_0).
$$
\end{cla}

\begin{proof}It is a straightforward exercise to show that the energy is lower-semicontinuous:  If $U\subset \mathbb{R}^d$ is open and bounded, then
$$\mathbb{E}^{\mu}(U)\leq \liminf_{k\rightarrow \infty}\mathbb{E}^{\mu_k}(U).
$$
Now, let $U$ be the open $\tfrac{\lambda}{16}\ell(\wh{Q}_0)$ neighbourhood of $\wh{Q}_0$.   Then for $k$ large enough, $U\subset (1+\tfrac{\lambda}{8})\wh{Q}_0^{\mu_k}$.  But then  Lemma \ref{blowupshellenergy} ensures that
$$\mathbb{E}^{\mu}(U)\leq \liminf_{k\rightarrow \infty}\mathbb{E}^{\mu_k}(U)\leq C\Bigl(a^{\eps_1}+\frac{1}{a^{\beta_3}B}\Bigl)\liminf_{k\rightarrow \infty}\mu_k(\wh{Q}_0^{\mu_k})\ell(\wh{Q}_0^{\mu_k}),
$$
which proves the claim after passing to the limit on the right hand side (cf. (\ref{Hausdupper})).
\end{proof}

\subsection{The Weak-$L^2$ property of the maximal density}  We continue to study the limit measure $\mu$.  Set $$\overline{D}_{\mu}(x) = \sup_{r>0}D_{\mu}(B(x,r)).$$

\begin{lem}  There exists a constant $C=C(B)>0$, such that for any $T>1$
$$
\mu\Bigl(\Bigl\{x\in \frac{B}{4}Q_0: \overline{D}_{\mu}(x)>T\Bigl\}\Bigl)\leq \frac{C(B)}{T^2}.
$$
\end{lem}

\begin{proof}Note that for each  $x\in \R^d$ and $r>0$, there is a cube $Q'\in \mathcal{D}$ containing $B(x,r)$ with $\ell(Q')\leq 8r$.  Thus, it suffices to estimate the measure of the set
$$\Bigl\{x\in \frac{B}{4}Q_0:\sup_{Q'\in \mathcal{D}': \, x\in Q'}D_{\mu}(Q')>T\Bigl\}.
$$
for every $T>1$.
Moreover, since $\mu(\tfrac{B}{4}Q_0)\leq CB^s$, it suffices to estimate the measure of this set for $T>\max(4\cdot 2^sC_1,C_2)a^{-s}$.
Consider the collection maximal (by inclusion) cubes in $\mathcal{D}'$ intersecting $\tfrac{B}{4}Q_0$ of $\mu$-density strictly greater than $T$.  The third of the density properties for $\mu$ proved in Section \ref{limitdensityprops} ensures that every such maximal cube has sidelength no greater than $\ell(Q_0)/2$.  Let $\mathcal{H}_T$ denote the collection of these maximal cubes.  Now fix $m\in \mathbb{N}$, and consider $\mathcal{H}_{T,m}$:  the collection of maximal high density cubes in $\mathcal{H}_T$ of sidelength at least $2^{-m}$.  It is clear that this is a finite collection of cubes.  As such, we can find $k$ sufficiently large so that $Q_0\in \dyselA(\mu_k)$ and $ D_{\mu_k}(Q')>T2^{\eps^{(k)}[Q_0:Q']}$ for every $Q'\in \mathcal{H}_{T,m}$.

Since $Q_0\in \dyselA(\mu_k)$, Lemma \ref{Bcontain} ensures that for each $Q'\in \mathcal{H}_{T,m}$ we can find $Q''\in \dysel(\mu_k)$ with $BQ''\subset BQ_0$, $aBQ''\supset Q'$ and $ D_{\mu_k}(Q'')>T2^{\eps^{(k)}[Q'':Q_0]}$ (if $Q'\in \dysel(\mu_k)$ then set $Q''=Q'$).  Since $Q''$ is not dominated from above, Lemma \ref{upcontrol} applies to yield $D_{\mu_k}(5BQ'')\leq CD_{\mu_k}(Q'')$.  Thus
$$\mu_k(5BQ'')\leq CB^s\mu_k(Q'')\leq CB^s\mu_k(\wh{Q}'').
$$

Applying the Vitali covering lemma, we can find a subcollection $\wt{\mathcal{H}}_{T,m}$ of the $Q''$ with $BQ''$ disjoint, contained in $BQ_0$, and such that $5BQ''$ cover every $Q'\in \mathcal{H}_{T,m}$.  Insofar as $Q_0\in \dyselA(\mu_k)$, and so cannot be dominated from below by a bunch, we have that
\begin{equation}\nonumber\begin{split}T^2\sum_{Q''\in \wt{\mathcal{H}}_{T,m}} \mu_k(5BQ'')&\leq CB^s \sum_{Q''\in \wt{\mathcal{H}}_{T,m}} D_{\mu_k}(Q'')^2 2^{-2\eps^{(k)}[Q_0:Q'']}\mu_k(\wh{Q}'')\\&\leq CB^s\mu_k(\wh{Q}_0^{(k)}).
\end{split}\end{equation}
Letting $k\rightarrow \infty$ and then $m\rightarrow \infty$ completes the proof of the lemma.\end{proof}

\subsection{Summary}  In the following alternative we recap what has been proved so far.

\begin{alt}\label{strangealt}  For every $B\geq B_0$ and $a\leq a_0$ satisfying $a^{-\beta}\ll B$, one of the following two statements holds:

(i) (Large Oscillation coefficient.) There exist $\eps >0$ satisfying $B^{\eps }\leq 2$, and $\Delta>0$,  such that for every measure $\mu$ and $Q\in \dyselA(\mu)$ we have that $\mu$ is diffuse in $\tfrac{B}{2}Q$ and $$\Theta_{\mu}^{B/2}(Q)\geq \Delta D_{\mu}(Q)\mu(Q).$$

(ii) (Existence of a strange reflectionless measure.)  There is a measure $\mu$, and a cube $\wh{Q}_0$ with $2aBQ_0\subset \wh{Q}_0 \subset \overline{4aBQ_0}$, satisfying the following properties:
\begin{enumerate}
\item $\mu$ is reflectionless in $\tfrac{B}{2}Q\supset \wh{Q}_0$;
\item $\mu(\overline{Q}_0)\geq 1$;
\item (Doubling in $\wh{Q}_0$) $\mu(\wh{Q}_0)\leq 2\mu((1-\tfrac{\lambda}{8}) \wh{Q}_0)$, with $\lambda= \tfrac{1}{s\log_2B};$
\item (Small energy in $\wh{Q}_0$)
$$\mathbb{E}^{\mu}(\wh{Q}_0)\leq C\Bigl(a^{\eps_1}+\frac{1}{a^{\beta_3}B}\Bigl)\mu(\wh{Q}_0)\ell(\wh{Q}_0); $$
\item $D_{\mu}(Q')\leq C$ for all $Q'\in \mathcal{D}'$ satisfying $Q'\cap \wh{Q}_0\neq \varnothing$ with $\ell(Q')\geq \tfrac{1}{a}\ell(Q_0)$;
\item $D_{\mu}(Q')\leq a^{\eps_1}$ for all $Q'\in \mathcal{D}'$ that intersect $\tfrac{B}{4}Q_0$ and have sidelength $a\ell(\wh{Q}_0)<\ell(Q')<\tfrac{1}{4\sqrt{a}}\ell(\wh{Q}_0)$;
\item For every $T>1$, $$
\mu\Bigl(\Bigl\{x\in \wh{Q}_0: \overline{D}_{\mu}(x)>T\Bigl\}\Bigl)\leq \frac{C(B)}{T^2}.
$$
\end{enumerate}
\end{alt}

\section{Localization around the shell}\label{localize}


Consider the reflectionless measure $\mu$ given in part (ii) of Alternative \ref{strangealt}.  The goal of this section is to localise the measure $\mu$ to the shell $\wh{Q}_0$.  Restricting the measure in this way will of course distort the reflectionless property.  The following calculation shows that this distortion is well controlled using the energy property (4) and the density properties (5)--(6).

First recall that for any $x\in \R^d$, $r>0$ there is a cube $Q_{x,r}\in \mathcal{D}'$ such that $B(x,r)\subset Q_{x,r}$ and $\ell(Q_{x,r})<8r$.  Thus, if $x\in \wh{Q}_0$ and $r\in (a\ell(\wh{Q}_0), \tfrac{1}{32\sqrt{a}}\ell(\wh{Q}_0))$, then $\tfrac{\mu(B(x,r))}{r^s}\leq Ca^{\eps_1}$ (by property (6)), while if $r>\tfrac{1}{32\sqrt{a}}\ell(\wh{Q}_0)\geq \tfrac{1}{a}\ell(Q_0)$, then $\tfrac{\mu(B(x,r))}{r^s}\leq C$ (by property (5)).

Fix $\varphi\in \Lip_0((1-\tfrac{\lambda}{16})\wh{Q}_0)$ so that $\varphi \equiv 1$ on $(1-\tfrac{\lambda}{8})\wh{Q}_0$, $0\leq \varphi\leq 1$ on $\R^d$, and $\|\varphi\|_{\Lip}\leq\tfrac{C}{\lambda\ell(\wh{Q}_0)}\leq \tfrac{C\log B}{\ell(\wh{Q}_0)}.$  Also take $\wt\psi\in \Lip_0(\R^d)$, satisfying $\int \wt\psi\varphi d\mu=0$.

By the reflectionless property, $\langle \RSO(\wt\psi\varphi\mu), 1\rangle_{\mu}=0$.  
Let's now split $ \langle \RSO(\wt\psi\varphi\mu), 1\rangle_{\mu}$ into its local and non-local parts:
$$\langle \RSO(\wt\psi\varphi\mu), \chi_{\wh{Q}_0}\rangle_{\mu}+ \langle \RSO(\wt\psi\varphi\mu), \chi_{\mathbb{R}^d\backslash \wh{Q}_0}\rangle_{\mu}.
$$
The local term is interpreted as the Lebesgue integral
$$\langle \RSO(\wt\psi\varphi\mu), \chi_{\wh{Q}_0}\rangle_{\mu} = \frac{1}{2}\iint\limits_{\wh{Q}_0\times \wh{Q}_0}K(x-y)[\wt\psi(y)\varphi(y)-\wt\psi(x)\varphi(x)]d\mu(x)d\mu(y).
$$
Since $\mu$ has restricted growth at infinity, the non-local term can be expressed as $$\langle \RSO(\wt\psi\varphi\mu), \chi_{\mathbb{R}^d\backslash \wh{Q}_0}\rangle_{\mu}=\lim_{N\to\infty}\langle \RSO(\wt\psi\varphi\mu), \chi_{B(0,N)\backslash \wh{Q}_0}\rangle_{\mu}.$$
On the other hand, since $\supp(\varphi)\subset (1-\tfrac{\lambda}{16})\wh{Q}_0$, we have that for fixed $N$,
\begin{equation}\begin{split}\nonumber|\langle \RSO(\wt\psi\varphi\mu), &\chi_{B(0,N)\backslash \wh{Q}_0}\rangle_{\mu}|\leq \text{osc}_{(1-\tfrac{\lambda}{16})\wh{Q}_0}[\RSO(\chi_{B(0,N)\backslash \wh{Q}_0})]\|\wt\psi \varphi\|_{L^1(\mu)}\\
&\leq \sqrt{d}\ell(\wh{Q}_0)\|\nabla\RSO(\chi_{B(0,N)\backslash \wh{Q}_0})\|_{L^{\infty}((1-\tfrac{\lambda}{16})\wh{Q}_0)}\|\wt\psi \varphi\|_{L^1(\mu)}.
\end{split}\end{equation}

Note that for each $x\in (1-\tfrac{\lambda}{16})\wh{Q}_0$,
\begin{equation}\begin{split}|\nabla \RSO(\chi_{B(0,N)\backslash \wh{Q}_0}\mu)(x)|&\leq C\int_{\mathbb{R}^d\backslash \wh{Q}_0}\frac{1}{|x-y|^{s+1}}d\mu(y)\\&\leq C\int\limits_{\tfrac{\lambda}{16}\ell(\wh Q_0)\leq |x-y|\leq \tfrac{1}{32\sqrt{a}}\ell(\wh{Q}_0)}\cdots \,d\mu(y)\\
 &+ C\int\limits_{|x-y|\geq \,\tfrac{1}{32\sqrt{a}}\ell(\wh{Q}_0)}\cdots\, d\mu(y)= I+II.
\end{split}\end{equation}
As long as $\tfrac{1}{a}\gg \log B$ we have $\tfrac{\lambda}{16}>a$, and so $\tfrac{\mu(B(x,r))}{r^s}\leq Ca^{\eps_1}$ for any $x\in (1-\tfrac{\lambda}{16})\wh{Q}_0$ and $r\in (\tfrac{\lambda}{16}\ell(\wh{Q}_0), \tfrac{1}{32\sqrt{a}}\ell(\wh{Q}_0))$.   The first integral $I$ is therefore at most
$$C\int_{\tfrac{\lambda}{16}\ell(\wh{Q}_0)}^{\tfrac{1}{32\sqrt{a}}\ell(\wh{Q}_0)}\frac{\mu(B(x,r))}{r^{s+1}}\frac{dr}{r}\leq \frac{Ca^{\eps_1}\log B}{\ell(\wh{Q}_0)} \text{ on }(1-\tfrac{\lambda}{16})\wh{Q}_0.
$$
The second integral is much smaller: Since $\frac{\mu(B(x,r))}{r^s}\leq C$ for any $x\in \wh{Q}_0$ and $r>2\sqrt{d}\ell(\wh{Q}_0)$, this integral is estimated by
$$II\leq C\int_{\tfrac{1}{32\sqrt{a}}\ell(\wh{Q}_0)}^{\infty}\frac{\mu(B(x,r))}{r^{s+1}}\frac{dr}{r}\leq C \frac{\sqrt{a}}{\ell(\wh{Q}_0)} \text{ on }(1-\tfrac{\lambda}{16})\wh{Q}_0.
$$
We arrive at the following estimate for the nonlocal term
$$|\langle \RSO(\wt\psi\varphi\mu), \chi_{\mathbb{R}^d\backslash \wh{Q}_0}\rangle_{\mu}|\leq Ca^{\eps_1}\log B\|\wt\psi\varphi\|_{L^1(\mu)}.
$$

Set $\mu_0=\chi_{\wh{Q}_0}\mu$, and define
$$U(x) = \int_{\wh{Q}_0}K(x-y)(\varphi(x)-\varphi(y))d\mu_0(y).
$$
Since $\mu$ is diffuse in $\wh{Q}_0$, $U$ lies in $L^1(\mu_0)$.  Thus $\langle U, \psi\rangle_{\mu_0}$ is well defined for any $\psi\in \Lip_0(\R^d)$.  Consequently, by the anti-symmetry of the kernel $K$, we have that
$$\langle \RSO(\varphi \mu), \psi\rangle_{\mu_0} + \langle \RSO(\psi\varphi\mu), \chi_{\wh{Q}_0}\rangle_{\mu}=-\langle U, \psi\rangle_{\mu_0} \text{ whenever }\psi\in \Lip_0(\R^d).
$$
To see this, we write out the left hand side
\begin{equation}\begin{split}\nonumber\iint_{\wh{Q}_0\times\wh{Q}_0}&K(x-y)\frac{1}{2}\bigl[\varphi(y)\psi(x)- \varphi(x)\psi(y)\bigl]d\mu(x)d\mu(y)\\
&+\iint_{\wh{Q}_0\times\wh{Q}_0}K(x-y)\frac{1}{2}\bigl[\varphi(y)\psi(y) - \varphi(x)\psi(x)\bigl]d\mu(x)d\mu(y),
\end{split}\end{equation}
and by combining the integrals and grouping together terms with a common $\psi$ variable, we get
\begin{equation}\begin{split}\nonumber\iint_{\wh{Q}_0\times\wh{Q}_0}&K(x-y)\frac{1}{2}\bigl[\varphi(y)-\varphi(x)\big]\psi(x)d\mu(x)d\mu(y)\\
&+\iint_{\wh{Q}_0\times\wh{Q}_0}K(x-y)\frac{1}{2}\bigl[\varphi(y)- \varphi(x)\bigl]\psi(y)d\mu(x)d\mu(y),
\end{split}\end{equation}
but this equals $-\langle U, \psi\rangle_{\mu_0}$.

Now take any $\psi\in \Lip_0(\R^d)$.  Then the Lipschitz continuous function $$\widetilde\psi(x) = \psi(x) - \frac{1}{\|\varphi\|_{L^2(\mu)}^2}\int_{\R^d}\psi\varphi\, d\mu \cdot \varphi(x)$$ has $\varphi \,d\mu$-mean zero, and its $L^1(\varphi \mu)$ norm is at most $2\int|\psi|\varphi d\mu$.  Thus
\begin{equation}\begin{split}\nonumber |\langle \RSO(\varphi\mu_0), \wt{\psi}\rangle_{\mu_0} + \langle U, \wt{\psi}\rangle_{\mu_0}|&=|\langle \RSO(\wt{\psi}\varphi\mu), \chi_{\wh{Q}_0}\rangle_{\mu}|=|\langle \RSO(\wt{\psi}\varphi\mu), \chi_{\R^d\backslash \wh{Q}_0}\rangle_{\mu}|\\
&\leq Ca^{\eps_1}\log B\|\psi\varphi\|_{L^1(\mu)}.
\end{split}\end{equation}


From the anti-symmetry of the kernel, we see that $\langle \RSO(\varphi \mu_0),\psi\rangle_{\mu_0} = \langle \RSO(\varphi\mu_0),\widetilde{\psi}\rangle_{\mu_0}$, and therefore
\begin{equation}\begin{split}\nonumber|\langle &\RSO(\varphi\mu_0),\psi\rangle_{\mu_0} + \langle U, \psi\rangle_{\mu_0}|\\&\leq \Bigl[\frac{1}{\|\varphi\|_{L^2(\mu)}^2}\int_{\R^d} |U| \varphi\,d\mu(x) + Ca^{\eps_1}\log B\Bigl]\|\psi\varphi\|_{L^1(\mu)}.
\end{split}\end{equation}

We now wish to estimate $\tfrac{1}{\|\varphi\|_{L^2(\mu)}^2}\int_{\R^d} |U| \varphi \,d\mu(x)$.  First note that $$|U(x)|\leq \|\varphi\|_{\Lip}\int_{\wh{Q}_0}\frac{1}{|x-y|^{s-1}}d\mu(y) \leq \frac{C\log B}{\ell(\wh{Q}_0)}\int_{\wh{Q}_0}\frac{1}{|x-y|^{s-1}}d\mu(y).$$
One the other hand, the doubling property of $\wh{Q}_0$ ensures that $\tfrac{1}{2}\mu(\wh{Q}_0)\leq \|\varphi\|_{L^2(\mu)}^2\leq \mu(\wh{Q}_0)$.  Combining these two observations yields that
$$\frac{1}{\|\varphi\|_{L^2(\mu)}^2}\int_{\wh{Q}_0} |U(x)| \varphi(x)d\mu(x) \leq \frac{C\log B}{\mu(\wh{Q}_0)\ell(\wh{Q}_0)}\iint\limits_{\wh{Q}_0\times \wh{Q}_0}\frac{1}{|x-y|^{s-1}}d\mu(x)d\mu(y).
$$
It's now time to use property (4) of the measure $\mu$ given in Alternative \ref{strangealt}.  This bound on the energy $\mathbb{E}^{\mu}(\wh{Q}_0)$ yields that
$$\frac{1}{\|\varphi\|_{L^2(\mu)}^2}\int_{\wh{Q}_0} |U(x)| \varphi(x)d\mu(x) \leq C \Bigl(a^{\eps_1}+ \frac{1}{a^{\beta_3}B}\Bigl)\log B.$$

Let's now define
$$\wt{U}(x) = \int_{\wh{Q}_0\backslash \tfrac{1}{2}\wh{Q}_0}(\varphi(x)-\varphi(y))K(x-y)d\mu(y).
$$
Then for every $x\in \R^d$,
\begin{equation}\label{innerpartnothing}|\wt{U}(x)-U(x)|\leq Ca^{\eps_1}.
\end{equation}
Indeed, if $x\in \tfrac{3}{4}\wh{Q}_0$, then for every $y\in \tfrac{1}{2}\wh{Q}_0$, $\varphi(x)=\varphi(y)$, and so $\wt{U}(x)=U(x)$.
Otherwise, $x\not\in \tfrac{3}{4}\wh{Q}_0$, but then
$$\int_{\tfrac{1}{2}\wh{Q}_0}|K(x-y)|d\mu(y)\leq \frac{C\mu(\wh{Q}_0)}{\ell(\wh{Q}_0)^s}\leq Ca^{\eps_1}.
$$

We have proved the following statement:  \emph{For every $\psi\in \Lip_0(\R^d)$,}
$$
|\langle \RSO(\varphi\mu_0), \psi\rangle_{\mu_0} + \langle \wt{U}, \psi\rangle_{\mu_0}|\leq C\Bigl[a^{\eps_1}\log B+\frac{a^{-\beta_3}\log B}{B}\Bigl]\|\psi\|_{L^1(\mu_0)}.
$$

We have one more observation to make.  Notice that $$|\wt{U}(x)|\leq C\int_{\wh{Q}_0\backslash \tfrac{1}{2}\wh{Q}_0}\frac{(\log B)}{\ell(\wh{Q}_0)|x-y|^{s-1}}d\mu_0(y).$$
As the Riesz potential of a positive measure, the quantity on the right hand side of this inequality is $\alpha$-superharmonic ($\alpha = \tfrac{d-s+1}{2}$), and this will be used crucially in Part II of the paper.

Finally, we have arrived at the following conclusion:

\begin{lem}For every $\psi\in \Lip_0(\R^d)$,
\begin{equation}\begin{split}\nonumber
|\langle \RSO(\varphi\mu_0), \psi\rangle_{\mu_0}| \leq  &C\int_{\wh{Q}_0}\Bigl[\int_{\wh{Q}_0\backslash \tfrac{1}{2}\wh{Q}_0}\frac{\log B}{\ell(\wh{Q}_0)|x-y|^{s-1}}d\mu_0(y)\Bigl]|\psi(x)|d\mu_0(x)\\& + C\tau_{a,B}\|\psi\|_{L^1(\mu_0)},
\end{split}\end{equation}
where $\tau_{a,B} = a^{\eps_1}\log B+\frac{a^{-\beta_3}\log B}{B}$.\end{lem}

\subsection{Fixing parameters and the end of Part I}

It's time now to fix $a$ in terms of $B$.  Put $$a = B^{-\tfrac{1}{2\beta'}},$$ with $\beta' = \max \bigl(\beta, \beta_3\bigl)$.

By choosing $B$ sufficiently large, we may ensure that $B\geq B_0$, $a\leq a_0$, and $a^{-\beta}\ll B$.  Now set $A=\tfrac{\ell(\wh{Q}_0)}{\ell(Q_0)}$, so $2aB\leq A\leq 4aB$.  Then if $B$ is sufficiently large we have $\log(B)\leq 4\log(A)$.

By the choice of $\beta'$, we have that $\tau_{a,B}\leq CA^{-\gamma}$, and $D_{\mu}(\wh{Q}_0)\leq CA^{-\gamma}$ for some $\gamma>0$ depending on $d$ and $s$. By relabelling $\mu_0$ by $\mu$ we arrive at the following statement:


\begin{alt}\label{restate}  Either

(i) Theorem \ref{thm} holds, or

(ii) there exist $C_4>0$ and $\gamma>0$, depending solely on $d$ and $s$, such that for arbitrarily large $A>0$, there is a finite diffuse measure $\mu$ supported on the cube $AQ_0$, and a function $\varphi\in \Lip_0(AQ_0)$ satisfying

(a) $\varphi\equiv 1$ on $\tfrac{9A}{10}Q_0 \text{ and }0\leq \varphi\leq 1 \text{ in }\R^d$,

(b) $\mu(\overline{Q}_0)\geq 1$,

(c) $D_{\mu}(AQ_0)\leq CA^{-\gamma}$,

(d) $\overline{D}_{\mu}\in L^{2,\infty}(\mu)$ where $\overline{D}_{\mu}(x) = \sup_{r>0}\frac{\mu(B(x,r))}{r^s}.$ In other words, $$\|\overline{D}_{\mu}\|^2_{2,\infty} = \sup_{T>0} \Bigl[ T^2\mu\bigl(\bigl\{x\in \R^d: \overline{D}_{\mu}(x)>T\bigl\}\bigl)\Bigl]<\infty,$$

(e) for every $\psi\in \Lip_0(\R^d)$,
\begin{equation}\begin{split}\nonumber
|\langle \RSO(\varphi\mu), \psi\rangle_{\mu}| \leq   C_4&\log A\int_{AQ_0}G_A(\chi_{AQ_0\backslash \tfrac{A}{2}Q_0}\mu)(x)|\psi(x)|d\mu(x)\\& + C_4A^{-\gamma}\|\psi\|_{L^1(\mu)},
\end{split}\end{equation}
where, for a measure $\nu$,
$$G_A(\nu)(x) = \int_{\R^d}\frac{1}{A|x-y|^{s-1}}d\nu(y).
$$
\end{alt}

\part*{Part II:  The non-existence of an impossible object}

\section{The scheme}\label{pt2scheme}

Our goal is to obtain a contradiction by assuming that part (ii) of the Alternative \ref{restate} holds.   The argument that we shall employ to obtain this contradiction is based on the ideas introduced by Eiderman-Nazarov-Volberg \cite{ENV} and Reguera-Tolsa \cite{RT}, and is quite involved.  Therefore, we shall here attempt to outline the scheme that shows that part (ii) of the alternative cannot be true.

For $A$ very large, let's suppose that we can construct finite measures $\nu$ and $\mu$ both supported in $AQ_0$ which satisfy
\begin{equation}\label{ptwiseineq}|\RSO (\nu)|\leq (\log A) G_A(\chi_{AQ_0\backslash \tfrac{A}{2}Q_0}\mu)+C_4A^{-\gamma}\,\quad m_d\text{-a.e. in }\R^d,
\end{equation}
but also such that $\nu(\overline{Q}_0)\geq 1$ and $D_{\mu}(AQ_0)\leq A^{-\gamma}$.  A standard Fourier analytic construction (see Section \ref{Fourier}) provides us with a non-negative function $\Psi\in L^1(m_d)$ such that $\Psi(x)\leq \tfrac{C}{(1+|x|)^{2d-s}}$ for every $x\in \R^d$ and
$$\int_{\R^d}|\RSO(\nu)|\Psi dm_d\geq \nu(\overline{Q}_0)\geq 1.
$$
On the other hand, since $D_{\mu}(AQ_0)\leq A^{-\gamma}$, simple estimates yield that
$$\int_{\R^d}G(\chi_{AQ_0\backslash \tfrac{A}{2}Q_0}\mu)\Psi dm_d \leq CA^{-\gamma}.
$$
But then if we integrate (\ref{ptwiseineq}) against $\Psi m_d$ we get that $1\leq C(\log A) A^{-\gamma}$, and this of course yields a contradiction if $A$ was chosen large enough.

With this simple argument in mind, it is natural to attempt to derive a pointwise condition similar to (\ref{ptwiseineq}), with $\nu=\varphi\mu$, from the distributional inequality (e).  Hopefully this should remind the reader of a maximum principle, since the distributional inequality tells us about the behaviour of $\RSO(\varphi\mu)$ only on the support of $\mu$.  However, it is not feasible to derive (\ref{ptwiseineq}) from (e) directly.  We shall instead go through several steps of regularizing and modifying the measures $\varphi\mu$ and $\mu$ while preserving their key properties, and ultimately arrive at some pair of measures for which (a slight variant of) the above Fourier analytic argument can be pushed through.  We shall seek to explain the ideas behind these regularization steps in the subsequent few paragraphs. 

The ultimate goal of the first two regularization steps (Sections \ref{CZreg} and \ref{Smoothing}) is to smooth the measures $\mu$ and $\varphi\mu$.  However, if we were to just convolve these measures with some smooth mollifier immediately, we would not know the effect it would have on the crucial condition (e).  

If  we knew that the operator $\RSO_{\varphi\mu}:L^2(\varphi\mu)\rightarrow L^2(\varphi\mu)$ was bounded, then we could apply theorems of Vihtil\"{a} \cite{Vih} and \cite{ENV} to derive that $\varphi\mu$ has \emph{zero density}, that is,  $$\lim_{r\rightarrow 0}D_{\varphi\mu}(B(x,r))=0 \text{ for }\varphi\mu\text{-almost every }x\in \R^d.$$  The ideas in the paper \cite{ENV} indicate that under this zero density condition, one can perform a smoothing operation on the measures $\mu$ and $\varphi\mu$, while distorting the condition (e) an arbitrarily small amount.  Notice also that the condition (e) indicates that a $T(1)$-theorem may be applicable to obtain the operator boundedness of $\RSO_{\varphi\mu}$.  The obstacle behind applying a $T(1)$-theorem directly is that the measure $\mu$ does not have bounded density, but instead we only have $\overline{D}_{\mu}\in L^{2,\infty}(\mu)$.

We therefore introduce an exceptional set $\Omega$ outside of which the maximal density $\overline{D}_{\mu}$ is bounded above by some massive threshold $T\gg A$. The fancy $T(1)$-theorem of Nazarov-Treil-Volberg \cite{NTV2} yields that if we suppress the Riesz kernel with the distance function $\Phi(x) = \dist(x,\R^d\backslash \Omega)$, then the resulting suppressed Riesz transform operator $\RSO_{\Phi}$ is bounded in $L^2(\varphi\mu)$ with operator norm at most $CT$.  (See Lemma \ref{omegacotlar}.)  In particular, if we introduce the measure $\mu' = \chi_{\R^d\backslash \Omega}\varphi\mu$, then $\RSO_{\mu'}$ is bounded in $L^2(\mu')$.  Thus, the measure $\mu'$ has the zero density condition.  (See Section \ref{densitytheorems}.)

Using the condition (e), the boundedness of $\RSO_{\mu'}$ in $L^2(\mu')$ enables us to conclude that the $L^2(\mu')$ function $$H=(|\RSO_{\mu'}(1)|-C_4(\log A) G_{A}(\chi_{AQ_0\backslash \tfrac{A}{2}Q_0}\mu)-C_4A^{-\gamma})_+$$ has $L^2(\mu')$ norm at most $\|\RSO_{\Phi, \varphi\mu}(\chi_{\Omega})\|_{L^2(\varphi\mu)}$. (See Claim \ref{restrictB} and the discussion following it.)  But, since $\overline{D}_{\mu}\in L^{2,\infty}(\mu)$, the natural estimate for the measure of $\Omega$ (where the maximal density is larger than $T$) is $\mu(\Omega)\leq CT^{-2}$ (Lemma \ref{omegameas}), and so we can only conclude that $\|\RSO_{\Phi, \varphi\mu}(\chi_{\Omega})\|_{L^2(\varphi \mu)}\leq C$ for some constant $C$.

Only knowing that $H$ is at most constant size in $L^2(\mu')$ is insufficient for us to be able to derive a contradiction.   \emph{However}, non-homogeneous Calder\'{o}n-Zygmund theory also yields that if $1<p<2$, then $\|\RSO_{\Phi, \varphi\mu}(\chi_{\Omega})\|^p_{L^p(\varphi\mu)}\leq C(p)T^{-(2-p)}$, which is arbitrarily small.  Thus, we have that $H$ has arbitrarily small $L^p(\mu')$ norm!  (See Section \ref{omegacotlar}.) The idea to work in $L^p$ is one of the main technical innovations of the aforementioned paper \cite{RT}.

We then carry out the smoothing operation in Section \ref{Smoothing}.  This ultimately provides us (after some rescalings) with measures $\wt\mu$ and $\wt\nu$ that have $C^{\infty}$ densities with respect to $m_d$, and satisfy

$\bullet$  $\supp(\wt\nu)\subset AQ_0$ and $D_{\wt\nu}(\overline{Q}_0)\geq \tfrac{1}{2}$,

$\bullet$  $\supp(\wt\mu)\subset AQ_0\backslash \tfrac{A}{8}Q_0$ and $D_{\wt\mu}(AQ_0)\leq C(\log A)A^{-\gamma}$,

$\bullet$  the function $(|\RSO(\wt\nu)|- G_{A}(\wt\mu)-C_4A^{-\gamma})_+$ has arbitrarily small $L^p(\wt\nu)$ norm.

In Section \ref{rearrange} we perform the third modification, analogous to the Eiderman-Nazarov-Volberg variational construction.  We show that by minimizing a suitable functional, one can redistribute the Lebesgue density of $\wt\nu$ on its support to arrive at a measure $\wt\nu_a$ with bounded density with respect to $m_d$,  $\wt\nu_a(\overline{Q}_0)\geq \tfrac{1}{3}$, and such that, if  $$H_a = (|\RSO(\wt\nu_a)|- G_{A}(\wt\mu)-C_4A^{-\gamma})_+,$$ then the positive part of the expression $H_a^p  + p\RSO^*(H_a^{p-1}E\wt\nu_a)$ is \emph{pointwise} arbitrarily small on the support of $\wt\nu_a$ for some measurable unit vector field $E$.  The all important maximum principle for the fractional Laplacian permits the extension of this inequality to the entire space.

Finally, in Section \ref{contradiction}, we appeal to the simple argument that began this section to above to show that the function $H_a^p  + p\RSO^*(H_a^{p-1}E\wt\nu_a)$ cannot have small positive part in the entire space under the assumptions that $D_{\wt\nu_a}(\overline{Q}_0)\geq \tfrac{1}{3}$, and $D_{\wt\mu}(A\overline{Q}_0)\leq CA^{-\gamma}$, if $A$ is sufficiently large.  This will ensure that part (i) of Alternative \ref{restate} holds.

\section{Suppressed kernels}

\subsection{The suppressed kernel} For a $1$-Lipschitz function $\Phi$, define the suppressed Riesz kernel
$$K_{\Phi}(x,y) = \frac{x-y}{(|x-y|^2+\Phi(x)\Phi(y))^{\tfrac{s+1}{2}}}.
$$
Now set $\Phi_{\delta}(x) = \max(\Phi(x),\delta)$.  In this case we write $K_{\Phi,\delta}(x,y)$ instead of $K_{\Phi_{\delta}}(x,y)$.

Notice that if $\nu$ is a finite measure, and $x\in \R^d$ is such that $\Phi(x)>0$, then we may define the potential
$$\RSO_{\Phi}(\nu)(x) = \int_{\R^d}K_{\Phi}(x,y)d\nu(y).
$$
Moreover, if $\delta>0$, then the potential
$$\RSO_{\Phi, \delta}(\nu)(x) =  \int_{\R^d}K_{\Phi, \delta}(x,y)d\nu(y)
$$
is a continuous function on $\R^d$.

\begin{lem}\label{supkernelest}There is a constant $C>0$ such that for every $x,y\in \R^d$,
$$|K_{\Phi}(x,y)| \leq C\min\Bigl(\frac{1}{\Phi(x)^s},\frac{1}{\Phi(y)^s}, \frac{1}{|x-y|^s}\Bigl)$$
\end{lem}

\begin{proof}  Fix $x,y\in \R^d$.  The estimate
$$|K_{\Phi}(x,y)| \leq \frac{1}{|x-y|^{s}}
$$
is trivial.  In lieu of this, and the antisymmetry of $K_{\Phi}$, it suffices to prove the assertion that
$$|K_{\Phi}(x,y)| \leq \frac{C}{\Phi(x)^s}
$$
under the assumption that $|x-y|<\tfrac{1}{2}\Phi(x)$.  But then $\Phi$, as a $1$-Lipschitz function, satisfies $\Phi(y)\geq \Phi(x) - |x-y|\geq \tfrac{1}{2}\Phi(x)$.  Consequently,
$$|K_{\Phi}(x,y)| \leq \frac{2^{(s+1)/2}|x-y|}{\Phi(x)^{s+1}}\leq \frac{2^{(s-1)/2}}{\Phi(x)^s}.
$$
The lemma is proved.
\end{proof}

For a diffuse measure $\nu$, and for $f,\psi \in \Lip_0(\R^d)$ we set $\langle \RSO_{\Phi}(f\nu), \psi\rangle_{\nu}$ to be the bilinear form (\ref{bilinT}) with kernel $k=K_{\Phi}$.

\begin{lem}\label{supkernelcomp} There is a constant $C>0$ such that if $x,y\in \R^d$ with $x\neq y$, then
$$|K(x-y)-K_{\Phi}(x,y)|\leq \frac{C\Phi(x)}{|x-y|^{s+1}}.$$
\end{lem}

\begin{proof}
The claim is only non-trivial if $\Phi(x)>0$, so let us assume that this is the case.  The desired estimate follows immediately from Lemma \ref{supkernelest} if $|x-y|\leq \Phi(x)$.  For $x,y$ with $|x-y|>\Phi(x)$, we also have that $\Phi(y)\leq |x-y|+\Phi(x)\leq 2|x-y|$.  It follows that the quantity $|K(x-y)-K_{\Phi}(x,y)|$, which equals $$\frac{|x-y|\bigl[\bigl(|x-y|^2+\Phi(x)\Phi(y))^{(s+1)/2}-|x-y|^{s+1}\bigl]}{|x-y|^{s+1}\bigl(|x-y|^2+\Phi(x)\Phi(y)\bigl)^{(s+1)/2}},
$$
is at most $$\frac{1}{|x-y|^s}\Bigl[\Bigl(1+2\frac{\Phi(x)}{|x-y|}\Bigl)^{(s+1)/2}-1\Bigl]\leq \frac{C\Phi(x)}{|x-y|^{s+1}},$$
as required.
\end{proof}

The kernel bound of the previous lemma will be used to prove the following comparison result.

\begin{lem}\label{fancyT1comp}  There is a constant $C>0$ such that if $\nu$ is a finite measure, and $\Phi(x)>0$, then
$$\Bigl|\RSO_{\Phi}(\nu)(x)- \int_{|x-y|>\Phi(x)}K(x-y)d\nu(y)\Bigl|\leq C\sup_{r\geq \Phi(x)}D_{\nu}(B(x,r)).
$$
\end{lem}

\begin{proof}
First notice that, by Lemma \ref{supkernelest},
$$\int_{B(x,\Phi(x))}\!\!\!|K_{\Phi}(x,y)|d\nu(y)\leq \int_{B(x,\Phi(x))}\!\!\frac{C}{\Phi(x)^s}d\nu(y)\leq CD_{\nu}(B(x,\Phi(x))).
$$
We shall now consider
$$\int_{|x-y|>\Phi(x)}|K(x-y)-K_{\Phi}(x,y)|d\nu(y).
$$
Applying the estimate in Lemma \ref{supkernelcomp} yields that this integral is bounded by
$$C\int_{|x-y|> \Phi(x)}\frac{\Phi(x)}{|x-y|^{s+1}}d\nu(y) \leq C\Phi(x)\int_{r\geq \Phi(x)}\frac{\nu(B(x,r))}{r^s}\frac{dr}{r^2}.
$$
The right hand side of the previous inequality is clearly dominated by $C\sup_{r\geq \Phi(x)}D_{\nu}(B(x,r))$, and this completes the proof.
\end{proof}

\subsection{The non-homogeneous $\TSO(1)$-theorem for suppressed kernels}

The main result about suppressed kernels that we shall use is the Nazarov-Treil-Volberg $\TSO(1)$-theorem, see \cite{NTV2} and also Chapter 5 of \cite{Tol}.

\begin{thm}\label{biggun}  Let $\mu$ be a finite measure.  Suppose that $\Omega$ is an open set, and put $\Phi(x) = \dist(x,\R^d\backslash \Omega)$.  Assume that
\begin{enumerate}
\item $D_{\mu}(B(x,r))\leq 1$ whenever $r\geq \Phi(x)$, and
\item $\sup_{\delta>0}|\RSO_{\Phi, \delta}(\mu)(x)|\leq 1$ for every $x\in \R^d$.
\end{enumerate}
Then there is a constant $C>0$ such that for every $f\in L^2(\mu)$,
$$\sup_{\delta>0}\int_{\R^d}|\RSO_{\Phi,\delta}(f\mu)|^2d\mu(x)\leq C\|f\|^2_{L^2(\mu)}.
$$
\end{thm}

Assuming the assumptions of this theorem are satisfied for a finite diffuse measure $\mu$, we may then apply Lemma \ref{trunckernelslimit} to get that
$$|\langle \RSO_{\Phi}(f\mu), \psi\rangle_{\mu}|\leq C\|f\|_{L^2(\mu)}\|\psi\|_{L^2(\mu)},
$$
and so this bilinear form gives rise to a bounded linear operator $\RSO_{\Phi, \mu}$ on $L^2(\mu)$.

\subsection{The bilinear forms $\langle \RSO(\varphi\mu), \psi\rangle_{\mu}$ and $\langle \RSO_{\Phi}(\varphi\mu), \psi\rangle_{\mu}$}

Suppose that $\mu$ is a finite measure with $\overline{D}_{\mu}\in L^1(\mu)$.  Then $\mu$ is diffuse. Indeed, just note that for any $x\in \R^d$ and $R>0$, $$\int_{B(0,R)} \frac{1}{|x-y|^{s-1}}d\mu(y) \leq C\int_0^{2R}\frac{\mu(B(x,r))}{r^s} dr\leq CR\overline{D}_{\mu}(x).$$
and consequently
$$\iint_{B(0,R)\times B(0,R)}\frac{1}{|x-y|^{s-1}}d\mu(x)d\mu(y)\leq CR\int_{B(0,R)}\overline{D}_{\mu}d\mu<\infty.
$$
Consequently, for any Lipschitz function $\Phi$, the bilinear forms $\langle \RSO(\varphi\mu), \psi\rangle_{\mu}$ and $\langle \RSO_{\Phi}(\varphi\mu), \psi\rangle_{\mu}$ are well defined for $\varphi, \psi\in \Lip_0(\R^d)$.

Fix an open set $\Omega$, and set $\Phi(x) = \dist(x,\R^d\backslash \Omega)$.  For $\eta>0$, put $g_{\eta} = (1-\tfrac{\dist(\,\cdot\, ,\R^d\backslash\Omega)}{\eta})_{+}$.  This function satisfies

$\bullet$  $g_{\eta}\equiv 1$ on $\R^d\backslash \Omega$,

$\bullet$  $g_{\eta}(x) = 0$ if $\dist(x, \R^d\backslash \Omega)>\eta$,

$\bullet$  $\|g_{\eta}\|_{\Lip}\leq \tfrac{1}{\eta}$.

\begin{lem}\label{PhiNoPhi}  Suppose that $\overline{D}_{\mu}\in L^1(\mu)$.  Then for every $\varphi,\psi\in \Lip_0(\R^d)$, $$\lim_{\eta\rightarrow 0}\bigl|\langle \RSO_{\Phi}(\varphi\mu), g_{\eta}\psi\rangle_{\mu} - \langle \RSO(\varphi\mu), g_{\eta}\psi\rangle_{\mu}\bigl|=0.
$$

\end{lem}

\begin{proof}
Let's write out $\langle \RSO_{\Phi}(\varphi\mu), g_{\eta}\psi\rangle_{\mu} - \langle \RSO(\varphi\mu), g_{\eta}\psi\rangle_{\mu}$ as a double integral:
$$\iint\limits_{\R^d\times\R^d}[K_{\Phi}(x,y)-K(x-y)]H_{\varphi,g_{\eta}\psi}(x,y)d\mu(x)d\mu(y).
$$
Notice that the domain of integration can be restricted to $x,y\in \Omega$ without changing the value of the double integral.  Furthermore, notice that the integrand is zero if both $\Phi(x) = \dist(x, \R^d\backslash \Omega)>\eta$ and $\Phi(y)=\dist(y, \R^d\backslash \Omega)>\eta$.  Consequently, it shall suffice to show that
$$\lim_{\eta\rightarrow 0}\iint\limits_{\substack{(x,y)\in \Omega\times \Omega:\\ \Phi(x)<\eta}}|K_{\Phi}(x,y)-K(x-y)||H_{\varphi,g_{\eta}\psi}(x,y)|d\mu(y)d\mu(x)=0
$$

If $\eta$ is small enough, then
$$|H_{\varphi,g_{\eta}\psi}(x,y)|\leq C(\varphi, \psi)\frac{|x-y|}{\eta}.
$$
Thus, for such $\eta>0$, the integral
$$I=\int\limits_{x\in \Omega: \,\Phi(x)<\eta,}\int\limits_{y\in B(x,\Phi(x))}|K(x-y)-K_{\Phi}(x,y)||H_{\varphi,g_{\eta}\psi}(x,y)|d\mu(y)d\mu(x)
$$
is at most a constant multiple of
$$\int\limits_{x\in \Omega: \,\Phi(x)<\eta,}\int\limits_{B(x,\Phi(x))}\frac{1}{\eta|x-y|^{s-1}}d\mu(y)d\mu(x).
$$
Applying the easy estimate $\int_{B(x, \eta)}\tfrac{1}{\eta|x-y|^{s-1}}d\mu(y)\leq C\overline{D}_{\mu}(x)$, we conclude that
$$I\leq C(\varphi, \psi)\int_{x\in \Omega :\, \Phi(x)<\eta}\overline{D}_{\mu}(x)d\mu(x).
$$

On the other hand, trivially $|H_{\varphi,g_{\eta}\psi}(x,y)|\leq C(\varphi, \psi)$ for all $x,y\in \R^d$.  Therefore we may use Lemma \ref{supkernelcomp} to estimate the integral
$$\int\limits_{x\in \Omega: \,\Phi(x)<\eta,}\int\limits_{\R^d\backslash B(x,\Phi(x))}|K(x-y)-K_{\Phi}(x,y)||H_{\varphi,g_{\eta}\psi}(x,y)|d\mu(y)d\mu(x)
$$
by a constant multiple (which may depend on $\psi$ and $\varphi$) of
$$\int\limits_{x\in \Omega: \,\Phi(x)<\eta,}\int\limits_{\R^d\backslash B(x,\Phi(x))}\frac{\Phi(x)}{|x-y|^{s+1}}d\mu(y)d\mu(x).
$$
This integral is again dominated by
$$\int\limits_{x\in \Omega: \,\Phi(x)<\eta}\overline{D}_{\mu}(x)d\mu(x).
$$
It remains to show that $$\lim_{\eta\rightarrow 0^+}\int\limits_{x\in \Omega: \,\Phi(x)<\eta}\overline{D}_{\mu}(x)d\mu(x)=0.$$
But this follows readily from the Dominated Convergence Theorem as  $\overline{D}_{\mu}\in L^1(\mu)$, and $\chi_{\{x\in \Omega: \Phi(x)<\eta\}}$ tends to zero pointwise as $\eta\rightarrow 0^+$ ($\Omega$ is an open set).
\end{proof}

\section{Step I: Calder\'{o}n-Zygmund theory (From a distribution to an $L^p$-function)}\label{CZreg}

Fix $A, A_1>1$.  Throughout this section, let us suppose that there is a finite measure $\mu$ supported in $AQ_0$, and functions $\varphi\in \Lip_0(AQ_0)$ with $0\leq \varphi\leq 1$, and $F\in L^1(\mu)$ that satisfy

(A)  $\overline{D}_{\mu}\in L^{2,\infty}(\mu)$ (with norm $\|\overline{D}_{\mu}\|_{2,\infty}$),

(B)  $\langle \RSO(\varphi\mu),\psi\rangle_{\mu} = \langle F,\psi\rangle_{\mu}$ for every $\psi\in \Lip_0(\R^d)$, and

(C) $|F| \leq A_1(1+\overline{D}_{\mu})$.

We set $\mu_{\varphi}=\varphi\mu$.

\subsection{Cotlar's lemma and the exceptional set $\Omega$}

\begin{lem}\label{omegacotlar}  There exists a constant $C(\mu)>0$ depending on $\|\overline{D}_{\mu}\|_{2,\infty}$, $\|\varphi\|_{\Lip}$, $A$, $A_1$, $d$, and $s$, such that for every $T>1$ there is an open set $\Omega=\Omega(T)\subset \R^d$ such that if $\Phi=\Phi^{(T)} = \dist(\cdot, \R^d\backslash \Omega)$ then the following properties hold

(i)   (small measure of $\Omega$) $\displaystyle \mu(\Omega)\leq \frac{C(\mu)}{T^2},$

(ii)  (controlled density) for every $x\in \R^d$ and $r>\tfrac{\Phi(x)}{2}$, $D_{\mu}(B(x,r))\leq 2^s T$, and

(iii) (the Cotlar estimate) for every $x\in \R^d$ and $\delta>0$,
$$|\RSO_{\Phi,\delta}(\mu_{\varphi})(x)|\leq C(\mu)T.
$$
\end{lem}

The sets $\Omega(T)$ will further satisfy that if $t\leq T$, then $\Omega(t)\supset \Omega(T)$.

Fix $T\geq 1$.  We shall split the proof of this lemma into a few pieces.  Throughout the remainder of this section $C(\mu)$ will denote a constant that may depend on $\|\overline{D}_{\mu}\|_{2,\infty}$, $\|\varphi\|_{\Lip}$, $A$, $A_1$, $d$, and $s$, and can change from line to line.

\subsection{Doubling balls}\label{double}  We shall call a ball $B(x,r)$ \emph{doubling} if $\mu(B(x,15r))\leq 225^s\mu(B(x,r))$.  Equivalently, $B(x,r)$ is doubling if $D_{\mu}(B(x,15r))\leq 15^sD_{\mu}(B(x,r))$.

Fix $x\in \R^d$, $r>0$.  Set $\rho_{\ell} = 15^{\ell}r$.  If $B(x,\rho_{\ell})$ fails to be doubling for every $\ell\in \{0,\dots,j-1\}$ with $j\geq 1$, then
$$D_{\mu}(B(x,\rho_j))>15^sD_{\mu}(B(x, \rho_{j-1}))>\cdots> 15^{sj}D_{\mu}(B(x,r)).
$$
In particular, this inequality combined with the finiteness of $\mu$ ensures that if $\mu(B(x,r))>0$ then there is the least index $k\geq0$ for which $B(x,\rho_k)$ is doubling.  For $j\leq k$ we have
\begin{equation}\label{densbelowdoub}
D_{\mu}(B(x,\rho_j))\leq 15^{-s(k-j)}D_{\mu}(B(x,\rho_k)),
\end{equation}
and so clearly also
$$D_{\mu}(B(x,\rho_k))\geq D_{\mu}(B(x,r)).
$$

\subsection{The construction of $\Omega$}

Consider the collection $\mathcal{B}$ of balls $B(x,3r)$ with the properties that $B(x,r)$ is doubling and $\int_{B(x,3r)}\overline{D}_{\mu}d\mu>T\mu(B(x,r)).$  Set $\Omega = \bigcup_{B\in \mathcal{B}} B$.

Now let $B_j=B(x_j, 3r_j)$ be a Vitali subcollection of $\mathcal{B}$.  That is, $(B_j)_j$ is a (possibly finite) sequence of pairwise disjoint balls from $\mathcal{B}$ that satisfy
 $$\bigcup_j 5B_j = \bigcup_j B(x_j, 15r_j)\supset \bigcup_{B\in \mathcal{B}}B.$$


\begin{lem}\label{omegameas}  There is a constant $C(\mu)>0$ such that
$$\mu(\Omega)\leq \frac{C(\mu)}{T^2}.
$$
\end{lem}

\begin{proof}First note that, for a doubling ball $B(x,r)$,
$$\int_{B(x,3r)\cap\bigl\{\overline{D}_{\mu}\leq \tfrac{T}{2\cdot 225^s}\bigl\}}\overline{D}_{\mu}d\mu\leq \frac{T}{2\cdot 225^s}\mu(B(x,3r))\leq \frac{T}{2}\mu(B(x,r)).
$$
Set
$$\overline{D}_{T,\mu}(x) = \begin{cases}\, 0 \text{ if }\overline{D}_{\mu}(x)\leq \tfrac{T}{2\cdot 225^s},\\\,\overline{D}_{\mu}(x)\text{ otherwise.}\end{cases}
$$

Then, for each ball $B_j=B(x_j, 3r_j)$ in the Vitali subcollection we have
$$\int_{B(x_j, 3r_j)}\overline{D}_{T,\mu}d\mu\geq \frac{T}{2}\mu(B(x_j, r_j)).
$$

Consequently
\begin{equation}\begin{split}\nonumber\mu(\Omega)\leq &\sum_j \mu(B(x_j, 15r_j))\leq 225^s\sum_j\mu(B(x_j, r_j))\\&\leq \frac{2\cdot 225^{s}}{T}\sum_j \int_{B_j} \overline{D}_{T,\mu}\,d\mu\leq \frac{2\cdot 225^{s}}{T}\int_{\bigl\{\overline{D}_{\mu}> \tfrac{T}{2\cdot 225^s}\bigl\}}\overline{D}_{\mu}\,d\mu,
\end{split}\end{equation}
where the penultimate inequality follows from the pairwise disjointness of the Vitali subcollection.  On the other hand, we have that
\begin{equation}\begin{split}\nonumber\int_{\bigl\{\overline{D}_{\mu}> \tfrac{T}{2\cdot 225^s}\bigl\}}\overline{D}_{\mu}\,d\mu&\leq \int_{0}^{\infty} \mu\bigl(\bigl\{x\in \R^d:\overline{D}_{\mu}(x)>\max\bigl(t,\frac{T}{2\cdot 225^s}\bigl)\bigl\}\bigl)dt\\&\leq \frac{C\|\overline{D}_{\mu}\|^2_{2,\infty}}{T},
\end{split}\end{equation}
and the proof is complete.
\end{proof}

\begin{lem} Whenever $x\in \R^d$ and $r\geq \tfrac{\Phi(x)}{2}$, we have $D_{\mu}(B(x,r)) \leq 2^s T$.
\end{lem}

\begin{proof}

First suppose that $B(x,r)$ is doubling and $D_{\mu}(B(x,r)) \geq 2^s T$.  Then for every $y\in B(x,r)$, we have
$$\overline{D}_{\mu}(y)\geq \frac{\mu(B(y, 2r))}{(2r)^s}\geq T.
$$
Thus $$\int_{B(x,3r)}\overline{D}_{\mu}d\mu>T\mu(B(x,r)),$$ and so $B(x,3r)\in \mathcal{B}$.  It therefore follows that $\Phi(x)\geq 3r$.

But now if  $r\geq \tfrac{\Phi(x)}{2}$ and $D_{\mu}(B(x,r)) \geq 2^s T$, then by considering the smallest doubling ball $B(x, 15^kr)$ containing $B(x,r)$ (see Section \ref{double}) we reach a contradiction with the conclusion of the previous paragraph. The lemma follows.\end{proof}

\subsection{The proof of Lemma \ref{omegacotlar}.}

Having defined the set $\Omega$, and verified properties (i) and (ii) from Lemma \ref{omegacotlar}, we now complete the proof by proving property (iii), the Cotlar estimate.   The ideas primarily originate in the work of David and Mattila \cite{Dav, DM}.

Fix $x\in \R^d$ and $\delta>0$.  From Lemma \ref{fancyT1comp}, applied with the Lipschitz function $\Phi_{\delta}$ and finite measure $\mu_{\varphi}$, we see that it suffices to estimate $\int_{|x-y|\geq \Phi_{\delta}(x)}K(x-y)d\mu_{\varphi}(y)$, as $D_{\mu}(B(x,r))\leq CT$ for every $r\geq\Phi_{\delta}(x)\geq \Phi(x)$.

Now set $\rho_j = 15^j \Phi_{\delta}(x)$, $j\in \mathbb{N}$, and set $k$ to be the least index such that $B(x,\rho_k)$ is doubling and $\mu(B(x, \rho_k))>0$. Then by (\ref{densbelowdoub}), we have for $0\leq j\leq k$,
$$D_{\mu}(B(x, \rho_{j}))\leq \frac{1}{15^{s(k-j)}}D_{\mu}(B(x, \rho_k))\leq \frac{CT}{15^{s(k-j)}}.
$$
Using this bound in a crude manner, we obtain that
$$\int_{\Phi_{\delta}(x)\leq|x-y|\leq \rho_k}|K(x-y)|d\mu_{\varphi}(y)\leq C\sum_{0\leq j \leq k}D_{\mu}(B(x, \rho_j))\leq CT.$$
Also notice that since $\supp(\mu)\subset AQ_0$, we have that $\rho_k\leq C A$.

The upshot of these remarks is that it now suffices to prove the estimate
\begin{equation}\label{doubleest}\Bigl|\int_{|x-y|\geq \, r}K(x-y)d\mu_{\varphi}(y)\Bigl|\leq C(\mu)T,
\end{equation}
where $B(x,r)$ is a doubling ball with $\mu(B(x,r))>0$ and $\Phi(x)<r<C A$.  Observe that since $r>\Phi(x)$, the ball $B(x,3r)$ intersects $\R^d\backslash \Omega$, and so $\int_{B(x,3r)}\overline{D}_{\mu}d\mu\leq T\mu(B(x,r))$.

We shall use property (B)  to prove (\ref{doubleest}).  In order to do so, we will introduce two cut-off functions.  Let $\psi\in \Lip_0(B(x, \tfrac{3}{2}r))$ satisfy $\psi\equiv 1$ in $B(x,r)$, $\psi\geq 0$ in $\R^d$ and $\|\psi\|_{\Lip}\leq \tfrac{C}{r}$, and let $f\in \Lip_0(B(x, 3r))$ satisfy $f\equiv 1$ on $B(x, 2r)$, $f\geq 0$ on $\R^d$, and $\|f\|_{\Lip}\leq \tfrac{C}{r}$.

First notice that for every $x'\in \supp(\psi)$,
\begin{equation}\begin{split}\nonumber \Bigl|\int_{|x-y|>r}K(x-y)d\mu_{\varphi}(y) - \RSO([1-f]\mu_{\varphi})(x')\Bigl|\leq CT.
\end{split}\end{equation}
Indeed, since $\supp(\psi)\subset B(x,\tfrac{3}{2}r)$, and $f\equiv 1$ on $B(x,2r)$, we get that the left hand side is bounded by
\begin{equation}\begin{split}\nonumber \frac{C}{r^s}\mu(B(x,3r))&+ \Bigl|\int_{|x-y|>3r}[K(x-y)-K(x'-y)]d\mu_{\varphi}(y)\Bigl|\\&\leq CT+ Cr\int_{r}^{\infty}\frac{\mu(B(x,t))}{t^s}\frac{dt}{t^2}\leq CT.
\end{split}\end{equation}

Averaging this bound with respect to the measure $\psi d\mu$, we obtain that
$$\Bigl|\int_{|x-y|>r}K(x-y)d\mu_{\varphi}(y)-\frac{1}{\|\psi\|_{L^1(\mu)}}\langle \RSO([1-f]\mu_{\varphi}),\psi\rangle_{\mu}\Bigl|\leq CT.
$$

On the other hand,  write $$\langle \RSO(f\mu_{\varphi}),\psi\rangle_{\mu}=\iint\limits_{B(x,3r)\times B(x,3r)}K(y-z)H_{f\varphi,\psi}(y,z)d\mu(y)d\mu(z).
$$
Since $|H_{f\varphi, \psi}(y,z)|\leq C(\|\varphi\|_{\Lip}+\tfrac{1}{r})|y-z|$, we get that
$$|\langle \RSO(f\mu_{\varphi}),\psi\rangle_{\mu}|\leq C\Bigl(\|\varphi\|_{\Lip}+\frac{1}{r}\Bigl)\iint\limits_{B(x,3r)\times B(x,3r)}\frac{1}{|y-z|^{s-1}}d\mu(y)d\mu(z).
$$
The right hand side here is bounded by a constant multiple of
$$\Bigl(\|\varphi\|_{\Lip}+\frac{1}{r}\Bigl)r\int_{B(x,3r)}\overline{D}_{\mu}d\mu\leq C(\|\varphi\|_{\Lip}r+1)T\mu(B(x,r)),
$$
which is in turn bounded by $CA(1+\|\varphi\|_{\Lip})T\|\psi\|_{L^1(\mu)}.$  Thus
$$\frac{1}{\|\psi\|_{L^1(\mu)}}\Bigl|\langle \RSO(f\mu_{\varphi}),\psi\rangle_{\mu}\Bigl|\leq CA(1+\|\varphi\|_{\Lip})T.
$$

It remains to estimate $\tfrac{1}{\|\psi\|_{L^1(\mu)}}|\langle \RSO(\mu_{\varphi}), \psi\rangle_{\mu}|$, and this is where condition (B) will be used:
$$
|\langle \RSO(\mu_{\varphi}), \psi\rangle_{\mu}|=|\langle F,\psi\rangle_{\mu}|\leq CA_1\int_{B(x,3r)}(1+\overline{D}_{\mu})d\mu\leq CA_1T\mu(B(x,r)),
$$
and this is bounded by $CA_1T\|\psi\|_{L^1(\mu)}$.

Bringing together our estimates, we see that (\ref{doubleest}) holds, and so the proof of the lemma is completed.

\subsection{The Non-homogeneous $\mathcal{T}(1)$-theorem}\label{CZStuff}

Fix $T\geq 1$.  Set $\Omega=\Omega(T)$ and $\Phi=\Phi^{(T)}$ as in Lemma \ref{omegacotlar}.  Since
 $$|\RSO_{\Phi, \delta}(\tfrac{\mu_{\varphi}}{T})(x)|\leq C(\mu)\text{ for all }x\in \R^d\text{ and }\delta>0,
 $$
 and $\tfrac{\mu_{\varphi}(B(x,r))}{T}\leq 2^sr^s$ whenever $r\geq \Phi(x)$, we may apply the Nazarov-Treil-Volberg $\TSO(1)$-theorem for suppressed kernels (see Theorem \ref{biggun} and the discussion immediately following it), to conclude that $\RSO_{\Phi,\mu_{\varphi}}$ is a bounded operator on $L^2(\mu_{\varphi})$ with norm at most $C(\mu)T$.

\begin{lem}\label{smalllp}  Let $p\in (1,2)$.  Then there is a constant $C(\mu,p)>0$ such that
$$\|\RSO_{\Phi, \mu_{\varphi}}(\chi_{\Omega})\|_{L^p(\mu_{\varphi})}^p \leq \frac{C(\mu, p)}{T^{2-p}}.
$$
\end{lem}

 Readers with an advanced knowledge of non-homogeneous Calder\'{o}n-Zygmund theory can view this lemma as a corollary of the fact that the $L^2(\mu_{\varphi})$ boundedness of the operator $\RSO_{\Phi, \mu_{\varphi}}$ self-improves to yield that if $1<p<\infty$, then   $\RSO_{\Phi,\mu_{\varphi}}$  is also bounded on $L^p(\mu_{\varphi})$ with operator norm at most $C(\mu,p)T$ (see Chapter 5 of \cite{Tol}).   However, by using the structure of $\Omega$, we provide a simple self-contained proof.

\begin{proof}  Recall that $\Omega=\Omega(T)$ is an open set, so there is a sequence $f_{n}\in \Lip_0(\Omega)$ that pointwise increases to $\chi_{\Omega}$.  If $1\leq t\leq T$, then we set $\Omega(t)$ and $\Phi^{(t)}$ to be as in Lemma \ref{omegacotlar}.  Thus $\RSO_{\Phi^{(t)}, \mu_{\varphi}}$ is a bounded operator on $L^2(\mu_{\varphi})$ with norm at most $C(\mu)t$, and $\mu(\Omega(t/2))\leq \tfrac{C(\mu)}{t^2}$.  Since $\Omega(t)\supset \Omega(T)$,  Lemma \ref{aedefine} yields that for $\mu_{\varphi}$-almost every $x\not\in \Omega(t)$,
$$\RSO_{\Phi, \mu_{\varphi}}(f_n)(x) =  \int_{\R^d}K(x-y)f_n(y)d\mu(y)=\RSO_{\Phi^{(t)}, \mu_{\varphi}}(f_n)(x).
$$
But $f_n$ converges to $\chi_{\Omega}$ in $L^2(\mu_{\varphi})$ as $n\rightarrow \infty$, and so we have that
$$\RSO_{\Phi, \mu_{\varphi}}(\chi_{\Omega}) = \RSO_{\Phi^{(t)}, \mu_{\varphi}}(\chi_{\Omega})\; \mu_{\varphi}\text{-almost everywhere on }\R^d\backslash\Omega(t).
$$
Now write
\begin{equation}\begin{split}\nonumber\int_{\R^d}|\RSO_{\Phi, \mu_{\varphi}}(\chi_{\Omega})|^p&d\mu_{\varphi}\leq \int_{\Omega}|\RSO_{\Phi, \mu_{\varphi}}(\chi_{\Omega})|^pd\mu_{\varphi}\\&+\sum_{0\leq j\leq \log_2 T}\int_{\Omega(2^{-(j+1)}T)\backslash \Omega(2^{-j}T)}|\RSO_{\Phi, \mu_{\varphi}}(\chi_{\Omega})|^pd\mu_{\varphi}\\&+\int_{\R^d\backslash \Omega(1)}|\RSO_{\Phi, \mu_{\varphi}}(\chi_{\Omega})|^pd\mu_{\varphi}
\end{split}\end{equation}
With $t=2^{-j}T$,
\begin{equation}\begin{split}\nonumber\int_{\Omega(t/2)\backslash \Omega(t)}|&\RSO_{\Phi, \mu_{\varphi}}(\chi_{\Omega})|^pd\mu_{\varphi}=\int_{\Omega(t/2)\backslash \Omega(t)}|\RSO_{\Phi^{(t)}, \mu_{\varphi}}(\chi_{\Omega})|^pd\mu_{\varphi},
\end{split}\end{equation}
and so by H\"{o}lder's inequality
$$\int_{\Omega(t/2)\backslash \Omega(t)}|\RSO_{\Phi^{(t)}, \mu_{\varphi}}(\chi_{\Omega})|^pd\mu_{\varphi}\leq \|\RSO_{\Phi^{(t)}, \mu_{\varphi}}(\chi_{\Omega})\|_{L^2(\mu_{\varphi})}^{p}\mu(\Omega(t/2))^{1-\tfrac{p}{2}}.
$$
Plugging in the estimate for the operator norm of $\RSO_{\Phi^{(t)}, \mu_{\varphi}}$ and the measure estimate for $\mu(\Omega(t/2))$, we therefore arrive at the inequality
$$\int_{\Omega(t/2)\backslash \Omega(t)}|\RSO_{\Phi^{(t)}, \mu_{\varphi}}(\chi_{\Omega})|^pd\mu_{\varphi}\leq C(\mu) \Bigl(\frac{t}{T}\Bigl)^p\frac{1}{t^{2-p}}=C(\mu)\frac{t^{2(p-1)}}{T^p}.
$$
Similarly
$$\int_{\R^d\backslash \Omega(1)}|\RSO_{\Phi, \mu_{\varphi}}(\chi_{\Omega})|^pd\mu_{\varphi}=\int_{\R^d\backslash \Omega(1)}|\RSO_{\Phi^{(1)}, \mu_{\varphi}}(\chi_{\Omega})|^pd\mu_{\varphi}\leq \frac{C(\mu)}{T^{p}}.
$$
Bringing these estimates together, we therefore see that
$$\int_{\R^d}|\RSO_{\Phi, \mu_{\varphi}}(\chi_{\Omega})|^pd\mu_{\varphi}\leq C(\mu)\sum_{j=0}^{\infty} 2^{-2(p-1)j} T^{(p-2)},
$$
which proves the lemma.
\end{proof}


 For the remainder of the paper we shall fix $p=\tfrac{3}{2}$.  Any other fixed choice of $p\in (1,2)$ would work just as well in the subsequent argument.

Set $\mu' =\chi_{\R^d\backslash \Omega}\mu_{\varphi}$.   Then $\overline{D}_{\mu'}(x)\leq 2^sT$ for every $x\in \supp(\mu')$.

Notice that, for any $f,\psi\in \Lip_0(\R^d)$, we have $$ \langle \RSO_{\Phi}(f\mu'),\psi\rangle_{\mu'} = \langle \RSO(f\mu'),\psi\rangle_{\mu'}.$$
Therefore, Lemma \ref{restrictcoincidence} yields that
\begin{equation}\begin{split}\nonumber|\langle \RSO(f\mu'),\psi\rangle_{\mu'}| &=|\langle \RSO_{\Phi}(f\mu'),\psi\rangle_{\mu'}|=|\RSO_{\Phi, \mu_{\varphi}}(f\chi_{\R^d\backslash \Omega}\mu_{\varphi}), \psi\chi_{\R^d\backslash \Omega}\rangle_{\mu_{\varphi}}|\\&\leq C(\mu)T\|f\|_{L^2(\mu')}\|\psi\|_{L^{2}(\mu')},\end{split}\end{equation}
and so the operator $\RSO_{\mu'}:L^2(\mu')\rightarrow L^2(\mu')$ emerging from this bilinear form is bounded with norm at most $C(\mu)T$.

 We now wish to see how the property (B) relates to the measure $\mu'$.

\begin{cla}\label{restrictB}  For every $\psi\in \Lip_0(\R^d)$,
$$\langle \RSO_{\mu'}(1),\psi\rangle_{\mu'}=\langle \RSO(\mu'),\psi\rangle_{\mu'} = \langle F,\psi\rangle_{\mu'}-\langle \RSO_{\Phi, \mu_{\varphi}}(\chi_{\Omega}), \psi\chi_{\mathbb{R}^d\backslash \Omega}\rangle_{\mu_{\varphi}}.
$$
\end{cla}

\begin{proof}  We first notice that Corollary \ref{obvious} yields
$$\langle \RSO_{\mu'}(1), \psi\rangle_{\mu'}=\langle \RSO_{\Phi,\mu_{\varphi}}(\chi_{\mathbb{R}^d\backslash \Omega}), \chi_{\R^d\backslash \Omega}\psi\rangle_{\mu_{\varphi}}.
$$
Now write
\begin{equation}\begin{split}\nonumber\langle \RSO_{\Phi,\mu_{\varphi}}(\chi_{\mathbb{R}^d\backslash \Omega}), \chi_{\R^d\backslash \Omega}\psi\rangle_{\mu_{\varphi}}= & \langle \RSO_{\Phi,\mu_{\varphi}}(1), \chi_{\R^d\backslash \Omega}\psi\rangle_{\mu_{\varphi}}\\&- \langle \RSO_{\Phi,\mu_{\varphi}}(\chi_{\Omega}), \chi_{\R^d\backslash \Omega}\psi\rangle_{\mu_{\varphi}}.
\end{split}\end{equation}

The second term in the right hand side of the equality is precisely the second term appearing in the right hand side of the claimed identity, so it suffices to show that $\langle \RSO_{\Phi,\mu_{\varphi}}(1), \chi_{\R^d\backslash \Omega}\psi\rangle_{\mu_{\varphi}}=\langle F,\psi\rangle_{\mu'}$.

Since the functions $g_{\eta}$ converge to $\chi_{\R^d\backslash \Omega}$ in $L^2(\mu)$, we see that
$$
\langle \RSO_{\Phi,\mu_{\varphi}}(1), \chi_{\R^d\backslash \Omega}\psi\rangle_{\mu_{\varphi}}=\lim_{\eta\to 0}\langle \RSO_{\Phi,\mu_{\varphi}}(1), g_{\eta}\psi\rangle_{\mu_{\varphi}}.
$$
On the other hand, for $\eta>0$,
$$\langle \RSO_{\Phi,\mu_{\varphi}}(1), g_{\eta}\psi\rangle_{\mu_{\varphi}}=\langle \RSO_{\Phi}(\varphi\mu), g_{\eta}\psi\varphi\rangle_{\mu},
$$
and consequently, Lemma \ref{PhiNoPhi} ensures that
$$
\langle \RSO_{\Phi,\mu_{\varphi}}(1), \chi_{\R^d\backslash \Omega}\psi\rangle_{\mu_{\varphi}}=\lim_{\eta\to 0}\langle \RSO(\varphi\mu), g_{\eta}\psi\varphi\rangle_{\mu}.
$$
We therefore deduce from the property (B) that $$\langle \RSO(\varphi\mu), g_{\eta}\psi\varphi\rangle_{\mu}=\langle F, g_{\eta}\psi\varphi\rangle_{\mu}.$$

Finally, applying the the dominated convergence theorem yields that
$$\lim_{\eta\to 0}\langle F, g_{\eta}\psi\varphi\rangle_{\mu}= \langle F, \chi_{\R^d\backslash \Omega}\psi\varphi\rangle_{\mu}=\langle F, \psi\rangle_{\mu'}.
$$
The claim is proven.
\end{proof}

From the claim we find that the $L^p(\mu')$ function $\RSO_{\mu'}(1)$ satisfies
$$\RSO_{\mu'}(1) = F-\RSO_{\Phi, \mu_{\varphi}}(\chi_{\Omega})\;\;\; \mu'\text{-a.e.},
$$
and
\begin{equation}\begin{split}\nonumber\|\RSO_{\Phi, \mu_{\varphi}}(\chi_{\Omega})\|^p_{L^p(\mu')}\leq C(\mu,p)T^{-(2-p)}.
\end{split}\end{equation}

\subsection{The boundedness of the Riesz transform implies zero density of the measure (at least if $s\in (d-1,d)$)}\label{densitytheorems}  Since the Riesz transform operator $\RSO_{\mu'}$ is bounded in $L^2(\mu')$, and $s\in (d-1,d)$, we may deduce from a result in Eiderman-Nazarov-Volberg \cite{ENV} that for $\mu'$-almost every $x\in \R^d$
$$\limsup_{r\rightarrow 0}D_{\mu'}(B(x,r))=0.
$$
This result is only known in the case $s\in (d-1,d)$.  (As we will be broadly following the scheme of the paper \cite{ENV}, we shall run up against the authors' obstruction to extending this result to $s<d-1$: it occurs when we wish to extend the inequality that appears in Lemma \ref{firstvar} below from an inequality on the support of a measure to an inequality in the entire space.)

We actually know of two ways to arrive at the desired statement.  We shall describe momentary how to obtain the theorem directly from results in \cite{ENV}.  However, it is perhaps worth mentioning that one can alternatively derive the same conclusion directly from Theorem 1.3 of \cite{JN3}, where it is proved that there is some large exponent $q>0$ such that
$$\int_{\R^d}\Bigl[\int_0^{\infty}\Bigl(\frac{\mu'(B(x,r))}{r^s}\Bigl)^q\frac{dr}{r}\Bigl]d\mu'(x)<\infty.$$  The main advantage in doing so is that Theorem 1.3 of \cite{JN3} follows from Proposition \ref{regularprop} above (which already plays an essential role in this paper) without much difficulty, see Sections 4 and 5 of \cite{JN3}.

To derive the desired result directly from \cite{ENV}, consider the set
$$F= \Bigl\{x\in \R^d: \limsup_{r\rightarrow 0^+}D_{\mu'}(B(x,r))>0\Bigl\}=\bigcup_nF_n,
$$
where $F_n = \{x\in \R^d: \limsup_{r\rightarrow 0^+}D_{\mu'}(B(x,r))>\tfrac{1}{n}\}$.  A standard application of the Vitali covering lemma ensures that $\mathcal{H}^s(F_n)\leq Cn\mu'(\R^n)$.  But also we have that the $s$-Riesz transform associated to the measure $\chi_{F_n}\mu'$ (whose support has finite $\mathcal{H}^s$-measure) is bounded in $L^2(\chi_{F_n}\mu')$.  The theorem stated in Section 22 of \cite{ENV} then yields that $\mu(F_n)=0$.  Thus $\mu(F)=0$, which is precisely what was to be proved.\\

Since $\mu'$ is a finite measure, we may apply Egoroff's theorem (and Borel regularity) to find a closed set $E\subset \R^d$ such that $\mu'(\R^d\backslash E)\leq T^{-4/p}$, and also  $$\lim_{r\to 0}\sup_{x\in E}D_{\mu'}(B(x,r))=0.$$  Then
\begin{equation}\begin{split}\nonumber\|\RSO_{\mu'}(\chi_{\R^d\backslash E})\|^p_{L^p(\mu')}&\leq \|\RSO_{\mu'}(\chi_{\R^d\backslash E})\|^p_{L^2(\mu')}\mu'(\R^d)^{1-\tfrac{p}{2}}\\&
\leq C(\mu)T^p\mu'(\R^d\backslash E)^{p/2}\leq C(\mu)T^{-(2-p)}.
\end{split}\end{equation}

\subsection{The measure $\nu$}  Set $\nu = \chi_{E}\mu'$.  Then $\RSO_{\nu}(1)$, as a function in $L^p(\nu)$, equals
$$F-H,
$$
where $H=\RSO_{\Phi, \mu_{\varphi}}(\chi_{\Omega})+\RSO_{\mu'}(\chi_{\R^d\backslash E})$ satisfies
$$\|H\|^p_{L^p(\nu)}\leq C(\mu)T^{-(2-p)}.
$$
For future reference, let us record that this implies that

$\bullet$  $\displaystyle \int_{\R^d}(|\RSO_{\nu}(1)|-|F|)^p_+d\nu\leq C(\mu)T^{-(2-p)}.$

In addition, the measure $\nu$ has the following properties,

$\bullet$  $\overline{D}_{\mu}(x)\leq 2^sT$ for all $x\in \supp(\nu)$,

$\bullet$ $\lim_{r\rightarrow 0}\sup_{x\in \supp(\nu)}D_{\nu}(B(x,r))=0$,

$\bullet$ $\nu(\overline{Q}_0)\geq \mu(\overline{Q}_0) - \frac{C(\mu)}{T^2},$ and

$\bullet$ $\nu(AQ_0)=\nu(\R^d) \leq \mu(AQ_0).$

\section{Step II: The smoothing operation}\label{Smoothing}

\emph{Throughout this section we shall  suppose that part (ii) of Alternative \ref{restate} holds.}  Fix $A>0$ large enough and consider the finite measure $\mu$ given in the second part of the alternative.

Since $\overline{D}_{\mu}\in L^{2,\infty}(\mu)$, and $G_A(\chi_{AQ_0\backslash \tfrac{A}{2}Q_0}\mu)\leq C\overline{D}_{\mu}$, we have that $G_A(\chi_{AQ_0\backslash \tfrac{A}{2}Q_0}\mu)\in L^p(\mu)$ (recall that we have fixed $p=\tfrac{3}{2}$).  We infer from property (e) of Alternative \ref{restate} and H\"{o}lder's inequality that there is a constant $C(\mu)>0$ such that for all $\psi\in \Lip_0(\R^d)$,
$$|\langle \RSO(\varphi\mu), \psi\rangle_{\mu}|\leq C(\mu)\|\psi\|_{L^{p'}(\mu)},
$$
where $p' = \tfrac{p}{p-1}$.  Consequently, there exists $F\in L^p(\mu)\subset L^1(\mu)$ so that
$$\langle \RSO(\varphi\mu), \psi\rangle_{\mu}=\langle F,\psi\rangle_{\mu}\text{ for every }\psi\in \Lip_0(\R^d),
$$
 and $|F|\leq C_4 G_A(\mu_A)+C_4A^{-\gamma}$, where $\mu_A = (\log A)\chi_{AQ_0\backslash \tfrac{A}{2}Q_0}\mu$.  Therefore $|F| \leq A_1(1+\overline{D}_{\mu})$, where $A_1=C\log A$.

 Now fix a huge number $T$ to be chosen later (its order of magnitude will be much larger than some high power of $A$).  The Calder\'{o}n-Zygmund theory of the previous section provides us with a finite measure $\nu\leq \varphi\mu$ whose associated Riesz transform operator $R_{\nu}$ is bounded in $L^2(\nu)$ with norm at most $C(\mu)T$ for some constant $C(\mu)>0$ which may depend on $d$, $s$, $A$, $\|\varphi\|_{\Lip}$ and $\|\overline{D}_{\mu}\|_{2,\infty}$.  Additionally, the measures $\nu$ and $\mu_A$ have the following properties:

$\bullet$  If $S(\delta) = \sup_{x\in \R^d, r<\delta} D_{\nu}(B(x,r))$, then $\lim_{\delta\rightarrow 0}S(\delta)=0$.

$\bullet$  For every $x\in \supp(\nu)$, we have $\overline{D}_{\mu}(x)\leq 2^s T$, and so $\overline{D}_{\mu_A}(x)\leq 2^s (\log A)T$.

$\bullet$ $\nu(\overline{Q}_0)\geq 1-C(\mu)T^{-2}$.

$\bullet$ $\nu(AQ_0)\leq \mu(AQ_0)\leq A^{-\gamma}A^s$.

$\bullet$  The inequality $$\int_{\R^d}\Bigl[|\RSO_{\nu}(1)|-C_4 G_A(\mu_A)-C_4A^{-\gamma}\Bigl]_+^pd\nu\leq C(\mu)T^{-(2-p)}$$
holds.

 The goal of this section is to show that one can replace $\nu$ and $\mu_A$ by smoothed measures, while distorting other important characteristics of these measures by an arbitrarily small amount.  The construction we use is in essence a trivial version of the construction in \cite{ENV}.

Fix  a separation parameter $0<\sigma \ll 1$, an enlargement parameter $M \gg 1$, a density parameter $0<\kap \ll 1$, and a scale parameter $0<\delta\ll 1$, to be chosen in that order.

 We shall suppose that $\delta$ is chosen small enough to ensure that  \begin{equation}\label{Mdelta}S(2M\delta)\leq \kap.\end{equation}

 \subsection{The small boundary mesh}

 Consider a cube mesh of sidelength $\delta$ with the property that the $8\sigma\delta$-neighbourhood of the union of all boundaries of the cubes carries $\nu$ measure at most $C\sigma\nu(\R^d)$.  See Appendix \ref{smallbdary} for a proof of the existence of such a small boundary mesh.

  We shall label the (finite collextion of) cubes in the mesh that intersect $\supp(\nu)$ by $(Q_j)_j$.  Set $E'=\bigcup_j (1-8\sigma)Q_j$, and $\nu'=\chi_{E'}\nu$.  Notice that $\nu'(\overline{Q}_0)\geq \nu(\overline{Q}_0)-C\sigma\nu(\R^d)$.

 Choose a nonnegative function $\psi\in \Lip_0(B(0,\sigma\delta))$ satisfying $\|\psi\|_{\infty}\leq \tfrac{C}{(\sigma\delta)^d}$, and $\int \psi dm_d=1$.  We define the smoothed measures
 $$\wt{\nu}=\psi*\nu', \text{ and }\wt{\mu} = \psi*\mu_A.
 $$

Since $\wt\nu$ has bounded density with respect to $m_d$, we have that the potential $$\RSO(\wt\nu)(x) = \int_{\R^d}K(x-y)d\wt\nu(y)$$ is a continuous function on $\R^d$.

Certainly we have that both $\supp(\wt\mu)$ and $\supp(\wt\nu)$ are contained in $2AQ_0$, and also $\wt\nu(2\overline{Q}_0)\geq \nu'(\overline{Q}_0)$, while $$D_{\wt\mu}(2AQ_0)\leq D_{\mu_A}(AQ_0)\leq C(\log A)A^{-\gamma}, \text{ and } D_{\wt\nu}(2AQ_0)\leq C.$$

The smoothed measure $\wt{\nu}$ is supported in the union of the cells $\bigcup_j W_j$, where $W_j$ is the $\sigma\delta$-neighbourhood of $(1-8\sigma)Q_j$.  Set $\wt{W}_j$ to be the  $\sigma\delta$-neighbourhood of $W_j$.  Notice that $\dist(\wt{W}_j, \wt{W}_k)\geq \sigma\delta$ if $j\neq k$.   For each $W\in \{W_j\}_j$ that intersects $\supp(\nu)$, fix some $x_W\in W\cap\supp(\nu)$.

 Notice that
\begin{equation}\begin{split}\nonumber\int_{\R^d}|\RSO_{\nu}(\chi_{\R^d\backslash E'})|^pd\nu'&\leq \|\RSO_{\nu}(\chi_{\R^d\backslash E'})\|_{L^2(\nu)}^p\nu(\R^d)^{1-p/2}\\&\leq C(\mu) T^p\nu\bigl(\R^d\backslash E'\bigl)^{p/2}\nu(\R^d)^{1-p/2}\\&\leq C(\mu)T^p\sigma^{p/2} \nu(\R^d)\leq C(\mu)T^{-(2-p)},
\end{split}\end{equation}
provided that $\sigma\leq T^{-4/p}$.  Thus,
\begin{equation}\begin{split}\label{nuprimesmall}
\int_{\R^d}\Bigl[|\RSO_{\nu'}(1)|-&C_4 G_A(\mu_A)-C_4A^{-\gamma}\Bigl]_+^pd\nu'\leq C(\mu)T^{-(2-p)}.
\end{split}\end{equation}

In order to see how replacing $\nu'$ by $\wt\nu$ and $\mu_A$ by $\wt\mu$ impacts the inequality (\ref{nuprimesmall}), we shall prove two comparison lemmas.  We introduce the notation $\dashint_W \,fd\nu = \tfrac{1}{\nu(W)}\int_W f\,d\nu.$

\begin{lem}\label{Rieszcomp}  There is a constant $C>0$ such that for any cell $W$ and $x\in W$,
$$|\RSO(\wt\nu)(x)|\leq \dashint_W |\RSO_{\nu'}(1)|d\nu'+\frac{CM^s\kap}{\sigma^s}+\frac{CT}{M}.
$$
\end{lem}

\begin{lem}\label{potentialcomp}  There is a constant $C>0$ such that for any cell $W$ and $x\in W$,
$$\dashint_W G_{A}(\mu_A)d\nu'\leq G_A(\wt\mu)(x)+CT\delta\log^2\Bigl(\frac{A}{\delta}\Bigl).
$$
\end{lem}

These two comparison lemmas will be proved in the next subsection.  Let us fix $\sigma = T^{-4/p}$ and $M = T^2$.  Then fixing $\kap$ so that $\frac{M^s\kap}{\sigma^s}\leq \tfrac{1}{T}$ yields that for any cell $W$ and $x\in W$,
$$|\RSO(\wt\nu)(x)|\leq \dashint_W |\RSO_{\nu'}(1)|d\nu'+\frac{C}{T}.
$$
But now we may impose that $\delta$ be small enough so that $CT\delta\log^2\bigl(\tfrac{A}{\delta}\bigl)\leq \tfrac{1}{T}$.  Then for any cell $W$ and $x\in W$,
\begin{equation}\begin{split}\nonumber |\RSO(\wt\nu)(x)|&-C_4G_A(\wt\mu)(x)-C_4A^{-\gamma} \\&\leq \dashint_{W}\bigl(|\RSO_{\nu'}(1)|-C_4G_{A}(\mu_A)-C_4A^{-\gamma}\bigl) d\nu'+\frac{C}{T}.
\end{split}\end{equation}
Now raise both sides to the power $p$ after taking the positive part.  Then on $W$ we have that
\begin{equation}\begin{split}\nonumber\Bigl[&|\RSO(\wt\nu)|-C_4G_A(\wt\mu)-C_4A^{-\gamma}\Bigl]_+^p\\
&\leq 2^{p-1}\Bigl[\dashint_{W}\bigl(|\RSO_{\nu'}(1)|-C_4G_{A}(\mu_A)-C_4A^{-\gamma}\bigl)d\nu' \Bigl]_+^p+\frac{C}{T^p}.
\end{split}\end{equation}
Next we integrate both sides of this inequality with respect to $\wt\nu$.  Since $\wt\nu(W) = \nu'(W)$, from Jensen's inequality (applied with the convex function $t\to t_+^p$) we obtain that
\begin{equation}\begin{split}\nonumber&\int_{W}\Bigl[|\RSO(\wt\nu)|-C_4G_A(\wt\mu)-C_4A^{-\gamma}\Bigl]_+^p d\wt\nu\\&\leq  2^{p-1}\int_{W}\Bigl[|\RSO_{\nu'}(1)|-C_4G_{A}(\mu_A)-C_4A^{-\gamma}\Bigl]_+^pd\nu'+\frac{C}{T^p}\nu'(W).
\end{split}\end{equation}
After that, summing over the cells yields
\begin{equation}\begin{split}\nonumber&\int_{\R^d}\Bigl[|\RSO(\wt\nu)|-C_4G_A(\wt\mu)-C_4A^{-\gamma}\Bigl]_+^p d\wt\nu\\&\leq  2^{p-1}\!\!\int_{\R^d}\Bigl[|\RSO_{\nu'}(1)|-C_4G_{A}(\mu_A)-C_4A^{-\gamma} \Bigl]_+^pd\nu'+\frac{C}{T^p}\nu'(\R^d).
\end{split}\end{equation}
Finally, from (\ref{nuprimesmall}) we conclude that
\begin{equation}\begin{split}\label{smoothsmallexcess}\int_{\R^d}\Bigl[|\RSO(\wt\nu)|-C_4G_A(\wt\mu)&-C_4A^{-\gamma}\Bigl]_+^p d\wt\nu \leq \frac{C(\mu)}{T^{2-p}}+\frac{C\nu(\R^d)}{T^p}.
\end{split}\end{equation}

Now set ${\wt{\nu}}_1 = \wt\nu(2\cdot)$ and ${\wt{\mu}}_1=2\wt\mu(2\cdot)$.  Then for every $x\in \R^d$ $$\RSO(\wt\nu_1)(x) = 2^s\RSO(\wt\nu)(2x)\text{ and } G_{A}(\wt\mu_1)(x) =  2^s G_{A}(\wt\mu)(2x).$$
Thus from (\ref{smoothsmallexcess}) we have that
\begin{equation}\begin{split}\label{smoothsmallexcess2}\int_{\R^d}\Bigl[|\RSO(\wt\nu_1)|-C_4G_A(\wt\mu_1)&-2^{s}C_4A^{-\gamma}\Bigl]_+^p d\wt\nu_1 \\&\leq 2^s\Bigl(\frac{C(\mu)}{T^{2-p}}+\frac{C\nu(\R^d)}{T^p}\Bigl)\leq \lambda,
\end{split}\end{equation}
where $\lambda =\frac{C(\mu)}{T^{2-p}}$ is arbitrarily small.

Notice the following properties of $\wt\nu_1$ and $\wt\mu_1$:
\begin{align*}
\bullet \;&  \wt\nu_1(\overline{Q_0})\geq \wt\nu(2\overline{Q}_0)\geq \mu(\overline{Q}_0)-\frac{C(\mu)}{T^2}\geq 1-\frac{C(\mu)}{T^2}, \\
\bullet \;&\supp(\wt\nu_1)\subset AQ_0, \, \supp(\wt\mu_1)\subset AQ_0\backslash \tfrac{A}{8}Q_0,\\
\bullet \;& D_{\wt\mu_1}(AQ_0)\leq C D_{\wt\mu}(2AQ_0)\leq C(\log A)A^{-\gamma},\\
\bullet \;&D_{\wt\nu_1}(AQ_0)\leq C D_{\wt\nu}(2AQ_0)\leq C.
\end{align*}

It therefore remains to show that the statement (\ref{smoothsmallexcess2}) is absurd given the other properties of $\wt\mu_1$ and $\wt\nu_1$ if $A$ and $T$ are chosen large enough.  For the time being though, let us supply the proofs of Lemmas \ref{Rieszcomp} and \ref{potentialcomp}.

\subsection{The comparison estimates}

Fix a cell $W\in \{W_j\}_j$ and its $\sigma\delta$-neighbourhood $\wt{W}$.  Notice that, since $\wt{W}$ and $\supp(\nu')\backslash W$ are separated sets, the potential
$$ \RSO(\chi_{\R^d\backslash W}\nu')(x) = \int_{\R^d\backslash W} K(x-y) d\nu'(y)
$$
is a continuous function on $\wt{W}$.  Moreover, by Lemma \ref{aedefine}, this potential coincides with $\RSO_{\nu'}(\chi_{\R^d\backslash W})$ $\nu'$-almost everywhere on $W$ (or $\wt{W}$).  We begin with an oscillation estimate.

\begin{lem}\label{nuosc}There is a constant $C>0$ such that for any cell $ W$,
\begin{equation}\begin{split}\nonumber\osc_{\wt W}(\RSO(\chi_{\R^d\backslash W}\nu'))&=\sup_{x,x'\in \wt{W}}|\RSO(\chi_{\R^d\backslash W}\nu')(x)- \RSO(\chi_{\R^d\backslash W}\nu')(x')|\\&\leq \frac{CM^s}{\sigma^s}\kap+\frac{CT}{M}.
\end{split}\end{equation}
\end{lem}

\begin{proof}
First notice that any $x\in \wt W$ is at a distance of at least $\sigma\delta$ from $\supp(\nu')\backslash W$.  Thus, if $x,x'\in \wt W$ and $z\in \supp(\nu')\backslash W$,  we have the trivial estimate $|K(x-z)-K(x'-z)|\leq \tfrac{C}{(\sigma\delta)^s}$.  If in addition $z\in \supp(\nu')\backslash B(x_W, M\delta)$, then we have that $|K(x-z)-K(x'-z)|\leq \tfrac{C\delta}{|x_W-z|^{s+1}}$.

Now fix $x,x'\in \wt W$.  With a view to applying the two kernel bounds of the previous paragraph, we write
\begin{equation}\begin{split}\nonumber|\RSO&(\chi_{\R^d\backslash W}\nu')(x)- \RSO(\chi_{\R^d\backslash W}\nu')(x')|\leq\int_{\R^d\backslash W}|K(x-z)-K(x'-z)|d\nu'(z)\\
&= \int_{B(x_W,M\delta)\backslash W}\cdots \;d\nu'(z)+ \int_{\R^d \backslash B(x_W,M\delta)}\cdots \;d\nu'(z)=I+II.
\end{split}\end{equation}

We estimate $I$ trivially using (\ref{Mdelta}),
$$I\leq \frac{C}{(\sigma\delta)^s}\nu'(B(x_W, M\delta))\leq C\kap \frac{(M\delta)^s}{(\sigma\delta)^s}=C\frac{M^s}{\sigma^s}\kap.
$$
However, since $\overline{D}_{\nu'}(x_W)\leq CT$, we also have
$$II\leq C\int_{\R^d\backslash B(x_W, M\delta)}\frac{\delta}{|x_W-z|^{s+1}}d\nu'(z)\leq C\frac{T}{M},
$$
 and the lemma is proved.\end{proof}

\begin{cor}\label{nuosccor} For any cell $W$, and any $x\in W$, $$|\RSO(\chi_{\R^d\backslash W}\wt\nu)(x)- \RSO(\chi_{\R^d\backslash W}\nu')(x)|\leq \frac{CM^s}{\sigma^s}\kap+\frac{CT}{M}.$$
\end{cor}

\begin{proof} Since $\wt{\nu}=\psi*\nu'$,  the estimate follows immediately from Lemma \ref{nuosc} by noticing that $\RSO(\chi_{\R^d\backslash W}\wt{\nu})(x) = [\psi*\RSO(\chi_{\R^d\backslash W}\nu')](x)$. 
\end{proof}



\begin{lem}\label{nusmoothinside}  There is a constant $C>0$ such that for every cell $W$ and $x\in W$,
$$|\RSO(\chi_{W}\wt\nu)(x)|\leq \frac{C\kap}{\sigma^s}.
$$
\end{lem}

\begin{proof}  We shall estimate the integral $\int_{W}\tfrac{1}{|x-y|^s}d\wt\nu(y)$ for $x\in W$.  By Tonelli's theorem, we see that this integral is bounded by
$$C\int_{W}\Bigl[\dashint_{B(z, \sigma\delta)}\frac{1}{|x-z-y|^s}dm_d(y)\Bigl]d\nu'(z).
$$
The inner integral average has size at most $C(\sigma\delta)^{-s}$, whereas $\nu'(W)\leq C\kap\delta^s$ (using (\ref{Mdelta}), and by combining these estimates the lemma follows.
\end{proof}

\begin{proof}[Proof of Lemma \ref{Rieszcomp}]  First note that by Lemma \ref{nuosc}, we have
$$|\RSO(\chi_{\R^d\backslash W}\nu')(x)|\leq \Bigl|\dashint_{W}\RSO(\chi_{\R^d\backslash W}\nu')d\nu'\Bigl|+\frac{CM^s}{\sigma^s}\kap+\frac{CT}{M}.
$$
for any $x\in W$.  But, due to Lemma \ref{aedefine}, we have that
$$\dashint_{W}\RSO(\chi_{\R^d\backslash W}\nu')d\nu'= \dashint_{W}\RSO_{\nu'}(\chi_{\R^d\backslash W})d\nu'.$$
On the other hand, the antisymmetry of the kernel $K$ yields that
$$\int_{W}\RSO_{\nu'}(\chi_{\R^d\backslash W})d\nu' = \int_{W}\RSO_{\nu'}(1)d\nu',
$$
and so from Corollary \ref{nuosccor},
$$|\RSO(\chi_{\R^d\backslash W}\wt\nu)(x)|\leq \dashint_{W}|\RSO_{\nu'}(1)|d\nu'+\frac{CM^s}{\sigma^s}\kap+\frac{CT}{M}.
$$
Finally, appealing to the estimate for $|\RSO(\chi_W \wt\nu)|$ from Lemma \ref{nusmoothinside} completes the proof.
\end{proof}




We now move on to proving Lemma \ref{potentialcomp}.  For this, we set
$$G_{A,\delta}(\nu)(x) = \frac{1}{A}\int_{\R^d} \frac{1}{\max(3\sqrt{d}\delta, |x-y|)^{s-1}}d\nu(y),
$$
and
$$G^{\delta}_A(\nu) = G_A(\nu)- G_{A,\delta}(\nu).
$$

\begin{lem} \label{Gmuosc} There is a constant $C>0$ such that for each cell $W$, we have

(i)  if $x\in \supp(\nu')\cap \wt{W}$, then $$G_{A}^{\delta}(\mu_A)(x)\leq  \frac{CT\delta}{A}\log A,$$

and

(ii) $\displaystyle\osc_{\wt{W}}[G_{A, \delta}(\mu_A)]\leq \frac{CT\delta}{A}\log^2 \bigl(\frac{A}{\delta}\bigl).$

\end{lem}

\begin{proof}  The estimate (i) follows readily from the estimate $\overline{D}_{\mu_A}(x)\leq CT\log A$ for $x\in \supp(\nu')$:
$$G_{A}^{\delta}(\mu_A)(x)\leq \frac{C}{A}\int_0^{3\sqrt{d}\delta}\frac{\mu_A(B(x,r))}{r^{s-1}}\frac{dr}{r}\leq \frac{CT\delta\log A}{A}.
$$

For (ii), we ape the proof of Lemma \ref{nuosc}.   Notice that if $x,x'\in \wt{W}$, then
\begin{equation}\begin{split}\nonumber|G_{A, \delta}&(\mu_A)(x)-G_{A, \delta}(\mu_A)(x')|\\&\leq \frac{C\mu_A(B(x_W, 3\sqrt{d}\delta))}{A\delta^{s-1}}+\frac{C}{A}\int_{|x_W-y|>\delta}\frac{\delta}{|x_W-y|^s}d\mu_A(y)
\end{split}\end{equation}
The first term here is bounded by $\tfrac{CT\delta}{A}\log A$.  Since $\supp(\mu_A)\subset AQ_0$,  the second term is bounded by $$\frac{C\delta}{A}\int_{\delta}^{CA}D_{\mu_A}(B(x_W, r))\frac{dr}{r}\leq C\log\bigl(\frac{A}{\delta}\bigl)\frac{\delta T\log A}{A},$$ where we have used that $\overline{D}_{\mu_A}(x_W)\leq CT\log A$.
\end{proof}

\begin{cor}\label{gmucompcor} For every cell $W$ and $x\in W$,
$$|G_{A, \delta}(\mu_A)(x)-G_{A, \delta}(\wt\mu)(x)|\leq \frac{C\delta T}{A}\log^2 \bigl(\frac{A}{\delta}\bigl).
$$
\end{cor}

To prove the Corollary, just notice that $G_{A,\delta}(\wt{\mu}) = \psi*G_{A,\delta}(\mu_A)$, and so the estimate follows from the oscillation estimate of statement (ii) in Lemma \ref{Gmuosc}.

Lemma \ref{potentialcomp} now easily follows from Lemma \ref{Gmuosc} and Corollary \ref{gmucompcor}.  Indeed, fix $W$, and first use statement (i) of Lemma \ref{Gmuosc} to see that
$$\dashint_{W}G_A(\mu_A)d\nu'\leq \dashint_{W}G_{A,\delta}(\mu_A)d\nu' +\frac{CT\delta}{A}\log A.
$$
But now for any $x\in W$, statement (ii) of Lemma \ref{Gmuosc} yields
$$\dashint_{W}G_{A,\delta}(\mu_A)d\nu' \leq G_{A,\delta}(\mu_A)(x)+\frac{CT\delta}{A}\log^2\Bigl(\frac{A}{\delta}\Bigl),
$$
and Corollary \ref{gmucompcor} implies that
$$ G_{A,\delta}(\mu_A)(x)\leq G_{A,\delta}(\wt\mu)(x)+\frac{CT\delta}{A} \log^2\Bigl(\frac{A}{\delta}\Bigl).
$$
To complete the proof of the lemma we only need combine these three inequalities and to notice that $G_{A,\delta}(\wt\mu)(x)\leq G_{A}(\wt\mu)(x)$.

\section{Step III:  The variational argument}\label{rearrange}

We continue to work under the assumption that Part (ii) of Alternative \ref{restate} is in force.  With $A$ sufficiently large and $\lambda>0$ arbitrarily small, the smoothing procedure of the previous section provides us with measures $\wt{\mu}$ and $\wt\nu$ (we relabel the measure $\wt\mu_1$ of the previous section by $\wt\mu$ and the measure $\wt\nu_1$ by $\wt\nu$) with $C^{\infty}$ densities with respect to Lebesgue measure satisfying \begin{enumerate}
\item $\supp(\wt\nu)\subset AQ_0$, $\wt\nu(\overline{Q}_0)\geq \tfrac{1}{2}$, and $\wt\nu(AQ_0)\leq CA^s$,
\item $\supp(\wt\mu)\subset AQ_0\backslash \tfrac{A}{8}Q_0$, and $D_{\wt\mu}(AQ_0)\leq \tau$, where $\tau = C(\log A)A^{-\gamma}$,
\item the inequality (\ref{smoothsmallexcess2}), that is,
\begin{equation}\begin{split}\label{restatesmoothsmall}\int_{\R^d}\Bigl[|\RSO(\wt\nu)|&-G_A(\wt\mu)-\tau\Bigl]_+^p d\wt\nu \leq \lambda.
\end{split}\end{equation}
\end{enumerate}

Instead of the integral inequality (\ref{restatesmoothsmall}), we would like to have that $(|\RSO(\wt\nu)|-G_A(\wt\mu)-\tau)_+$ is \emph{pointwise} very small on $\supp(\wt\nu)$.   This is of course not necessarily the case, but a simple variational argument from \cite{ENV} shows that if one is willing to redistribute $\wt\nu$ across its support, then one can indeed obtain such a conclusion.



\subsection{The Functional.}  For $a\in L^{\infty}(m_d)$, $a\geq 0$, define $\wt\nu_a = a\wt\nu$.  Set $$H_a= (|\RSO(\wt\nu_a)|-G_A(\wt\mu)-\tau)_+.
$$Consider the functional
$$\textbf{I}(a) = \int_{\R^d}H_a^pd\wt\nu_a+\frac{\lambda}{\wt\nu_a(\overline{Q}_0)}.
$$

Next, set
$$m = \inf\limits_{\substack{a\geq 0:\,\\\|a\|_{L^{\infty}(m_d)}\leq 1}}\textbf{I}(a).
$$

We shall first show that a minimizer exists.  Certainly $m$ is finite.  The Banach-Alaoglu theorem ensures that there is a minimizing sequence $(a_k)_k\in L^{\infty}(m_d)$, with $a_k\geq 0$ and $\|a_k\|_{L^{\infty}(m_d)}\leq 1$, such that $\textbf{I}(a_k)\rightarrow m$ and $a_k$ converges weakly over $L^1(m_d)$ to a function $a\in L^{\infty}(m_d)$ with $\|a\|_{L^{\infty}(m_d)}\leq1$.  Notice that $a\geq 0$.

Since $\wt\nu$ has a bounded compactly supported density with respect to $m_d$, we see that
$$\wt\nu_{a_k}(\overline{Q}_0)\rightarrow \wt\nu_{a}(\overline{Q}_0).
$$
On the other hand, since $\sup_{x\in \R^d}\int_{\R^d}\tfrac{1}{|x-y|^s}d\wt\nu(y)<\infty$, we have that $\sup_k\|H_{a_k}\|_{L^{\infty}(m_d)}<\infty.$  In addition, for any $x\in \R^d$, the function $y\mapsto \tfrac{x-y}{|x-y|^{s+1}}\frac{d\wt\nu}{dm_d}(y)\in L^1(m_d)$, and so $\RSO(\wt\nu_{a_k})\rightarrow \RSO(\wt\nu_a)$ pointwise.  As a consequence of these two facts, $H_{a_k}^p\rightarrow H_a^p$ in $L^1(\chi_{AQ_0}m_d)$ as $k\to\infty$.  But now,
\begin{equation}\begin{split}\nonumber\Bigl|\int_{\R^d}H_{a_k}^pd\wt\nu_{a_k}- \int_{\R^d}H_{a}^pd\wt\nu_{a}\Bigl|\leq & \int_{\R^d}|H_{a_k}^p-H_a^p|d\wt\nu_{a_k} + \Bigl|\int_{\R^d}H_a^pd(\wt\nu_{a_k}-\wt\nu_a)\Bigl|,
\end{split}\end{equation}
and
\begin{equation}\begin{split}\nonumber\int_{\R^d}|H_{a_k}^p-H_a^p|d\wt\nu_{a_k}\leq \big\|\tfrac{d\wt\nu}{dm_d}\big\|_{L^{\infty}(m_d)}\int_{AQ_0}|H_{a_k}^p-H_a^p|dm_d\rightarrow 0 \text{ as }k\to\infty,
\end{split}\end{equation}
while, as $H_a^p\tfrac{d\wt\nu}{dm_d}\in L^{1}(m_d)$, the weak convergence of $a_k$ to $a$ yields that
$$
\Bigl|\int_{\R^d}H_a^pd(\wt\nu_{a_k}-\wt\nu_a)\Bigl|\to 0\text{ as }k\to\infty.
$$
We conclude that $I(a_k)\rightarrow I(a)$ as $k\to\infty$.

Of course, the constant function $a=1$ is admissible for the minimization problem, and so we see that
$$\textbf{I}(a)=m\leq \textbf{I}(1) = \int_{\R^d}(|\RSO(\wt\nu)|-G_A(\wt\mu)-\tau)_+^p d\wt\nu + \frac{\lambda}{\wt\nu(\overline{Q}_0)}\leq \lambda+2\lambda=3\lambda.$$
By considering each term in $\textbf{I}(a)$ separately, this inequality yields that
\begin{equation}\label{aenergysmall}
\int_{\R^d}H_a^p d\wt\nu_a\leq 3\lambda,
\end{equation}
and
\begin{equation}\label{ameasurelarge}
\wt\nu_a(\overline{Q}_0)\geq \frac{1}{3}.
\end{equation}

\subsection{The First Variation}  Let us now take the first variation of the functional $\textbf{I}$.  Consider a Borel set $U$ with $\wt\nu_a(U)>0$.   For $t\in (0,1)$, consider the function $a_t= a-ta\chi_U$.  The function $a_t$ is admissible for the minimization problem, and so
$$\textbf{I}(a_t)-\textbf{I}(a)\geq 0 \text{ for all }t\in (0,1).$$

Let us first consider the increment of the second term in the definition of the functional $\textbf{I}$.  For this, we notice that
$$\frac{1}{\wt\nu_{a_t}(\overline{Q}_0)} = \frac{1}{\wt\nu_a(\overline{Q}_0)}\Bigl(1+t\frac{\wt\nu_a(U\cap \overline{Q}_0)}{\wt\nu_a(\overline{Q}_0)}+o(t)\Bigl)\text{ as } t\to 0^+.
$$
Combined with (\ref{ameasurelarge}), this gives us that
\begin{equation}\label{linetermmeas}\frac{\lambda}{\wt\nu_{a_t}(\overline{Q}_0)}  - \frac{\lambda}{\wt\nu_a(\overline{Q}_0)} \leq 9 \lambda t\wt\nu_a(U)+o(t) \text{ as }t\to 0^+.
\end{equation}

Next we calculate the first order increment of the first term in the expression for $\textbf{I}(a)$.  For this, we claim that there is a Borel measurable unit vector field $E$ such that the function
$$\frac{1}{t}\Bigl[H^p_{a_t}-H^p_a-tpH_a^{p-1}\langle E, \RSO(\chi_U \wt\nu_a)\rangle\Bigl]
$$
converges to zero uniformly on $\R^d$ as $t\rightarrow 0^+$.

To see this, recall that both $\RSO(\wt\nu_a)$ and $\RSO(\chi_U \wt\nu_a)$ are bounded continuous functions, and  $\RSO(\wt\nu_{a_t})=\RSO(\wt\nu_a)-t\RSO(\chi_U \wt\nu_a)$.  Now, if at some point $x\in \R^d$ it holds that $|\RSO(\wt\nu_a)|\leq \tfrac{\tau}{2}$, then $H_a=0$ and  if $t>0$ is small enough (depending only on the uniform bound for $\RSO(\chi_U \wt\nu_a)$), then we have that $H_{a_t}=0$.

Otherwise $|\RSO(\wt\nu_a)|\geq \tfrac{\tau}{2}$ at $x$, and we can write
$$|\RSO(\wt\nu_{a_t})| = |\RSO(\wt\nu_a)|-t\Bigl\langle \frac{\RSO(\wt\nu_a)}{|\RSO(\wt\nu_a)|},  \RSO(\chi_U \wt\nu_a)\Bigl\rangle+ O\Bigl(\frac{t^2|\RSO(\chi_U\wt\nu_a)|^2}{\tau}\Bigl)
$$
as $t\to 0^+$.  But then since the function $u\to  u^p_+$ has locally uniformly continuous derivative on $\R$, we see that as $t\to 0^+$,
$$H_{a_t}^p-H_a^p =-pH_a^{p-1}t\Big\langle \frac{\RSO(\wt\nu_a)}{|\RSO(\wt\nu_a)|},  \RSO(\chi_U \wt\nu_a)\Big\rangle+ o(t),
$$
and the claim follows with the Borel measurable unit vector field $$E = -\frac{\RSO(\wt\nu_a)}{|\RSO(\wt\nu_a)|}\chi\ci{\{|\RSO(\wt\nu_a)|\geq \tfrac{\tau}{2}\}} + \textbf{e}\chi\ci{\{|\RSO(\wt\nu_a)|< \tfrac{\tau}{2}\}},$$
for any fixed unit vector $\textbf{e}$.



From the claim we find that as $t\rightarrow 0^+$,
\begin{equation}\label{Hexpansion} \int_{\R^d}[H_{a_t}^p - H_{a}^p]d\wt\nu_a = tp\int_{\R^d}H_a^{p-1}\langle E, \RSO(\chi_U \wt\nu_a)\rangle d\wt\nu_a +o(t).
\end{equation}
Additionally, we also have that
$$\int_{\R^d}H_{a_t}^p d\wt\nu_{a_t} = \int_{\R^d}H_{a_t}^p d\wt\nu_{a} - t\int_U H_{a_t}^p d\wt\nu_{a} = \int_{\R^d}H_{a_t}^p d\wt\nu_{a} -t\int_U H_{a}^p d\wt\nu_{a}+o(t),
$$
where the second equality follows from noticing that the expression $H_a^{p-1}\langle E, \RSO(\chi_U \wt\nu_a)\rangle$ is a bounded function.

Notice that the integral appearing on the right hand side of (\ref{Hexpansion}) can be rewritten as
$$\int_{\R^d}H_a^{p-1}\langle E, \RSO(\chi_U \wt\nu_a)\rangle d\wt\nu_a = -\int_{U}\RSO^*(H_a^{p-1}E\wt\nu_a)d\wt\nu_a.$$
So see that the increment in $t$ in $\textbf{I}(a_t)-\textbf{I}(a)$ does not exceed
$$t\Bigl[9\lambda \wt\nu_a(U) - \int_U H_a^p d\wt\nu_a - p\int_{U}\RSO^*(H_a^{p-1}E\wt\nu_a)d\wt\nu_a\Bigl] +o(t)
$$
as $t\to 0^+$.  The minimizing property therefore yields that
$$\int_U \bigl[H_a^p  + p\RSO^*(H_a^{p-1}E\wt\nu_a)\bigl]d\wt\nu_a\leq 9\lambda\wt\nu_a(U),
$$
for every Borel set $U\subset \R^d$  with $\wt\nu_a(U)>0$.  But $H_a^p + p\RSO^*(H_a^{p-1}E\wt\nu_a)$ is a continuous function, so we arrive at the following lemma.

\begin{lem}[The First Variation]\label{firstvar}  There is a Borel measurable unit vector field $E$ such that on $\supp(\wt\nu_a)$ we have
\begin{equation}\label{Hptwisebd}H_a^p  + p\RSO^*(H_a^{p-1}E\wt\nu_a) \leq 9\lambda.
\end{equation}
\end{lem}

But now we may apply the maximum principle for the fractional Laplacian\footnote{We reiterate that this is the underlying reason behind the restriction to $s\in (d-1,d)$ throughout the paper.} in the form of Lemma \ref{finalmaxprinc}.  This yields that the inequality (\ref{Hptwisebd}) in fact holds throughout $\R^d$.

\section{Contradiction}\label{contradiction}

We continue to work under the assumption that statement (ii) of Alternative \ref{restate} holds.  Then with $A$ sufficiently large, and $\lambda$ as small as we wish, there are two finite measures $\wt\mu$ and $\wt\nu_a$ that have bounded densities with respect to $m_d$, are both supported in $AQ_0$, and satisfy the following properties:
\begin{enumerate}
\item $\wt\nu_a(\overline{Q}_0)\geq \tfrac{1}{3}$, and $\wt\nu_a(AQ_0)\leq CA^s$,
\item $\supp(\wt\mu)\subset AQ_0\backslash \tfrac{A}{8}Q_0$, and $D_{\wt\mu}(AQ_0)\leq \tau$, where $\tau = C(\log A)A^{-\gamma}$,
\item with $H_a= (|\RSO(\wt\nu_a)|-G_A(\wt\mu)-\tau)_+,$ we have that
\begin{equation}\label{aenergysmall1}
\int_{\R^d}H_a^p d\wt\nu_a\leq 3\lambda,
\end{equation}
and for some Borel measurable unit vector field $E$,
\begin{equation}\label{Hptwisebd1}H_a^p  + p\RSO^*(H_a^{p-1}E\wt\nu_a) \leq 9\lambda \text{ in } \R^d.
\end{equation}
\end{enumerate}

We shall show that this is preposterous if $A$ is large enough and $\lambda$ is small enough.  This will force us into part (i) of Alternative \ref{restate}, which is our desired result.

\subsection{The $\Psi$-function}\label{Fourier}

Fix a non-negative function $f\in C^{\infty}_0(\R^d)$ so that $f\equiv 1$ on $\overline{Q}_0$, while $\supp(f)\subset 2Q_0$.

We define $\psi$ via its Fourier Transform:
\begin{equation}\label{psiFourier}
\wh\psi(\xi)=b\bigl[\xi|\xi|^{d-1-s}\wh{f}(\xi)\bigl],
\end{equation}
where the constant $b\in \mathbb{C}\backslash \{0\}$ has been chosen to ensure that\footnote{Recall that, if $g$ is a smooth vector valued function with suitable decay, then $\wh{\RSO^*(g m_d)}(\xi) = b'\tfrac{1}{|\xi|^{d-s+1}}\xi\cdot \wh g(\xi)$ for some $b'\in \mathbb{C}$.} $\RSO^*(\psi m_d) = f$.  From (\ref{psiFourier}), it is a standard exercise in Fourier analysis to show that there is a constant $C>0$ such that $$\Psi(x):=|\psi(x)|\leq \frac{C}{(1+|x|)^{2d-s}}\text{ for every }x\in \R^d.$$
(For instance, see Lemma \ref{fourdecay} of the appendix.)  In particular, we observe that $\Psi\in L^1(m_d)$ with norm at most some constant depending on $d$ and $s$.

Notice that
\begin{equation}\begin{split}\nonumber \frac{1}{3}&\leq \wt\nu_a(\overline{Q}_0)\leq \int_{\R^d}fd\wt\nu_a = \int_{\R^d} \RSO^*(\psi m_d)d\wt\nu_a \\
&= \int_{\R^d} \langle \RSO(\wt\nu_a), \psi\rangle dm_d\leq \int_{\R^d}|\RSO(\wt\nu_a)|\Psi dm_d.
\end{split}\end{equation}
Our task is to show that the right hand side of this inequality is smaller than $\tfrac{1}{3}$ with a decent choice of $\tau$ and $\lambda$.

To obtain this contradiction, notice that for any $\alpha>0$,
\begin{equation}\nonumber
|\RSO(\wt\nu_a)|\leq \tau + G_A(\wt\mu)+H_a\leq \tau +G_{A}(\wt\mu)+\frac{1}{p'}\lambda^{\alpha p'}+\frac{1}{p}\frac{H_a^p}{\lambda^{\alpha p}},
\end{equation}
where the second inequality follows from Young's inequality ($ab\leq \tfrac{a^{p'}}{p'}+\tfrac{b^{p}}{p}$ for $a,b\geq 0$).

But now using (\ref{Hptwisebd1}), we get that
\begin{equation}\begin{split}\label{Rnuabd}
|\RSO(\wt\nu_a)|\leq  \tau +G_{A}(\wt\mu)+\lambda^{\alpha p'}+9\lambda^{1-\alpha p}-\lambda^{-\alpha p}\RSO^*(H_a^{p-1}E\wt\nu_a)
\end{split}\end{equation}
in $\R^d$.  We shall estimate each term on the right hand side of this inequality when integrated against the measure $\Psi m_d$.

First noting that $G_A(\wt\mu)\leq C\tau $ on $\tfrac{A}{16}Q_0\cup (\R^d\backslash 2AQ_0)$ and $\Psi\leq \tfrac{C}{A^{2d-s}}\leq \tfrac{C}{A^d}$ on $2AQ_0\backslash \tfrac{A}{16}Q_0$, we get
$$\int_{\R^d}G_{A}(\wt\mu)\Psi dm_d\leq C\tau\int_{\R^d}\Psi dm_d+\frac{C}{A^d}\int_{2AQ_0}G_{A}(\wt\mu)dm_d.
$$
By Tonelli's theorem
$$\frac{1}{A^d}\int_{2AQ_0}G_{A}(\wt\mu)dm_d\leq \frac{1}{A^d}\int_{AQ_0}\int_{2AQ_0}\frac{1}{A|x-y|^{s-1}}dm_d(x)d\wt\mu(y).
$$
Since we have the straightforward bound $\int_{2AQ_0}\frac{1}{A|x-y|^{s-1}}dm_d(x) \leq CA^{d-s}$, we bound the right hand side of the previous inequality by $D_{\wt\mu}(AQ_0)\leq \tau$.  Consequently, we see that
$$
\int_{\R^d}G_{A}(\wt\mu)\Psi dm_d\leq C\tau.
$$

On the other hand, notice that
$$\int_{\R^d}\RSO^*(H_a^{p-1}E\wt\nu_a)\Psi dm_d = -\int_{\R^d}H_{a}^{p-1}\langle E, \RSO(\Psi m_d)\rangle d\wt\nu_a,
$$
but since $\|\RSO(\Psi m_d)\|_{L^{\infty}(m_d)}\leq C$,  we get from (\ref{aenergysmall1}) that
$$\Bigl|\int_{\R^d}\RSO^*(H_a^{p-1}E\wt\nu_a)\Psi dm_d\Bigl|\leq C\wt\nu_a(\R^d)^{1/p}\|H_a\|_{L^p(\wt\nu_a)}^{(p-1)/p}\leq CA^{s/p}\lambda^{(p-1)/p},$$
here we have also used that $E$ is a unit vector field.

Bringing our estimates together, we see that there is a constant $C>0$ such that for any $\alpha>0$,
$$\frac{1}{3}\leq  C\bigl(\tau+\lambda^{\alpha p'}+\lambda^{1-\alpha p}+ \lambda^{-\alpha p}\lambda^{(p-1)/p}A^{s/p}\bigl)
$$
Now fix $\alpha <\tfrac{p-1}{p^2}$.   Let us fix $A$ so large (depending only on $d$ and $s$) that we have $C\tau\leq \tfrac{1}{6}$.  Then for arbitrarily small $\lambda>0$,
$$\frac{1}{6}\leq C\bigl(\lambda^{\alpha p'}+\lambda^{1-\alpha p}+ \lambda^{-\alpha p}\lambda^{(p-1)/p}A^{s/p}\bigl).
$$
However, one only needs to choose $\lambda$ to be smaller than some large negative power of $A$ to make this absurd.  We conclude that statement (ii) of Alternative \ref{restate} cannot hold true for large enough $A$.

This completes the proof of Theorem \ref{thm}.

\appendix

\part*{Appendices}

\section{The maximum principle}\label{mpsec}

In this appendix we review the maximum principle for the fractional Laplacian operator.  The standard reference for potential theory for the fractional Laplacian is Landkof's book \cite{Lan}.

Fix $s\in (d-1,d)$.  Set $\alpha = d-s+1$.  Consider the $\alpha$-Poisson kernel of the ball $B(0,r)$:
$$P^{\alpha}_r(x) =\begin{cases} \Gamma(\tfrac{d}{2})\tfrac{1}{\pi^{d/2+1}}\sin\bigl(\tfrac{\pi\alpha}{2}\bigl)r^{\alpha}\bigl(|x|^2-r^2\bigl)^{-\alpha/2}|x|^{-d} \text{ if }|x|\geq r\\0 \text{ if }|x|<r. \end{cases}
$$

Set $k_{\alpha}(x) = \tfrac{1}{|x|^{d-\alpha}}( = \tfrac{1}{|x|^{s-1}})$. The following three properties of the $\alpha$-Poisson kernel are proved in an appendix in Landkof \cite{Lan} (see also p.112).

\begin{enumerate}
\item $P^{\alpha}_r*k_{\alpha}(x) = k_{\alpha}(x) \text{ for }|x|\geq r$,
\item $P^{\alpha}_r*k_{\alpha}(x) < k_{\alpha}(x) \text{ for }|x| < r$,
\item $\int_{\R^d}P^{\alpha}_r(x)dm_d(x)=1$.
\end{enumerate}

For a signed measure $\nu$, set $I_{\alpha}(\nu) = k_{\alpha}*\nu$.

Note that properties (1) and (2) combine to yield that for any finite (positive) measure $\mu$,
$$I_{\alpha}(\mu)(x) \geq P^{\alpha}_r*(I_{\alpha}(\mu))(x) \text{ for any }x\in \R^d \text{ and }r>0,
$$
while if $\nu$ is a finite signed measure, and $\dist(x, \supp(\nu))>r$, then $I_{\alpha}(\nu)(x) = [P^{\alpha}_r*I_{\alpha}(\nu)](x)$.

\begin{lem}  Suppose that $\nu$ is a finite signed (vector) measure in $\R^d$.  Let $x\in\R^d\backslash \supp(\nu)$ and $r<\dist(x,\supp(\nu))$. Then
$$\RSO^*(\nu)(x) = [P^{\alpha}_r*(\RSO^*(\nu)](x).
$$
\end{lem}

\begin{proof}Let $\eps>0$ and choose $\varphi_{\eps}\in C^{\infty}_0(B(0,\eps))$ with $\int_{\R^d}\varphi_{\eps}dm_d=1$.  Then notice that
$$\varphi_{\eps}*\RSO^*(\nu)(x) = \RSO^*(\varphi_{\eps}*\nu)(x) = b_{\alpha}I_{\alpha}(\text{div}(\varphi_{\eps}*\nu))(x), $$
for some $b_{\alpha}\in \mathbb{R}\backslash\{0\}$.   But if $\eps<\dist(x, \supp(\nu))-r$, then
\begin{equation}\begin{split}\nonumber b_{\alpha}I_{\alpha}(\text{div}(\varphi_{\eps}*\nu))(x) & = b_{\alpha}P^{\alpha}_r*I_{\alpha}(\text{div}(\varphi_{\eps}*\nu))(x) \\& =P^{\alpha}_r*[\RSO^*(\varphi_{\eps}*\nu)](x) \\&= [\varphi_{\eps}*P^{\alpha}_r*\RSO^*(\nu)](x).
\end{split}\end{equation}
Consequently $\varphi_{\eps}*\RSO^*(\nu)(x) =[\varphi_{\eps}*P^{\alpha}_r*\RSO^*(\nu)](x)$, and letting $\eps\rightarrow 0^+$ proves the lemma.
\end{proof}

Let us recall that if $\nu$ is a finite signed measure with bounded density with respect to $m_d$, then $\RSO(\nu)$ is a bounded continuous function on $\R^d$ and $\lim_{|x|\rightarrow \infty}|\RSO(\nu)(x)|=0$.

\begin{lem}[Strong Maximum Principle]\label{strongmaxprinc}  Suppose that $\nu$ is a finite signed vector measure, and $\mu$ is a finite (positive) measure, both of which have bounded density with respect to $m_d$.  If $\RSO^*(\nu) - I_{\alpha}(\mu)$ attains a point of global maximum outside of $\supp(\nu)$, then $\RSO^*(\nu)-I_{\alpha}(\mu)$ is constant in $\R^d$.
\end{lem}

\begin{proof}  Suppose that $\RSO^*(\nu) - I_{\alpha}(\mu)$ attains its maximum at some $x\not\in \supp(\nu)$.  For any $r<\dist(x, \supp(\nu))$, we then have
\begin{equation}\begin{split}\nonumber\RSO^*(\nu)(x) - I_{\alpha}(\mu)(x) & = [P^{\alpha}_r*\RSO^*(\nu)](x) - I_{\alpha}(\mu)(x)\\&\leq P^{\alpha}_r*[\RSO^*(\nu) - I_{\alpha}(\mu)](x).
\end{split}\end{equation}
But since the maximum of $\RSO^*(\nu) - I_{\alpha}(\mu)$ is attained at $x$, property (3) of the non-negative kernel $P^{\alpha}_r$ ensures that $\RSO^*(\nu) - I_{\alpha}(\mu)$ is constant in $\R^d\backslash B(0,r)$.  Since $r< \dist(x,\supp(\nu))$ was chosen arbitrarily, the lemma follows.  \end{proof}

\begin{lem}\label{finalmaxprinc}  Suppose that $\mu$ is a finite positive measure, $\nu$ is a finite signed measure, and $\wt{\nu}$ is a finite signed vector measure with $\supp(\wt{\nu})\subset \supp(\nu)$, with all three measures having bounded density with respect to $m_d$. Then for any $\tau>0$,
\begin{equation}\begin{split}\sup_{\R^d}\bigl[ (|\RSO(\nu)|&-I_{\alpha}(\mu)-\tau)_+^p-\RSO^*(\wt{\nu})\bigl]\\
&\leq\max\bigl(0,\sup_{\supp(\nu)}\bigl[ (|\RSO(\nu)|-I_{\alpha}(\mu)-\tau)_+^p-\RSO^*(\wt{\nu})\bigl]\bigl).
\end{split}\end{equation}
\end{lem}

\begin{proof}
The convex function $v:\R\rightarrow [0, \infty)$, $v(t)=t_+^p$ can be represented by the formula
$$v(t) = \max_{\lambda\geq 0}\{ \lambda t - v^*(\lambda)\}, \text{ for }t\in \R,
$$
where $v^*:[0,\infty)\rightarrow [0, \infty)$, $v^*(t) = \tfrac{p}{p'p^{p'}}t^{p'}$ is the Legendre transform of $v$ (the exact form of $v^*$ is not important).

Now suppose that the continuous function $H = (|\RSO(\nu)|-I_{\alpha}(\mu)-\tau)_+^p-\RSO^*(\wt{\nu})$ has a positive supremum on $\R^d$ (otherwise the result is proved).  Because $H$ tends to zero at infinity, it attains its maximum at some $x_0\in \R^d$.  Then for some $e\in \mathbb{S}^{d-1}$ and $\lambda\geq 0$,
$$ H(x_0) =\lambda\langle \RSO(\nu)(x_0), e\rangle   - \lambda I_{\alpha}(\mu)(x_0) - \lambda \tau - v^*(\lambda) - \RSO^*(\wt{\nu})(x_0),
$$
while,
$$\lambda\langle \RSO(\nu), e\rangle   - \lambda I_{\alpha}(\mu) - \lambda \tau  - v^*(\lambda)- \RSO^*(\wt{\nu})\leq H \text{ on }\R^d.
$$
But now by writing
\begin{equation}\begin{split}\nonumber \lambda \langle \RSO(\nu), e\rangle  &  - \RSO^*(\wt{\nu}) = \RSO^*(-\wt\nu - \lambda \nu e)
\end{split}\end{equation}
we obtain from Lemma \ref{strongmaxprinc} that
\begin{equation}\begin{split}\nonumber H(x_0)&= \RSO^*(-\wt\nu -  \lambda \nu e)(x_0) -\lambda I_{\alpha}(\mu)(x_0) - \lambda \tau-v^*(\lambda)\\
& \leq \sup_{\supp(\nu)}\bigl[\RSO^*(-\wt\nu - \lambda \nu e) - \lambda I_{\alpha}(\mu) - \lambda \tau-v^*(\lambda)\bigl]\\
&\leq \sup_{\supp(\nu)}H.
\end{split}\end{equation}
The result follows.
\end{proof}

\section{The small boundary mesh}\label{smallbdary}

In this appendix we show how to find a small boundary mesh relative to a finite measure $\nu$.  For $\delta>0$ and $\sigma\in (0,1)$, we want to find a mesh of cubes of sidelength $\delta$ so that the $\sigma\delta$ neighbourhood of the boundary of the cubes has measure at most $C\sigma\nu(\R^d)$, where $C>0$ depends only on the dimension.

For a coordinate  $j\in \{1,\dots,d\}$, set
$$E_j = \R^{j-1}\times \bigcup_{k\in \mathbb{Z}} [k\delta-\sigma\delta, k\delta+\sigma\delta]\times \R^{d-j}.
$$

Fix $M>0$.  Notice that from the union bound and Chebyshev's inequality, we have
\begin{equation}\begin{split}\nonumber\frac{1}{\delta^d}m_d\bigl(\bigl\{t\in [0,\delta]^d: \, &\nu\bigl(\bigcup_{j=1}^d(E_j+t_je_j)\bigl)>M\sigma\nu(\R^d)\bigl\}\bigl)\\&\leq \frac{1}{M\sigma\delta\nu(\R^d)}\sum_{j=1}^d\int_0^{\delta}\nu(E_j+t_je_j)dt_j.
\end{split}\end{equation}
But now
$$\int_0^{\delta}\nu(E_j+t_je_j)dt_j  =\int_{\R^d}\int_0^{\delta}\chi_{E_j+t_je_j}(x)dt_jd\nu(x).
$$
Any fixed $x\in \R^d$ can lie in the set $E_j+t_je_j$ only if $t_j$ lies in the union of two intervals of total width $2\sigma\delta$.  Thus $\int_0^{\delta}\nu(E_j+t_je_j)dt_j\leq 2\sigma\delta\nu(\R^d)$.

We conclude that
$$\frac{1}{\delta^d}m_d\bigl(\bigl\{t\in [0,\delta]^d: \, \nu\bigl(\bigcup_{j=1}^d(E_j+t_je_j)\bigl)>M\sigma\nu(\R^d)\bigl\}\bigl)\leq \frac{2d}{M}.
$$

But since the $\sigma\delta$ neighbourhood of the boundary of any $\delta$-mesh of cubes is contained in the union $\bigcup_{j=1}^d(E_j+t_je_j)$ for some $t\in [0,\delta]^d$,
we see that there must exists some $\delta$-mesh whose $\sigma\delta$ neighbourhood has $\nu$ measure at most $3d\sigma\nu(\R^d)$ (just set $M=3d$).





\section{Lipschitz continuous solutions of the fractional Laplacian equation}\label{distributiontheory}

We denote by $\mathcal{S}(\R^d)$ the space of Schwartz class functions on $\R^d$, and set $\mathcal{S}'(\R^d)$ to be the space of tempered distributions.

We define the action of the fractional Laplacian $(-\Delta)^{\alpha/2}$ on a Schwartz class function $f\in \mathcal{S}(\R^d)$ via the Fourier transform:
$$(-\Delta)^{\alpha/2}f = \mathcal{F}^{-1}(|\xi|^{\alpha}\wh{f}\,(\xi)).
$$
To show that the fractional Laplacian of a Lipschitz continuous function can be interpreted as a tempered distribution, we shall require a standard lemma.

\begin{lem} \label{fourdecay} Fix $\beta>0$.  Suppose that $p\in C^{\infty}(\R^d\backslash\{0\})$ satisfies, for every multi-index $\gamma$,
$$|D^{\gamma}p(\xi)|\leq C_{\gamma}|\xi|^{\beta-|\gamma|} \text{ for every }\xi\in \R^d\backslash \{0\}.
$$
Then, for every $f\in \mathcal{S}(\R^d)$, the function $\mathcal{F}^{-1}(p\wh{f})\in C^{\infty}(\R^d)$, and
$$|D^{\gamma}\mathcal{F}^{-1}(p\wh{f})(x)|\leq \frac{C_{\gamma}}{(1+|x|)^{d+\beta+|\gamma|}}
$$
for every $x\in \R^d$ and multi-index $\gamma$.
\end{lem}

\begin{proof} The fact that $\mathcal{F}^{-1}(p\wh{f})$ is a smooth bounded function merely follows from the fact that the function $\xi\mapsto |\xi|^m \wh{f}(\xi)$ lies in $L^1(m_d)$ for every $m \geq 0$.  We shall prove that $|\mathcal{F}^{-1}(p\wh{f})(x)|\leq \tfrac{C}{|x|^{d+\beta}}$ for $|x|>1$.  The estimate for the derivatives follows in the same manner.

Fix $k_0\in \mathbb{N}$ with $2^{-k_0-1}\leq \tfrac{1}{|x|}<2^{-k_0}.$  Suppose that $(\eta_j)_{j\geq 0 }$ is a partition of unity in the frequency space such that $\sum_{j\geq 0}\eta_j \equiv 1$ on $\R^d$,
$$\supp(\eta_0) \subset B(0, 2^{-k_0+1}) \text{ and }\supp(\eta_j) \subset B(0, 2^{-k_0+j+1})\backslash B(0, 2^{-k_0+j-1}),
$$
and $$|D^{\gamma}\eta_j|\leq \frac{C(\gamma)}{2^{(-k_0+j)|\gamma|}} \text{ on }\R^d,$$
for every multi-index $\gamma$.

First note that
$$|\mathcal{F}^{-1}(p\wh{f}\eta_0)(x)|\leq C\int_{B(0, 2^{-k_0+1})}|p||\wh{f}|dm_d\leq C2^{-k_0(d+\beta)}\leq \frac{C}{|x|^{d+\beta}}.
$$

Now fix $j\geq 1$.  Notice that, for $m\in \mathbb{N}$,
$$|x|^{2m}|\mathcal{F}^{-1}(p\wh{f}\eta_j)(x)|=c_m |\mathcal{F}^{-1}(\Delta^m(p\wh{f}\eta_j))(x)|.
$$
But, for any multi-index $\gamma$, we have that $|D^{\gamma}\wh{f}|\leq C_{\gamma}2^{-|\gamma|(-k_0+j)}$, and $|D^{\gamma}p|\leq C_{\gamma}2^{(\beta-|\gamma|)(-k_0+j)}$, on $\supp(\eta_j)$.  Thus
\begin{equation}\begin{split}\nonumber|x|^{2m}|\mathcal{F}^{-1}\bigl(p\wh{f}\eta_j\bigl)(x)|&\leq C_m\int_{B(0, 2^{-k_0+j+1})\backslash B(0, 2^{-k_0+j-1})}2^{(-k_0+j)(\beta-2m)}dm_d\\&\leq C_m2^{(-k_0+j)(d+\beta-2m)}.
\end{split}\end{equation}
But now, if $m>\tfrac{d+\beta}{2}$, then
$$\nonumber|x|^{2m}|\mathcal{F}^{-1}(p\wh{f}\sum_{j\geq 0}\eta_j)(x)|\leq C_m\sum_{j\geq 0}2^{(-k_0+j)(d+\beta-2m)}\leq C_m|x|^{2m-d-\beta}.
$$
The lemma follows.
\end{proof}

It is an immediate consequence of the lemma that, if $f\in \mathcal{S}(\R^d)$, then for any multi-index $\gamma$,
\begin{equation}\label{fracschwartzdecay}|D^{\gamma}[(-\Delta)^{\alpha/2} f](x)|\leq \frac{C_{\gamma}}{(1+|x|)^{d+\alpha+|\gamma|}} \text{ for every }x\in \R^d.
\end{equation}
As such, when considering a class of generalized functions for which a distributional notion of the fractional Laplacian may be defined, a natural class of smooth functions presents itself.  Denote by $\mathcal{S}_{\alpha}(\R^d)$ the class of smooth functions $g\in C^{\infty}(\R^d)$ such that for every multi-index $\gamma$ there is a constant $C_{\gamma}>0$ such that
$$|(D^{\gamma}g)(x)|\leq \frac{C_{\gamma}}{(1+|x|)^{d+\alpha+|\gamma|}} \text{ for all }x\in \R^d.
$$
In this language, we may restate (\ref{fracschwartzdecay}) as \begin{equation}\label{relabelschwartz}(-\Delta)^{\alpha/2}\mathcal{S}(\R^d)\subset \mathcal{S}_{\alpha}(\R^d).\end{equation} The space of distributions acting on $\mathcal{S}_{\alpha}(\R^d)$ is denoted by $\mathcal{S}'_{\alpha}(\R^d)$.

For $F\in \mathcal{S}'_{\alpha}(\R^d)$, we may define $(-\Delta)^{\alpha/2}F \in \mathcal{S}'(\R^d)$ by
$$\langle (-\Delta)^{\alpha/2}F, f\rangle = \langle F, (-\Delta)^{\alpha/2}f\rangle, \; f\in \mathcal{S}(\R^d).
$$

Notice that if $u$ satisfies $\int_{\R^d}\frac{|u(x)|}{(1+|x|)^{d+\alpha}}dm_d(x)<\infty$, then $(-\Delta)^{\alpha/2}u$ is the tempered distribution given by the absolutely convergent integral
$$\langle (-\Delta)^{\alpha/2}u, f\rangle = \int_{\R^d} u [(-\Delta)^{\alpha/2}f] dm_d \text{ for }f\in \mathcal{S}(\R^d).
$$

We shall henceforth assume that $\alpha>1$.  Under this assumption, note that if $u\in \Lip(\R^d)$ then the condition $\int_{\R^d}\frac{|u(x)|}{(1+|x|)^{d+\alpha}}dm_d(x)<\infty$ holds.

We call a compact set $E$ \emph{$\alpha$-removable (in the Lipschitz category)} if for any $u$ that is Lipschitz continuous in a neighbourhood $U$ of $E$, satisfies $\int_{\R^d}\frac{|u(x)|}{(1+|x|)^{d+\alpha}}dm_d(x)<\infty$, and such that $$\langle (-\Delta)^{\alpha/2}u, f\rangle=0 \text{ for all }f\in C^{\infty}_0(U)\text{ with }\supp(f)\cap E=\varnothing,$$
we in fact have that
$$\langle (-\Delta)^{\alpha/2}u, f\rangle=0 \text{ for all }f\in C^{\infty}_0(U).$$

The main goal of this section is to show that the $\alpha$-removable sets  coincide with the sets of $s$-Calder\'{o}n-Zygmund capacity zero, where $s=d-\alpha+1$.    Recall that if $T$ is compactly supported tempered distribution, then we may define $\langle T, f\rangle$ against any smooth function, not necessarily with compact support.  Also recall that $K(x) = \tfrac{x}{|x|^{s+1}}$ is the $s$-Riesz kernel.  For a compact set $E$, set \begin{equation}\begin{split}\nonumber\gamma_s(E)=\sup\bigl\{\langle T, 1\rangle: & \; T\text{ is a distribution satisfying } \supp(T)\subset E,\\& K*T\in L^{\infty}(m_d),\, \|K*T\|_{L^{\infty}(m_d)}\leq 1\bigl\}.
\end{split}\end{equation}
and
\begin{equation}\begin{split}\nonumber\gamma_{s,+}(E)=\sup\bigl\{\mu(\R^d): &\, \mu \text{ is a measure satisfying } \supp(\mu)\subset E,\\&\,K*\mu\in L^{\infty}(m_d),\, \|K*\mu\|_{L^{\infty}(m_d)}\leq 1\bigl\}.
\end{split}\end{equation}
It is clear that $\gamma_{s,+}(E)\leq \gamma_{s}(E)$.  Building upon prior results of Tolsa \cite{Tol3} and Volberg \cite{Vol}, Prat \cite{Pra} proved that for any $s\in (0,d)$ there is a constant $C>0$ such that $$ \gamma_{s}(E)\leq C\gamma_{s,+}(E) \text{ for every compact set }E\subset \R^d.$$

On the other hand, Theorem \ref{introthm} implies that the capacty $\gamma_{s,+}$ is equivalent to a certain capacity that arises in non-linear potential theory.   Set
\begin{equation}\begin{split}\nonumber\text{cap}_{s}(E) = \sup\Bigl\{\mu(\R^d): &\,\supp(\mu)\subset E\text{ and } \\&\sup_{x\in \R^d}\int_0^{\infty}\Bigl(\frac{\mu(B(x,r))}{r^s}\Bigl)^2\frac{dr}{r}\leq 1\Bigl\}.
\end{split}\end{equation}

Usually the set function $\text{cap}_s(E)$ is denoted by $\text{cap}_{\tfrac{2}{3}(d-s), \tfrac{3}{2}}(E)$, see \cite{AH}, largely because of the role it plays in approximation theory for the homoegeneous Sobolev space $\dot{H}^{\tfrac{2}{3}(d-s), \tfrac{3}{2}}(\R^d)$, see Chapters 10 and 11 of \cite{AH}.

We now state a Corollary of Theorem \ref{introthm}.

\begin{lem} If $s\in (d-1,d)$, then there is a constant $C>0$ such that
$$\frac{1}{C}\operatorname{cap}_{s}(E)\leq \gamma_{s,+}(E)\leq C\operatorname{cap}_{s}(E)
$$
for every compact set $E\subset \R^d$. \end{lem}

Given Theorem \ref{introthm}, the proof that follows is quite standard, and strings together a collection of known (but non-trivial) results.

\begin{proof}  Fix a compact set $E$.  The left hand inequality holds for all $s$, interger or not.  Indeed, suppose first that there is a non-zero measure $\mu$ supported on $E$ such that $$\sup_{x\in \R^d}\int_0^{\infty}\Bigl(\frac{\mu(B(x,r))}{r^s}\Bigl)^2\frac{dr}{r}\leq 1.$$
Then certainly $\mathcal{W}_{2}(\mu, Q)\leq \mu(Q)$ for every cube $Q$, and hence the $s$-Riesz transform is bounded in $L^2(\mu)$ (as indicated in the introduction, this direction of Theorem \ref{introthm} is well-known and holds for all $s$).  But then non-homogeneous Calder\'{o}n-Zygmund theory ensures that the $s$-Riesz transform is also of weak-type 1-1, see \cite{NTV}, with norm at most some constant $C>0$ depending on $d$ and $s$.  A standard (but mysterious) argument involving the Hahn-Banach theorem, see \cite{Chr} Theorem VII.23, yields that there is a function $h$, $0\leq h\leq 1$ such that
$$\int_{E}hd\mu\geq \frac{1}{2}\mu(E),
$$
while $\|K*(h\mu)\|_{L^{\infty}(m_d)}\leq 16C$.  It follows that $$\gamma_{+,s}(E)\geq \frac{1}{32C}\text{cap}_s(E).$$

For the right hand inequality, suppose that $\mu$ is a measure supported on $E$ for which  $\|K*\mu\|_{L^{\infty}(m_d)}\leq 1$.  Then it is an elementary exercise in the Fourier transform to show that this implies that there is a constant $C>0$ such that $\mu(B(x,r))\leq Cr^s$ for any $x\in \R^d, r>0$ (see for instance \cite{ENV}).   But now we may apply the $\TSO(1)$-theorem \cite{NTV2} to deduce that the $s$-Riesz transform of $\mu$ is bounded in $L^2(\mu)$ with norm at most some constant $C'>0$.  Applying Theorem \ref{introthm} we find a  constant $C''>0$ such that
$$\int_{\R^d}\int_0^{\infty}\Bigl(\frac{\mu(B(x,r))}{r^s}\Bigl)^2\frac{dr}{r}d\mu(x)\leq C''\mu(\R^d)=C''\mu(E).
$$
We now use the Chebyshev inequality to find a set $E'\subset E$ with $\mu(E')\geq \frac{1}{2}\mu(E)$ and
$$\int_0^{\infty}\Bigl(\frac{\mu(B(x,r))}{r^s}\Bigl)^2\frac{dr}{r}\leq 2C'' \text{ on }E'.
$$
Therefore, if we set $\mu'=\chi_{E'}\mu$, then $\mu'(E)\geq \tfrac{1}{2}\mu(E)$ and
$$\int_0^{\infty}\Bigl(\frac{\mu'(B(x,r))}{r^s}\Bigl)^2\frac{dr}{r}\leq 2C'' \text{ on }\supp(\mu').
$$
But now it is easy to see that
$$\int_0^{\infty}\Bigl(\frac{\mu'(B(x,r))}{r^s}\Bigl)^2\frac{dr}{r}\leq 3^{2s}2C'' \text{ on }\R^d,
$$
and so consequently the measure $\frac{\mu'}{3^{2s}2C''}$ is admissible for the definition of $\text{cap}_s(E)$.  We conclude that  $\text{cap}_s(E) \geq \frac{1}{3^{2s}4C''}\gamma_{s,+}(E).$
\end{proof}

As a consequence of these remarks, Theorem \ref{remove} of the introduction will be proved once we verify the following proposition:

\begin{prop}  Suppose that $\alpha\in (1, 2)$.  A compact set $E$ is $\alpha$-removable if and only if $\gamma_s(E)=0$, where $s=d-\alpha+1$.
\end{prop}

The fact that $\gamma_s(E)>0$ ensures that $E$ is non-removable is the easier assertion.  Indeed, in this case $E$ supports a non-trivial distribution $T$ with $K*T\in L^{\infty}(m_d)$.  But then $$u(x) = \frac{1}{|\,\cdot\,|^{d-\alpha}}*T$$ is a Lipschitz continuous function such that $(-\Delta)^{\alpha/2}u=0$ outside of $E$, but not in $\R^d$.

On the other hand suppose that $E$ is non-removable.   Then  there is a function $u$ that is Lipschitz continuous in some open neighbourhood $U$ of $E$ satisfying $\int_{\R^d}\frac{|u(x)|}{(1+|x|)^{d-\alpha}}dm_d(x)<\infty$ and  $$\langle (-\Delta)^{\alpha/2}u,\varphi\rangle=\langle T,\varphi\rangle, \text{ for every }\varphi\in C^{\infty}_0(U)$$ where $T$ is a non-zero distribution supported on $E$.  We may define a distribution $w\in \mathcal{S}'_{\alpha}(\R^d)$ by
$$\langle w, g\rangle = \langle T, \frac{1}{|\,\cdot\,|^{d-\alpha}}*g \rangle \text{ for }g\in \mathcal{S}_{\alpha}(\R^d).$$
(Merely notice that $\frac{1}{|\,\cdot\,|^{d-\alpha}}*g\in C^{\infty}(\R^d)$ as $g$ and all its derivatives lie in $L^1(\R^d)\cap L^{\infty}(\R^d)$.)

Notice that $w$ can be represented by the smooth function $x\mapsto \langle T, \tfrac{1}{|x-\cdot|^{d-\alpha}}\rangle$ in $\R^d\backslash E$.  For any $f\in \mathcal{S}(\R^d)$, the mapping property (\ref{relabelschwartz}) ensures that
$$\langle w,(-\Delta)^{\alpha/2} f \rangle = \langle T, \frac{1}{|\,\cdot\,|^{d-\alpha}}*(-\Delta)^{\alpha/2}f\rangle,
$$
but the right hand side is just $b_{\alpha}\langle T, f\rangle$ for some non-zero $b_{\alpha}\in \mathbb{C}$.  Thus there is a constant $B_{\alpha}\in \mathbb{C}$ such that, if we set
$v = B_{\alpha} w$, then $v$ satisfies $(-\Delta)^{\alpha/2}v=T.$

In particular notice that
$$\langle (u-v), (-\Delta)^{\alpha/2} f\rangle =0 \text{ for all }f\in C^{\infty}_0(U).
$$

Now pick some open neighbourhood $V$ of $E$ that is compactly contained in $U$.  Fix $\rho = \tfrac{1}{4}\min(\dist(E, V^c), \dist(E, U^c))$.  For an open set $W$, we denote $W_{\rho} = \{x\in W: \dist(x, W^c)>\rho\}$.

Notice that $v\in \Lip(\R^d\backslash V_{\rho})$, and $\lim_{|x|\rightarrow \infty}v=0$.  From this we infer the following properties of the distribution $F=u-v$:

$\bullet$  $\langle (-\Delta)^{\alpha/2}F, \varphi\rangle=0$ for every $\varphi\in C^{\infty}_0(U)$,

$\bullet$  $F\in \Lip(U\backslash V_{\rho})$,

$\bullet$  $F$ is locally integrable outside $V$ and moreover $$\int_{\R^d\backslash V}\frac{|F(y)|}{(1+|y|)^{d+\alpha}}dm_d(y)<\infty.$$

Take a function $\varphi\in C^{\infty}_0(B(0,1))$ with $\int_{\R^d}\varphi dm_d =1$.  For $\eps>0$, set $\varphi_{\eps} = \eps^{-d}\varphi(\tfrac{\cdot}{\eps})$.  For $\eps<\rho$,  consider the smooth function $F_{\eps} = \varphi_{\eps}*F$.  Since the fractional Laplacian is a linear operator, we have that $(-\Delta)^{\alpha/2}F_{\eps}=0$ in $U_{\eps}$, and moreover $(-\Delta)^{\alpha/2}(\nabla F_{\eps})=0$ in $U_{\eps}$.

 Certainly $\int_{\R^d}\frac{|\nabla F_{\eps}(x)|}{(1+|x|)^{d+\alpha}}dm_d(x)<\infty$, and so we may write, for any ball $B(x,r)\subset U_{\eps}$,
$$\nabla F_{\eps} = P^{\alpha}_{x,r}*\nabla F_{\eps},$$
where $P^{\alpha}_{x,r} = P_r^{\alpha}(x+\,\cdot\,)$ and $P_r^{\alpha}$ is the $\alpha$-Poisson kernel introduced in Appendix \ref{mpsec}.   For the derivation of this formula, see the discussion in Chapter 1.6 of \cite{Lan} leading to the formula (1.6.19).

Consider $M_{\eps} := \max_{\overline{V}}|\nabla u_{\eps}|$, so $M_{\eps}=|\nabla u_{\eps}(x_0)|$ for some $x_0\in \overline{V}$.  We wish to bound $M_{\eps}$ independently of $\eps$.  To this end, notice that since $B(x_0, \rho)\subset U_{\rho}\subset U_{\eps}$, we have that
$$M_{\eps}=|\nabla F_{\eps}(x_0)| = |P^{\alpha}_{x_0,\rho}*(\nabla F_{\eps})(x_0)|.
$$
Let us introduce a smooth function $\psi\in C^{\infty}(\R^d)$ such that $\psi\equiv 0$ in $\R^d\backslash U_{\rho}$, and $\psi\equiv 1$ inside $U_{2\rho}$ (an open set that contains $\overline{B(x_0, \rho)}$).  Then we split the convolution $P^{\alpha}_{x_0,\rho}*(\nabla F_{\eps})(x_0)$ into a local term $I=P^{\alpha}_{x_0,\rho}*(\psi\nabla F_{\eps})(x_0)$, and
a non-local term $II=P^{\alpha}_{x_0,\rho}*([1-\psi]\nabla F_{\eps})(x_0)$.

We first examine the local term.  Note that
since the Poisson kernel $P^{\alpha}_{x_0,\rho}$ has integral $1$, and is non-negative, there exists some $\lambda = \lambda(V, \rho)\in (0,1)$ such that $0\leq \int_{\overline{V}}P^{\alpha}_{x_0,\rho}(x_0-y)dm_d(y)\leq \lambda$,
and consequently $\bigl|\int_{\overline{V}}P^{\alpha}_{x_0,\rho}(x_0-y)\nabla F_{\eps}(y) dm_d(y) \bigl|\leq \lambda M_{\eps}.$

On the other hand, the function $F\in \Lip(U\backslash V_{\rho})$, and so $\|\nabla F_{\eps}\|_{L^{\infty}(U_{\rho}\backslash V)}\leq \|F\|_{\Lip(U\backslash V_{\rho})}$.  We therefore infer that $$I\leq \lambda M_{\eps} + \|F\|_{\Lip(U\backslash V_{\rho})}.$$

Regarding the non-local term, we may integrate by parts to deduce that
$$|II|\leq \int_{\R^d}|\nabla(P^{\alpha}_{x_0,\rho}(x_0-y)[1-\psi(y)])||F_{\eps}(y)|dm_d(y).
$$
As $\eps\to 0$, the right hand side converges to $\int_{\R^d}|\nabla(P^{\alpha}_{x_0,\rho}(x_0-y)[1-\psi(y)])||F(y)|dm_d(y)<\infty$.  Consequently, we infer that there is a constant $C>0$ (that may depend on $F$, $U$, $V$, and $E$), such that for all sufficiently small  $\eps>0$, $$M_{\eps}\leq \lambda M_{\eps}+C,$$
i.e., $\|F_{\eps}\|_{\Lip(V)}\leq \frac{C}{1-\lambda}$.  From the Arzela-Ascoli theorem we deduce that $F=u-v\in \Lip(V)$.

Since $v$ can be represented by a Lipschitz function outside of any neighbourhood of $E$,  we conclude that  $v$ is a Lipschitz continuous function in $\R^d$, and so its derivative, $CK*T$, lies in $L^{\infty}(m_d)$.  The only obstruction to concluding that $\gamma_{s}(E)>0$ is that we do not know that $\langle T,1\rangle \neq 0$.  However, since $T$ is non-zero, we can find some $\psi\in C^{\infty}_0(\R^d)$ so that
$$\langle \psi T, 1\rangle = \langle T, \psi\rangle \neq 0.
$$
But then by the localization property of distributions which have bounded convolution with the Riesz kernel, see Lemma 4 of \cite{Pra} or Lemma 3.1 of \cite{MPV}, we have that $ K*(\psi T)\in L^{\infty}(m_d)$.   Thus $\gamma_{s}(E)>0$.

 \end{document}